\documentclass[reqno,a4paper]{amsart}
\usepackage[english]{babel}

\usepackage{graphicx,mathrsfs,tikz,tikz-cd,latexsym,ifthen,amsmath,amsfonts,amssymb,amsthm,stmaryrd,fancyhdr,amscd,amsbsy,amstext,tensor, enumerate,enumitem,empheq,mathtools,mathpazo,verbatim,a4wide, epigraph} 
\mathtoolsset{showonlyrefs,showmanualtags}

\usepackage{tikz,graphicx,tikz-cd}
\usetikzlibrary{shapes, backgrounds} 
\usepackage[all,cmtip,knot]{xy}
\xyoption{arc}

\usepackage{hyperref}
\hypersetup{
	anchorcolor=black, 
	linktocpage=true,
	colorlinks=true,
	citecolor=blue,
	linkcolor=blue,
	urlcolor=black,
	pdfauthor={Andrea Appel, Francesco Sala, and Olivier Schiffmann},
	pdftitle={Continuum Kac--Moody algebras}, 
	breaklinks=true,
	plainpages=true
}

\usepackage{epigraph}

\usepackage[utf8x]{inputenc}

\usepackage[toc, page]{appendix}
\usepackage[mathscr]{euscript} 
\allowdisplaybreaks

\usepackage{color}
\definecolor{MyDarkBlue}{rgb}{0.15,0.25,0.45}

\newcommand{\calA}{{\mathcal A}}

\newcommand{\calJ}{{\mathcal J}}

\newcommand{\calQ}{{\mathcal Q}}

\newcommand{\K}{{\mathbb{K}}}
\newcommand{\R}{{\mathbb{R}}}

\newcommand{\Z}{{\mathbb{Z}}}
\newcommand{\Q}{{\mathbb{Q}}}

\newcommand{\scrR}{{\mathscr R}}

\newcommand{\sfA}{{\mathsf A}}

\newcommand{\sfS}{{\mathsf S}}
\newcommand{\sfT}{{\mathsf T}}
\newcommand{\sft}{{\mathsf t}}
\newcommand{\sfQ}{{\mathsf Q}}
\newcommand{\sfR}{{\mathsf R}}

\newcommand {\bfI}{\mathbf I}
\newcommand{\bfU}{\mathbf U}

\renewcommand {\c}{\mathfrak c}

\newcommand {\g}{\mathfrak{g}}
\newcommand {\h}{\mathfrak h}

\newcommand {\n}{\mathfrak n}

\renewcommand {\r}{\mathfrak r}

\renewcommand{\sl}{\mathfrak{sl}}
\newcommand{\heis}{\mathfrak{heis}}

\newcommand{\rank}{\mathsf{rk}}
\newcommand{\opspan}{\mathsf{span}}

\makeatletter
\newcommand{\colim@}[2]{%
	\vtop{\m@th\ialign{##\cr
			\hfil$\operator@font colim$\hfil\cr
			\noalign{\nointerlineskip\kern1.5\ex@}\cr
			\noalign{\nointerlineskip\kern-\ex@}\cr}}%
}
\newcommand{\colim}{%
	\mathop{\mathpalette\colim@{\rightarrowfill@\scriptscriptstyle}}\nmlimits@
}
\makeatother

\newcommand{\ldegp}{\mathsf{D}^{\mathsf{loc}}}
\newcommand{\degp}{\mathsf{D}}

\newcommand {\ol}{\overline}
\newcommand {\wt}{\widetilde}
\newcommand {\wh}{\widehat}

\newcommand{\scsop}{\scriptscriptstyle\operatorname} 

\newcommand {\ie}{{\textit{i.e.}}, }
\newcommand {\eg}{{\textit{e.g.}}, }

\newcommand{\oa}{a}
\newcommand{\ob}{b}
\newcommand{\oc}{c}

\newcommand{\serre}[1]{\mathsf{Serre}(#1)}
\newcommand{\serreadm}[1]{{\mathsf{Serre}}(#1)^{\scriptstyle\mathsf{adm}}}

\newcommand{\drc}[1]{\delta_{#1}}

\newcommand{\sfad}{\mathsf{ad}}
\newcommand{\ten}{\otimes}

\newcommand{\bsfld}{{\mathbf k}} 
\newcommand{\gcm}{\sfA} 
\newcommand{\vaa}{\calA} 

\newcommand{\sgp}{\sfS} 
\newcommand{\dsgp}{\sgp^{(2)}} 
\newcommand{\sgpp}{\oplus} 
\newcommand{\sgpm}{\ominus} 
\newcommand{\rsgp}{\sfS^{\mathsf{re}}}
\newcommand{\isgp}{\sfS^{\mathsf{im}}}
\newcommand{\csgp}{\mathbf{S}} 

\newcommand{\cf}[1]{{\mathbb 1}_{#1}} 
\newcommand{\fun}[2]{\mathfrak{fun}_{#2}(#1)} 
\newcommand{\intsf}[2]{\mathsf{Int}_{#2}(#1)} 
\newcommand{\abf}[2]{\langle #1, #2\rangle} 
\newcommand{\rbf}[2]{\left( #1, #2 \right)} 
\newcommand{\abfcf}[2]{\langle \cf{#1}, \cf{#2}\rangle} 
\newcommand{\rbfcf}[2]{\left( \cf{#1}, \cf{#2}\right)} 
\newcommand{\intnext}{\to} 
\newcommand{\rsub}{\vdash} 
\newcommand{\lsub}{\dashv} 
\newcommand{\intcap}{\pitchfork} 
\newcommand{\rtl}{\sfQ} 
\newcommand{\rts}{\sfR} 
\newcommand{\hcor}[1]{h_{#1}} 
\newcommand{\cow}[1]{\lambda^{\vee}_{#1}} 
\newcommand{\rls}[1]{\scrR} 
\newcommand{\mrls}[1]{\scrR_{\scsop{min}}} 
\newcommand{\crls}[1]{\ol{\scrR}} 

\newcommand{\sgprt}[1]{\varphi_{#1}}

\newcommand{\sfB}{\mathsf{B}}

\newcommand{\ia}{\alpha}
\newcommand{\ib}{\beta}
\newcommand{\ic}{\gamma}

\newcommand{\xg}[2]{x^{#1}_{#2}} 
\newcommand{\xe}[1]{\xg{\epsilon}{#1}} 
\newcommand{\xz}[1]{\zeta_{#1}} 
\newcommand{\xp}[1]{\xg{+}{#1}} 
\newcommand{\xm}[1]{\xg{-}{#1}} 
\newcommand{\xpm}[1]{\xg{\pm}{#1}} 
\newcommand{\xmp}[1]{\xg{\mp}{#1}} 

\newcommand{\KZQR}{\K=\Z,\Q,\R}
\newcommand{\KZQ}{\K=\Z,\Q}

\newcommand{\cq}[1]{\mathbf{Q}_{#1}}

\newcommand{\cqserre}[1]{\mathcal{S}_{#1}}

\newcommand{\cqg}[1]{\g_{#1}}
\newcommand{\cqwtg}[1]{\wt{\g}_{#1}}

\newcommand{\torus}{\mathbf{T^2}} 
\newcommand{\goldman}{\g_{\torus}}
\newcommand{\gsgp}{\mathsf{H}}
\newcommand{\gcsgp}{\mathbf{H}}

\newcommand{\triend}{\parbox{2mm}{\hfill} \hfill\text{\hspace{0.2mm}}\hfill$\triangle$}
\newcommand{\ocend}{\parbox{2mm}{\hfill} \hfill\text{\hspace{0.2mm}}\hfill$\oslash$}

\newtheorem{theorem}{Theorem}
\newtheorem{proposition}[theorem]{Proposition}
\newtheorem{lemma}[theorem]{Lemma}
\newtheorem{corollary}[theorem]{Corollary}

\newtheorem*{corollary*}{Corollary}
\newtheorem*{theorem*}{Theorem}
\newtheorem*{proposition*}{Proposition}
\newtheorem*{conjecture*}{Conjecture}
\newtheorem*{lemma*}{Lemma}

\numberwithin{equation}{section}
\numberwithin{theorem}{section}

\theoremstyle{remark}
\newtheorem{ex}[theorem]{Example}
\newenvironment{example}{\begin{ex}}{\triend\end{ex}}

\theoremstyle{remark}
\newtheorem{rem}[theorem]{Remark}
\newenvironment{remark}{\begin{rem}}{\hfill\triend\end{rem}}

\theoremstyle{remark}
\newtheorem{cl}[theorem]{Claim}

\theoremstyle{definition}
\newtheorem*{rem*}{Remark}
\newenvironment{remark*}{\begin{rem*}}{\triend\end{rem*}}

\theoremstyle{definition}
\newtheorem{defin}[theorem]{Definition}
\newenvironment{definition}{\begin{defin}}{\ocend\end{defin}}

\title[Continuum Kac--Moody algebras]
{Continuum Kac--Moody algebras}

\author[A.~Appel]{Andrea Appel}
\address[Andrea Appel]{Università di Parma, Dipartimento di Scienze Matematiche, Fisiche e Informatiche, 
	Italy}
\curraddr{}
\email{\href{mailto:andrea.appel@unipr.it}{andrea.appel@unipr.it}}

\author[F.~Sala]{Francesco Sala}
\address[Francesco Sala]{Università di Pisa, Dipartimento di Matematica, 
	Italy}
\address{Kavli IPMU (WPI), UTIAS, The University of Tokyo,
	Japan}
\curraddr{}
\email{\href{mailto:francesco.sala@unipi.it}{francesco.sala@unipi.it}}

\author[O.~Schiffmann]{Olivier Schiffmann}
\address[Olivier Schiffmann]{Laboratoire de Math\'ematiques, Universit\'e de Paris-Sud Paris-Saclay, 
	France}
\email{\href{mailto:olivier.schiffmann@math.u-psud.fr}{olivier.schiffmann@math.u-psud.fr}}

\thanks{The first--named author is partially supported by the ERC Grant 637618.
	The second--named author is partially supported by World Premier International 
	Research Center Initiative (WPI), MEXT, Japan, by JSPS KAKENHI Grant number JP17H06598 and 
	by JSPS KAKENHI Grant number JP18K13402.}


\begin{document}
	
	\begin{abstract}
		We introduce a new class of infinite--dimensional Lie algebras, which we
		refer to as continuum Kac--Moody algebras. Their construction is closely related
		to that of usual Kac--Moody algebras, but they feature a continuum root system with
		no simple roots. Their Cartan datum encodes the topology of a one--dimensional real 
		space and can be thought of as a generalization of a quiver, where vertices are replaced 
		by connected intervals. For these Lie algebras, we prove an analogue of the 
		Gabber--Kac--Serre theorem, providing a complete set of defining relations featuring 
		only quadratic Serre relations. Moreover, we provide an alternative realization as continuum
		colimits of symmetric Borcherds--Kac--Moody algebras with at most isotropic simple roots.
		The approach we follow deeply relies on the more general notion of a semigroup Lie algebra and
		its structural properties.
	\end{abstract}
	
	\maketitle\thispagestyle{empty}

	\tableofcontents
	\addtocontents{toc}{\protect\setcounter{tocdepth}{1}}
	
	
	\bigskip\section{Introduction}

	In the present paper, we introduce a new class of infinite--dimensional Lie algebras, 
	which we refer to as \textit{continuum Kac--Moody algebras}, associated to a topological 
	generalization of the notion of a quiver, where vertices are replaced by intervals in a real 
	one--dimensional topological space. These Lie algebras do not fall into the realm of 
	Kac--Moody algebras (nor of their several generalizations due to Borcherds \cite{borcherds-88},
	Saveliev--Vershik \cite{saveliev-vershik-90-2},  and Bozec \cite{bozec-16}). Rather, they encode the algebraic 
	structure of certain colimits of symmetric Borcherds--Kac--Moody algebras, corresponding, roughly, to families 
	of quivers with a number of vertices tending to infinity. 
	
	The simplest non--trivial examples of continuum Kac--Moody algebras are the Lie algebras 
	of the line and the circle, constructed in \cite{sala-schiffmann-17} together with 
	their quantizations, the \textit{quantum groups of the line and the circle}. The latter has 
	various geometric realizations via the theory of Hall algebras. Originally, it arises from the 
	Hall algebra of coherent sheaves on the \textit{infinite root stack} of a pointed curve. 
	In addition, T. Kuwagaki provided in \emph{loc. cit.} a \textit{mirror symmetry} 
	type construction of the same quantum group from the (derived infinitesimally wrapped) Fukaya category 
	of the cotangent bundle of the circle.
	In analogy with these constructions, we expect that continuum Kac--Moody algebras
	admit various geometric realizations, as classical limits of Hall algebras associated to coherent sheaves on weighted projective lines, \textit{coherent} persistence modules \footnote{See \cite{sala-schiffmann-19} for the notion of \textit{coherent} persistence 
		modules when $X$ is the line or the circle, and the references therein about the general theory of 
		persistence modules.}, and Fukaya categories of cotangent bundles.
	
	In the remaining part of the introduction, we provide some motivations and the description of the 
	continuum Kac--Moody algebras. These algebras are introduced as examples of a more general class 
	of Lie algebras, the \textit{semigroup} Lie algebras, which are naturally associated to (partial) semigroups. 
	Finally, we prove that the Goldman Lie algebra of the torus can be thought of as a remarkable example of
	such Lie algebras.
	
	\subsection*{The continuum Kac--Moody algebras of the line and the circle}
	In \cite{sala-schiffmann-17}, the second and third--named authors introduced the so-called \textit{circle quantum group} $\bfU_\upsilon\big(\sl(\Q/\Z)\big)$ and its classical limit
	$\sl(\Q/\Z)$. Their defining relations are somewhat cumbersome. Roughly, they are generated
	by infinitely--many $\sl(2)$--triples indexed by pairs of points (or equivalently \emph{intervals}
	\footnote{We call \textit{interval} the image in $\Q/\Z$ of an open--closed interval $(a, b] \subset \Q$ with $b-a \leqslant 1$.}) in the rational circle $\Q/\Z$ and commutation relations 
	depending upon their mutual position in $\Q/\Z$ (cf.~Definition~\ref{def:sl(K)}).
	The origin of such relations resides in the colimit realization, induced by the underlying 
	Hall algebra structure. More precisely, $\sl(\Q/\Z)$ can be thought of as a direct limit
	of standard affine Lie algebras $\wh{\sl}(n)$, $n\geqslant 2$, with system of morphisms 
	given by iterated applications on simple root vectors of the elementary map $\sl(2)\to\sl(3)$ 
	identified by the highest root in $\sl(3)$. 
	This construction endows $\sl(\Q/\Z)$ with several rather exotic and interesting features. 
	For example, its root system has no simple roots, since every simple root vector in $\wh{\sl}(n)$ 
	can be identified with a commutator in $\wh{\sl}(n+1)$,  and it is presented 
	by quadratic, apparently non--homogeneous, Serre relations.
	Several straightforward generalization are at hand. For instance, one may replace $\Q/\Z$ 
	with $\Q$, thus obtaining a Lie algebra isomorphic to a colimit of $\sl(n)$, $n\geqslant2$.
	Then, one may replace $\Q$ and $\Q/\Z$ with the real line $\R$ and the circle 
	$S^1\coloneqq\R/\Z$, respectively, so to obtain the \emph{continuum} analogues
	$\sl(\R)$ and $\sl(S^1)$. In Section~\ref{s:km-vkm}, we provide a much shorter and 
	concise presentation of these Lie algebras (cf. Corollary~\ref{cor:slR-presentation}),
	which allows to think of them as the simplest examples of \textit{continuum Kac--Moody algebras}.
	
	\subsection*{Semigroup Lie algebras}
	The original defining relations of the Lie algebras $\sl(\R)$ and $\sl(S^1)$ show some similarities, but also some striking differences
	with the usual relations appearing in the theory of Kac--Moody algebras. In particular, although 
	the role of vertices in the Dynkin diagram is seemingly played by open--closed intervals, the commutation rules between 
	the elements $e_{J}$ and $f_{J}$ above do not quite match the usual Kac--Moody relations. From a purely combinatorial
	point of view, these new commutation rules depend upon two simple operations on the set of intervals: the 
	\textit{concatenation} of two adjacent intervals and the \textit{truncation} of an interval into a smaller one. 
	These operations amount to a partial semigroup structure on the set of intervals.\\
	
	Motivated by this observation, we develop in Section~\ref{s:sgp-lie} a general theory of Lie algebras $\g(\csgp)$ associated to a triple $\csgp=(\sgp,\kappa,\xi_{\pm})$ 
	where $\sgp$ is a partial semigroup and $\kappa, \xi_{\pm}\colon \sgp\times\sgp\to\bsfld$ are functions satisfying some natural conditions. 
	The construction of $\g(\csgp)$ is close in spirit to that of Kac--Moody algebras.
	More precisely, we define $\wt{\g}(\csgp)$ as the Lie algebra generated by elements 
	$\xpm{\ia}$ and $\xz{\ia}$, with $\ia\in\sgp$, subject to the conditions
	\begin{align}
		\xz{\ia\sgpp\ib} =\,  \xz{\ia}+\xz{\ib}
	\end{align}
	whenever $\ia\sgpp\ib$ is defined, and
	\begin{align}
		[\xz{\ia},\xz{\ib}]=& 0\ , \\
		[\xz{\ia},\xpm{\ib}]=& \pm\kappa(\alpha,\beta)\xpm{\ib}\ ,\\
		[\xp{\ia},\xm{\ib}]=&\, \drc{\alpha,\beta}\xz{\ia}+
		\xi_+(\alpha,\beta)\xp{\ia\sgpm\ib}-\xi_-(\beta,\alpha)\xm{\ib\sgpm\ia}\ .
	\end{align}
	Then, we set
	$\g(\csgp)\coloneqq \wt{\g}(\csgp)/\r$, where $\r$ is the sum of all two--sided graded 
	ideals in $\wt{\g}(\csgp)$ having trivial intersection with the Cartan subalgebra generated 
	by the $\xz{\ia}$'s. 
	
	It is natural to ask whether $\g(\csgp)$ is graded over $\csgp$. This led to the
	study in Section~\ref{s:sgp-serre-relations} and \ref{s:sgp-serre} of a semigroup analogue of the 
	usual Serre relations of the form
	\begin{align}
		[\xpm{\ia}, \xpm{\ib}]=\mu_{\pm}(\alpha,\beta)\cdot \xpm{\ia\sgpp\ib}\ ,
	\end{align}
	for some suitable $\mu_{\pm}\colon\sgp\times\sgp\to\bsfld$. We identify a list of key properties,
	encoded by the notion of a {\em good Cartan semigroup}, which guarantee the existence of
	Serre relations for a large class of pairs $(\ia,\ib)$ (Theorem~\ref{thm:sym-serre-rel}). 
	
	\subsection*{Continuum quivers and continuum Kac--Moody algebras}
	We apply the structural results on semigroup Lie algebras to study a new class 
	of infinite--dimensional Lie algebras, which we refer to as continuum Kac--Moody 
	algebras. They are naturally associated to a topological datum, which we refer to 
	as continuum quiver in that it can be thought of as a topological generalization of 
	a quiver. Roughly, a continuum quiver is the datum of a one--dimensional real 
	CW complex $X$, which is locally modeled by \textit{smooth trees} glued with copies of 
	$S^1$, with a bilinear form on the space of locally constant functions 
	(cf.\ Definition~\ref{def:topological-quiver}). More precisely, continuum quivers 
	shall be thought of as good Cartan semigroups with a topological origin 
	and similar properties to that of positive roots for a Kac-Moody algebras.
	
	The notion of interval lifts easily from $\R$ to $X$, and the set $\intsf{X}{}$ 
	of \textit{intervals} in $X$ is therefore naturally endowed with two partially defined 
	operations, \ie the sum of intervals $\sgpp$\ , given by concatenation, and their 
	difference $\sgpm$, given by truncation. The set $\intsf{X}{}$ is naturally endowed 
	with a $\sgpp$--bilinear form.
	
	On the space of locally constant, left--continuous functions on $\R$ with limited 
	support, we consider a bilinear form given by:
	\begin{align}
		\abf{f}{g}\coloneqq \sum_{x} f_-(x)(g_-(x)-g_+(x)) \ ,
	\end{align}
	where $h_\pm(x)\coloneqq \lim_{t\to 0,\, t>0}\, h(x\pm t)$. 
	Note that the form $\abf{\cdot}{\cdot}$ is essentially defined by its values on 
	the characteristic functions of connected intervals. Therefore, the form 
	$\abf{f}{g}$ and its symmetrization $\rbf{f}{g}\coloneqq\abf{f}{g}+\abf{g}{f}$
	extend uniquely from $\R$ to $X$ by $\sgpp$--bilinearity. Moreover, we regard them 
	as forms on $\intsf{X}{}$ by setting $\rbf{\alpha}{\beta}\coloneqq \rbfcf{\ia}{\ib}$ and $\abf{\ia}{\ib}\coloneqq\abfcf{\ia}{\ib}$. 
	Finally, the continuum quiver is the good Cartan semigroup $\cq{X}=(\intsf{X}{},\kappa_X, \xi_X)$
	with $\kappa_X,\xi_X\colon\intsf{X}{}\times\intsf{X}{}\to\bsfld$ given by 
	$\kappa_X(\ia,\ib)\coloneqq \rbf{\ia}{\ib}$ and
	$\xi_X(\ia,\ib)\coloneqq(-1)^{\abf{\ia}{\ib}}\rbf{\ia}{\ib}$.
	The datum of $\cq{X}$ should be interpreted as a topological generalization of the 
	Borcherds--Cartan datum associated to a locally finite quiver with loops.
	
	Given a continuum quiver $\cq{X}$, the continuum Kac--Moody algebra
	$\cqg{X}\coloneqq\g(\cq{X})$ is by definition the semigroup Lie algebra associated to
	$\cq{X}$. Note that its Cartan subalgebra essentially coincides with the algebra of
	locally constant functions on $X$. 
	
	\subsection*{Complete presentation and colimit realization}
	Our main theorem is a continuum analogue of 
	the Gabber--Kac theorem \cite{gabber-kac-81}. More precisely, we show that the
	maximal ideal in $\cqwtg{X}$ is generated by the semigroup Serre relations
	described by $\cq{X}$. This leads to the following explicit description.
	
	\begin{theorem*}[{Theorem~\ref{thm:main}}]
		The continuum Kac--Moody algebra $\cqg{X}$ is generated by the elements $\xpm{\ia}$ and $\xz{\ia}$ with $\ia\in \intsf{X}{}$,  
		subject to the following defining relations: 
		\begin{enumerate}\itemsep0.2cm
			\item for any $\ia, \ib\in\intsf{X}{}$ such that $\ia\sgpp\ib$ is defined, 
			\begin{align}
				\xz{\ia\sgpp\ib}=\xz{\ia}+\xz{\ib}\ ;
			\end{align}
			\item \textbf{Diagonal action:}
			for any $\ia, \ib\in\intsf{X}{}$, 
			\begin{align}
				[\xz{\ia},\xz{\ib}] =0\ ,\quad 
				[\xz{\ia}, \xpm{\beta}] =\pm \rbf{\ia}{\ib}\cdot\, \xpm{\beta}\ ;
			\end{align}
			\item \textbf{Double relations:} for any $\ia, \ib\in\intsf{X}{}$, 
			\begin{align}
				[\xp{\ia},\xm{\ib}]=\drc{\ia,\ib}\, \xz{\ia}+(-1)^{\abf{\ia}{\ib}}\rbf{\ia}{\ib}\cdot
				\left(\xp{\ia\sgpm\ib}-\xm{\ib\sgpm\ia}\right)\ ;
			\end{align}
			\item \textbf{Serre relations:} if $(\ia, \ib)\in \cqserre{X}$, then
			\begin{align}
				\begin{aligned}
					[\xp{\ia}, \xp{\ib}]= & (-1)^{\abf{\ib}{\ia}}\xp{\ia\sgpp\ib} \ , \\
					[\xm{\ia}, \xm{\ib}]= & (-1)^{\abf{\ia}{\ib}}\xm{\ia\sgpp\ib}\ .
				\end{aligned}
			\end{align}	
		\end{enumerate}
	\end{theorem*}
	Roughly, $\cqserre{X}$ consists of unordered pairs $(\ia, \ib)\in\intsf{X}{}\times\intsf{X}{}$ such that 
	either 
	\begin{itemize}\itemsep0.2cm
		\item $\ia\sgpp\ib$ does not exist and $\ia\cap\ib=\emptyset$ or
		\item $\ia$ is contractible and, for subintervals $\ia'\subseteq\ia$ and $\ib'\subseteq \ib$ with $\rbfcf{\ib}{\ib'}\neq 0$ 
		whenever $\ib'\neq\ib$, $\ia'\sgpp\ib'$ is either undefined or non--homeomorphic to $S^1$.
	\end{itemize}
	
	An immediate consequence is the realization of $\cqg{X}$ as a continuous colimit of symmetric 
	Bor\-cherds--Kac--Moody algebras with at most isotropic simple roots (Theorem~\ref{thm:colim}).
	This is based on the following observation. The semigroup Serre relations do imply the usual
	Serre relations appearing in the case of quivers with at most one loop in every vertex, suggesting
	that $\cqg{X}$ can be locally described in terms of standard Borcherds--Kac--Moody algebras.
	
	Let $\calJ$ be a finite set of intervals $\ia\in\intsf{X}{}$, which roughly corresponds to a local 
	description of $X$ as a CW complex. More precisely, this means that every interval is either contractible or homeomorphic to $S^1$, and given two intervals $\ia,\ib\in\calJ$, $\ia\neq\ib$, one of the following mutually exclusive cases occurs:
	\begin{itemize}\itemsep0.2cm
		\item[(a)] $\ia\sgpp\ib$ exists;
		\item[(b)] $\ia\sgpp\ib$ does not exist and $\ia\cap\ib = \emptyset$;
		\item[(c)] $\ia\simeq S^1$ and $\ib\subset\ia$\ .
	\end{itemize}
	For any such $\calJ$, we get a symmetric matrix
	$\sfA_{\calJ}\coloneqq\big(\rbf{\ia}{\ib}\big)_{\ia,\ib\in \calJ}$ 
	and therefore a quiver $\calQ_{\calJ}$. Note that the diagonal entries of 
	$\sfA_{\calJ}$ are either $2$ or $0$, while the off--diagonal entries 
	are $0,-1$, or $-2$. In the table below, we give few examples.
	
	\begin{align}
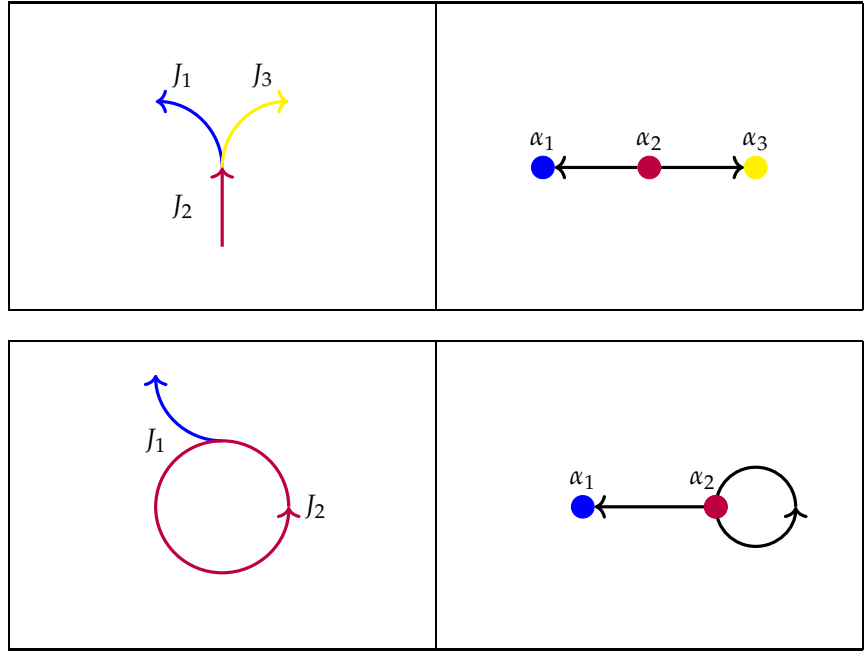
		\begin{array}{|c|c|}
			\hline
			\text{Configuration of intervals} & \text{Borcherds--Cartan diagram}\\
			\hline &\\
			\begin{tikzpicture}[scale=.35]
				\begin{scope}[on background layer]
					\draw [white] (0,0) rectangle (15,10);
				\end{scope}
				\draw [->, very thick, blue] (3,5) -- (6,5); 
				\draw [->, very thick, purple] (6,5) -- (9,5); 
				\draw [->, very thick, yellow] (9,5) -- (12,5);
				\node at (4.5, 6) {$\ia_1$};
				\node at (7.5, 6) {$\ia_2$};
				\node at (10.5, 6) {$\ia_3$};
			\end{tikzpicture}
			&
			\begin{tikzpicture}[scale=.35]
				\begin{scope}[on background layer]
					\draw [white] (0,0) rectangle (15,10);
				\end{scope}
				\node (V1) at (3.5,5)      [circle,draw=blue,fill=blue, inner sep=3pt]          {};
				\node (V2) at (7.5,5)      [circle,draw=purple,fill=purple, inner sep=3pt]    {};
				\node (V3) at (11.5,5) [circle,draw=yellow,fill=yellow, inner sep=3pt]   {};
				\draw [->, very thick] (V1) -- (V2);  
				\draw [<-, very thick] (V3) -- (V2);
				\node at (3.5, 6)    {$\alpha_1$};
				\node at (7.5, 6)    {$\alpha_2$};
				\node at (11.5, 6)  {$\alpha_3$};
			\end{tikzpicture}
			\\
			\hline
		\end{array}
	\end{align}
	\vspace{-0.5cm}
	\begin{align}
		\begin{array}{|c|c|}
			\hline &\\
			\begin{tikzpicture}[scale=.35]
				\begin{scope}[on background layer]
					\draw [white] (0,0) rectangle (15,10);
				\end{scope}
				\draw [->, very thick, purple] (7.5,2) -- (7.5,5); 
				\draw [->, very thick, blue] (7.5,5) arc (0:90:2.5); 
				\draw [->, very thick, yellow] (7.5,5) arc (180:90:2.5);
				\node at (6, 3.5) {$\ia_2$};
				\node at (6, 8.5) {$\ia_1$};
				\node at (9, 8.5) {$\ia_3$};
			\end{tikzpicture}
			&
			\begin{tikzpicture}[scale=.35]
				\begin{scope}[on background layer]
					\draw [white] (0,0) rectangle (15,10);
				\end{scope}
				\node (V1) at (3.5,5)      [circle,draw=blue,fill=blue, inner sep=3pt]          {};
				\node (V2) at (7.5,5)      [circle,draw=purple,fill=purple, inner sep=3pt]    {};
				\node (V3) at (11.5,5) [circle,draw=yellow,fill=yellow, inner sep=3pt]   {};
				\draw [->, very thick] (V2) -- (V1);  
				\draw [<-, very thick] (V3) -- (V2);
				\node at (3.5, 6)    {$\alpha_1$};
				\node at (7.5, 6)    {$\alpha_2$};
				\node at (11.5, 6)  {$\alpha_3$};
			\end{tikzpicture}
			\\
			\hline &\\
			\begin{tikzpicture}[scale=.35]
				\begin{scope}[on background layer]
					\draw [white] (0,0) rectangle (15,10);
				\end{scope}
				\draw [<-, very thick, blue] (5,10) arc (180:270:2.5);
				\draw [->, very thick, purple] (10,5) arc (0:360:2.5);
				\node at (5, 7.5)    {$\ia_1$};
				\node at (11, 5)    {$\ia_2$};
			\end{tikzpicture}
			&
			\begin{tikzpicture}[scale=.35]
				\begin{scope}[on background layer]
					\draw [white] (0,0) rectangle (15,10);
				\end{scope}
				\node (V1) at (5, 5)  [circle,draw=blue,fill=blue, inner sep=3pt]    {};
				\node (V3) at (10, 5)  [circle,draw=purple,fill=purple, inner sep=3pt]    {};
				\draw [->, very thick] (V3) -- (V1);
				\draw [->, very thick] (13,5) arc (0:360:1.5);  
				\node at (10, 5)  [circle,draw=purple,fill=purple, inner sep=3pt]    {};
				\node at (5, 6)    {$\alpha_1$};
				\node at (9.5, 6)    {$\alpha_2$};
			\end{tikzpicture}
			\\
			\hline
		\end{array}\\
	\end{align}
	Note, in particular, that any contractible elementary interval corresponds to a vertex of 
	$\calQ_{\calJ}$ without loops, while any interval homeomorphic to $S^1$, corresponds 
	to a vertex having exactly one loop. This correspondence between intervals and quivers
	describes the relation between $\cqg{X}$ and standard Borcherds--Kac--Moody algebras.
	Indeed, we prove that the Borcherds--Kac--Moody algebra $\g_{\calQ_{\calJ}}$ is isomorphic
	to the subalgebra $\g(\calJ)\subset\cqg{X}$ generated by the elements $\xpm{\ia}$ and 
	$\xz{\ia}$ with $\ia\in\calJ$.
	
	Moreover, we prove that such local description gives rise to a global description of 
	$\cqg{X}$ as a colimit of Borcherds--Kac--Moody algebras. Indeed, let $\mathsf{Sh}(X)$
	be the set of all quivers which arise from a collection of intervals in $X$. Then,
	we prove that there is a direct system of embeddings 
	$\varphi_{\calQ,\calQ'}\colon \g_{\calQ}\to\g_{\calQ'}$, 
	indexed by pairs of quivers in $\mathsf{Sh}(X)$, which leads to the following:
	
	\begin{theorem*}[cf.\ Theorem~\ref{thm:colim}]
		There is a canonical isomorphism $\displaystyle\cqg{X}\simeq\colim_{\calQ\in\mathsf{Sh}(X)}\, \g_{\calQ}$.
	\end{theorem*}
	
	\begin{remark*}
		In \cite{appel-sala-19}, the first and second--named authors prove that 
		$\cqg{X}$ is further endowed with a non--degenerate invariant inner product 
		and a standard structure of quasi--triangular (topological) Lie bialgebra graded 
		over the semigroup $\cq{X}$. The methods used in the proof extend to the case 
		of a Hopf algebra and allow to construct \textit{algebraically} an explicit 
		quantization $\bfU_\upsilon (\cqg{X})$ of $\cqg{X}$, which is then called \textit{the continuum quantum group of $X$}.
	\end{remark*}

	\subsection*{Goldman Lie algebra of the torus and semigroup Lie algebras}
	
	The main motivation behind the theory of semigroup Lie algebras was to provide 
	a solid ground and an abstract approach to the theory of continuum Kac--Moody 
	algebras. The resulting theory is on the other hand extremely rich and features
	remarkable examples, such as the \textit{Goldman Lie algebra of the torus}.
	
	Recall that in \cite{morton-samuelson-17} the Goldman Lie algebra of the torus $\goldman$ is proved to be isomorphic to the Lie algebra generated by the elements $P_{\mathbf v}$ with $\mathbf v\coloneqq (a, b)\in \Z^2\smallsetminus\{(0,0)\}$, modulo the relation
	\begin{align}
		[P_{\mathbf v}, P_{\mathbf w}]=\det [\mathbf v\, \mathbf w] P_{\mathbf v + \mathbf w}\ ,
	\end{align}
	where $[\mathbf v, \mathbf w]$ is the order two matrix with columns 
	$\mathbf v$ and $\mathbf w$. 
	
	This presentation suggests an obvious interpretation of $\goldman$
	as a semigroup Lie algebra. The integral half plane 
	$\gsgp\coloneqq \{ (a, b)\in \Z^2\ \vert\ a\geqslant 1 \ \text{ or } \ a=0, b\geqslant 1\}$
	is naturally endowed with a structure of (total) semigroup with respect to the usual sum 
	of vectors and \emph{truncated difference} $\sgpm$ given by
	\begin{align}\label{eq:difference-goldman}
		(a,b) \sgpm (a', b') \coloneqq \begin{cases}
			(a-a', b-b') & \text{if } (a-a', b-b') \in \gsgp\ , \\[3pt]
			0 & \text{otherwise}\ .
		\end{cases} 
	\end{align}
	Set $\kappa(\mathbf{v},\mathbf{w})=0$ and 
	$\xi_{+}(\mathbf v, \mathbf w)\coloneqq  \det [\mathbf w\, \mathbf v]\eqqcolon 
	\xi_{-}(\mathbf v, \mathbf w)$ for any $\mathbf x, \mathbf y\in \gsgp$.
	We set $\gcsgp=(\gsgp,0,\xi)$ and consider the semigroup Lie algebra $\g(\gcsgp)$.
	Note that $\Z^2\setminus\{(0,0)\}=\gsgp\sqcup(-\gsgp)$ and, since $\kappa=0$, the 
	Cartan subalgebra $\mathfrak{z}\subset\g(\gcsgp)$ is central and we set
	$\g_{\gcsgp}\coloneqq\g(\gcsgp)/\mathfrak{z}$. 
	Moreover, by Proposition~\ref{pr:sgp-serre-example}, the Serre relations hold in $\g(\gcsgp)$
	and one has $[\xpm{\mathbf v}, \xpm{\mathbf w}]=\det [\mathbf v\, \mathbf w] \xpm{\mathbf v + \mathbf w}\, $, for any $\mathbf v, \mathbf w\in \gsgp$. This yields to the following:
	\begin{proposition*}
		There is a surjective homomorphism of Lie algebras
		$\psi: \goldman\to\g_{\gcsgp}$ given by the assigment
		\begin{align}
			P_{\mathbf v}\longmapsto \begin{cases}
				\xp{\mathbf v} & \text{if } a\geq 1 \ \text{or} \ a=0\ , b\geq 1\ , \\[3pt]
				\xm{-\mathbf v} & \text{otherwise}\ ,
			\end{cases}
			\qquad\mbox{where}\qquad \mathbf v\coloneqq (a,b)\in \Z^2\smallsetminus\{(0,0)\ .
		\end{align}
	\end{proposition*}
	
	We expect $\psi$ to be an isomorphism.
	
	\subsection*{Highest weight theory and Fock spaces}
	
	We conclude by outlining a further direction of 
	research, currently under investigation. The usual description of the highest weight theory of 
	Borcherds--Kac--Moody algebras does not extend immediately to 
	continuum Kac--Moody algebras, mainly due to the lack of simple roots and thus of
	fundamental weights. Nevertheless, we expect the existence of a continuum analog 
	of the theory of highest weight representations, Weyl groups, character formulas, both 
	in the classical and quantum setups. 
	
	A first example is given in \cite{sala-schiffmann-19}, where the second--named and third--named 
	authors define the Fock space for $\bfU_\upsilon (\sl(\R))$, considering 
	a continuum analogue of the usual combinatorial construction of the Fock space 
	of $\bfU_\upsilon (\sl(\infty))$ in terms of Maya diagrams. 
	In addition, the quantum group $\bfU_\upsilon (\sl(S^1))$ act on such a Fock space, in a way similar 
	to the \textit{folding procedure} of Hayashi--Misra--Miwa. It would be interesting to extend this construction to the case of an arbitrary continuum quiver $X$, producing a wide class of representations for the continuum quantum group $\bfU_\upsilon (\cqg{X})$, and 
	therefore for the continuum Kac--Moody algebra $\cqg{X}$.

	\subsection*{Outline}
	
	In Section~\ref{s:km-vkm} we give a brief review of Kac--Moody algebras, 
	their generalizations introduced by Saveliev and Vershik, and the Lie algebra 
	$\sl(\R)$ as a remarkable counterexample. A unifying approch is developed 
	in Section~\ref{s:sgp-lie} through the notion of \textit{semigroup Lie algebra}, whose 
	weight lattice is controlled by a partial semigroup. In Section~\ref{s:sgp-serre-relations} 
	we study a semigroup analogue of the classical Serre relations. We determine 
	necessary and sufficient conditions for such relations to hold and we prove them
	in several special cases. In Section~\ref{s:sgp-serre}, we identify a list of key properties,
	encoded by the notion of a good Cartan semigroup, which guarantee the existence of
	Serre relations (Theorem~\ref{thm:sym-serre-rel}). Finally, in Section~\ref{s:topological-quiver}, 
	we apply the abstract framework developed in the previous sections to describe
	a new class of infinite--dimensional Lie algebras associated with a one--dimensional
	real space. We introduce the notion of continuum quivers, whose defining data are 
	interpreted in terms of good  Cartan semigroup. We define the continuum Kac--Moody 
	algebra of a continuum quiver as a semigroup Lie algebra and prove a continuum analogue
	of the Gabber--Kac theorem (Theorem~\ref{thm:main}). Moreover, we prove that every continuum Kac--Moody algebra is isomorphic to an uncountable colimit of symmetric Borcherds--Kac--Moody algebras with at most isotropic simple roots (Theorem~\ref{thm:colim}). 
	Finally, in Appendix~\ref{app:sgp} we include a brief review of the basic definitions and 
	properties of the partial semigroups we are mainly interested in, while in Appendix~\ref{app:relationssl}, \ref{app:sgp-serre-general}, \ref{app:sym-serre-rel}, \ref{app:topologicalquiver} we report in details some of the most technical proofs.
	
	\subsection*{Acknowledgements}
	
	We are grateful to Antonio Lerario and Tatsuki Kuwagaki for many enlightening 
	conversations about the material in Section~\ref{s:topological-quiver} and to Fabio Gavarini for
	his comments and interest in the present work. This work was initiated while the 
	first--named author was visiting Kavli IPMU and completed while the second--named author 
	was visiting the University of Edinburgh. We are grateful to both institutions for their hospitality 
	and wonderful working conditions.
	
	\bigskip\section{Kac--Moody algebras and the Lie algebra of the real line}\label{s:km-vkm}
	
	In this section, we recall the definitions of Kac--Moody 
	algebras and their generalization introduced by Saveliev and Vershik. 
	We then introduce a distinguished Lie algebra $\sl(\R)$, which first 
	appeared in \cite{sala-schiffmann-17} and whose presentation is
	controlled by the topology of the real line. We prove that $\sl(\R)$
	admits a simpler presentation \emph{\`a la Kac--Moody}, depending 
	upon basic operations on open--closed connected intervals. Henceforth, 
	we fix a base field $\bsfld$ of characteristic zero.
	
	\subsection{Kac--Moody algebras}\label{ss:km}
	
	We recall the definition from \cite[Chapter~1]{kac-90}. Fix a finite set $\bfI$ 
	and a matrix $\gcm =(a_{ij})_{i,j\in\bfI}$ with entries in $\bsfld$. Recall that a realization $\rls{}\coloneqq (\h,\Pi,\Pi^{\vee})$ of $\gcm$ is the datum of 
	a finite dimensional $\bsfld$--vector space $\h$, and linearly independent vectors 
	$\Pi\coloneqq \{\alpha_i\}_{i\in\bfI}\subset\h^\ast$, $\Pi^{\vee}\coloneqq \{h_i\}_{i\in\bfI} 
	\subset\h$ such that $\alpha_i(h_j)= a_{ji}$. Note that these conditions imply $\dim\h\geqslant 2\vert\bfI\vert-\rank({\gcm})$. Moreover,
	up to a (non--unique) isomorphism, there is a unique realization of minimal dimension 
	$2\vert\bfI\vert-\rank({\gcm})$. 
	
	Let $\wt{\g}(\rls{\gcm})$ be the Lie 
	algebra generated by $\h$, $\{e_i, f_i\}_{i\in\bfI}$ with relations $[\h,\h]=0$ and 
	\begin{align}
		[h,e_i]=\alpha_i(h)\, e_i\ ,
		\quad
		[h,f_i]=-\alpha_i(h)\, f_i\ ,
		\quad
		[e_i,f_j]=\drc{ij}\, h_i\ .\label{eq:km-rel}
	\end{align}
	for any $h\in\h$ and $i,j\in\bfI$. Set
	\begin{align}
		\rtl_+\coloneqq\bigoplus_{i\in\bfI}\Z_{\geqslant0}\, {\alpha}_i\subseteq{\h}^\ast\ ,
	\end{align}
	$\rtl\coloneqq\rtl_+\oplus(-\rtl_+)$, and denote by $\wt{\n}_{+}$ (resp. $\wt{\n}_-$) the 
	subalgebra generated by $\{e_i\}_{i\in\bfI}$ (resp. $\{f_i\}_{i\in\bfI}$). Then, as vector spaces, 
	$\wt{\g}(\rls{\gcm})=\wt{\n}_+\oplus\h\oplus\wt{\n}_-$ with root space decomposition
	\begin{align}
		\wt{\n}_{\pm}=\bigoplus_{\genfrac{}{}{0pt}{}{\alpha\in\rtl_+}{\alpha\neq0}}\, \wt{\g}_{\pm\alpha}
	\end{align}
	where $\wt{\g}_{\pm\alpha}=\{X\in\wt{\g}(\rls{\gcm})\;\vert\;\forall h\in\h, [h,X]=\pm\alpha(h)X\}$. 
	Note also that $\wt{\g}_0=\h$ and $\dim\wt{\g}_{\pm\alpha}<\infty$.\\
	
	The \textit{Kac--Moody algebra} associated to the realization $\rls{\gcm}$ is the Lie 
	algebra $\g(\rls{\gcm})\coloneqq\wt{\g}(\rls{\gcm})/\r$, where $\r$ is the sum of all 
	two--sided graded ideals in $\wt{\g}(\rls{\gcm})$ having trivial intersection with $\h$. 
	In particular, as ideals, $\r=\r_+\oplus\r_-$, where $\r_{\pm}\coloneqq\r\cap\wt{\n}_{\pm}$. 
	\footnote{The terminology differs slightly from the one given in \cite{kac-90} where 
		$\g(\rls{\gcm})$ is called a Kac--Moody algebra if $\gcm$ is a generalised Cartan 
		matrix (cf.\ Remark~\ref{rem:r=s}) and $\rls{\gcm}$ is the minimal realization. Note also 
		that in \cite[Theorem~1.2]{kac-90} $\r$ is set to be the sum of all two--sided ideals, not 
		necessarily graded. However, since the functionals $\alpha_i$ are linearly independent in 
		$\h^\ast$ by construction, $\r$ is automatically graded and satisfies $\r=\r_+\oplus\r_-$ 
		(cf.\ \cite[Proposition~1.5]{kac-90}).} \\
	
	\begin{remark}\label{rem:r=s}
		If $\gcm $ is a \textit{generalised Cartan matrix} (\ie $a_{ii}=2$, $a_{ij}\in\Z_{\leqslant 0}$, 
		$i\neq j$, and $a_{ij}=0$ implies $a_{ji}=0$), then $\r$ contains the ideal $\mathfrak{s}$ generated 
		by the Serre relations
		\begin{align}\label{eq:SerrerelgCm}
			\sfad(e_i)^{1-a_{ij}}(e_j)=0=\sfad(f_i)^{1-a_{ij}}(f_j)\qquad i\neq j
		\end{align}
		and, if $\gcm$ is symmetrizable, is generated by $\mathfrak{s}$ \cite{gabber-kac-81}. 
		A similar property holds, more generally, for any symmetrizable $\gcm$ such that $a_{ij}\in\Z_{\leqslant0}$, 
		$i\neq j$, and $2a_{ij}/a_{ii}\in\Z$ whenever $a_{ii}>0$. In this case, $\g(\gcm)$ is 
		called a Borcherds--Kac--Moody algebra and the corresponding maximal ideal is generated by
		the Serre relations
		\begin{align}\label{eq:Serrerel-BKM-1}
			\sfad(e_i)^{1-\frac{2}{a_{ii}}a_{ij}}(e_j)=0=\sfad(f_i)^{1-\frac{2}{a_{ii}}a_{ij}}(f_j)
			\qquad \text{ if }\quad a_{ii}>0
		\end{align}
		and
		\begin{align}\label{eq:Serrerel-BKM-2}
			[e_i,e_j]=0=[f_i,f_j]\qquad 
			\text{ if }\quad a_{ii}\leqslant 0\quad \text{ and } \quad a_{ij}=0
		\end{align}
		for $i\neq j$ (cf.\ \cite[Corollary 2.6]{borcherds-88}).
	\end{remark}
	
	Since $\r=\r_+\oplus\r_-$, the Lie algebra $\g(\rls{\gcm})$ has an induced triangular decomposition 
	$\g(\rls{\gcm})=\n_-\oplus\h\oplus\n_+$ (as vector spaces), where
	\begin{align}
		\n_\pm\coloneqq \bigoplus_{\alpha\in\rtl_+\setminus\{0\}}\g_{\pm\alpha}\ , \quad 
		\g_{\alpha}\coloneqq \{X\in{\g(\rls{\gcm})}\;\vert \; \forall\, h\in{\h},\, [h,X]=\alpha(h)\, X\}\ .  
	\end{align}
	Note that $\dim\g_{\alpha}<\infty$. The set of \emph{positive roots} is $\rts_+\coloneqq 
	\{\alpha\in\rtl_+\setminus\{0\}\;\vert \; {\g}_{\alpha}\neq0\}$.
	
	\begin{remark}\label{sss:der-km}
		
		The derived subalgebra $\g(\rls{\gcm})'\coloneqq[\g(\rls{\gcm}),\g(\rls{\gcm})]$ is
		generated by $\{e_,f_i,h_i\}_{i\in\bfI}$ 
		and admits a presentation similar to that of $\g(\rls{\gcm})$. Namely, let $\wt{\g}'$ be the 
		Lie algebra generated by $\{h_i,e_i,f_i\}_{i\in\bfI}$ with relations
		\begin{align}
			[h_i,h_j]=0\ ,
			\quad
			[h_j,e_i]=\alpha_i(h_j)\, e_i\ ,
			\quad
			[h_j,f_i]=-\alpha_i(h_j)\, f_i\ ,
			\quad
			[e_i,f_j]=\drc{ij}\, h_i\ .\label{eq:der-km}
		\end{align}
		Then, $\wt{\g}'$ has a $\rtl$--gradation defined by $\deg(e_i)=\alpha_i$, $\deg(f_i)=-\alpha_i$, 
		$\deg(h_i)=0$, and $\wt{\g}'_0=\h'$, where the latter is the $|\bfI|$--dimensional span of 
		$\{h_i\}_{i\in\bfI}$. The quotient of $\wt{\g}'$ by the sum of all two--sided graded ideals with 
		trivial intersection with $\h'$ is easily seen to be canonically isomorphic to $\g(\rls{\gcm})'$.  
		
		Recall 
		that a direct sum of vector spaces $L=L_{-1}\oplus L_0\oplus L_{+1}$ is a \textit{local Lie algebra} 
		if there are bilinear maps $L_i\times L_j\to L_{i+j}$ for 
		$\vert i\vert, \vert j\vert, \vert i+j\vert \leqslant 1$, such that antisymmetry and Jacobi identity 
		hold whenever they make sense. We write $\pm$ instead of $\pm1$. It follows that $\wt{\g}(\rls{\gcm})'\coloneqq[\wt{\g}(\rls{\gcm}),\wt{\g}(\rls{\gcm})]$ 
		is freely generated by the local Lie algebra $L=L_{-}\oplus L_0\oplus L_{+}$ where 
		\begin{align}
			L_{-}\coloneqq \bigoplus_i\, \bsfld\cdot f_i\ , \quad L_0\coloneqq \bigoplus_i\,\bsfld\cdot h_i\ , 
			\quad L_{+}\coloneqq\bigoplus_i\,\bsfld\cdot e_i \ , 
		\end{align}
		with bracket defined by equation \eqref{eq:der-km}.
	\end{remark}
	
	\begin{remark}\label{ss:ext-KM}
		It is sometimes convenient to consider Kac--Moody algebras associated to a non--minimal
		realization (cf.\ \cite{feigin-zelevinsky-85, maulik-okounkov-12, appel-toledano-19b}).
		Let $\crls{\gcm}=(\ol{\h},\ol{\Pi},\ol{\Pi}^\vee)$ be the realization given by 
		$\ol{\h}\cong\bsfld^{2|\bfI |}$ with basis $\{\hcor{i}\}_{i\in\bfI}\cup\{\cow{i}\}_{i\in\bfI}$, 
		$\ol{\Pi}^\vee=\{\hcor{i}\}_{i\in\bfI}$ and $\ol{\Pi}=\{\alpha_i\}_{i\in\bfI}\subset\ol{\h}^*$, 
		where $\alpha_i$ is defined by
		\begin{align}
			\alpha_i(\hcor{j})=a_{ji}
			\qquad\mbox{and}\qquad
			\alpha_i(\cow{j})=\delta_{ij}
		\end{align}
		We refer to $\crls{\gcm}$ as the {\it canonical} realization of
		$\gcm$, and denote by $\Lambda^\vee\subset\ol{\h}$ the $|\bfI|$--dimensional
		subspace spanned by $\{\cow{i}\}_{i\in\bfI}$. 
		Let $\mrls{\gcm}$ be a minimal realization of $\gcm$. It is easy to check that the Kac--Moody 
		algebra $\g(\crls{\gcm})$ is a central extension of $\g(\mrls{\gcm})$, \ie $\g(\crls{\gcm})\simeq
		\g(\mrls{\gcm})\oplus\c$, with $\dim\c=\rank(\gcm)$.
	\end{remark}
	
	\subsection{Saveliev--Vershik algebras}\label{ss:cont-lie}
	
	In \cite{saveliev-vershik-91}, Saveliev and Vershik introduced the notion of \textit{continuum 
		Lie algebras}, providing a generalization of Kac--Moody algebras and covering a wide 
	spectrum of examples including Lie algebras arising, for example, from ergodic transformations, 
	crossed products, or infinitesimal area--preserving diffeomorphisms \cite{saveliev-vershik-90-2, 
		saveliev-vershik-92, vershik-92, vershik-02}. 
	Their generalization is entirely different from the one we shall introduce in Section~\ref{s:topological-quiver}. We briefly recall their construction and, in order to avoid confusion, 
	we shall refer to it as  \emph{Saveliev--Vershik algebras}.\\
	
	Let $(\vaa,\cdot)$ be an arbitrary associative $\bsfld$--algebra (possibly infinite--dimensional, 
	non--unital, and non--commutative) endowed with three bilinear mappings $\kappa_{\pm},\kappa_0\colon \vaa\times \vaa\to \vaa$, and $L$ 
	a vector space of the form $L_-\oplus L_0\oplus L_+$, where $L_{\epsilon}\simeq \vaa$, $\epsilon=\pm,0$.
	We denote $X_{\epsilon}(\phi)$ the element in $L_{\epsilon}$, $\epsilon=\pm,0$, corresponding to 
	$\phi\in \vaa$, and $[\cdot,\cdot]\colon L\ten L\to L$ the bilinear map given by 
	\begin{align}
		[X_0(\phi),X_0(\psi)]=& X_0([\phi,\psi])\label{eq:ver-1}\ , \\
		[X_0(\phi),X_{\pm}(\psi)]=& X_{\pm}(\kappa_{\pm}(\phi,\psi))\label{eq:ver-2}\ , \\
		[X_+(\phi),X_{-}(\psi)]=& X_{0}(\kappa_{0}(\phi,\psi))\label{eq:ver-3}\ ,
	\end{align}
	where $\phi,\psi\in\vaa$, $[\phi,\psi]\coloneqq \phi\cdot\psi-\psi\cdot\phi$.
	The following is straightforward.
	
	\begin{lemma}[\cite{saveliev-vershik-91}]\label{lem:Jacobi-1}
		The vector space $L$ is a local Lie algebra with bracket $[\cdot,\cdot]$ if and only if 
		the maps $\kappa_{\epsilon}$ satisfy the following relations:
		\begin{align}
			\kappa_{\pm}([\phi,\psi],\chi)=& 
			\kappa_{\pm}(\phi,\kappa_{\pm}(\psi,\chi))-\kappa_{\pm}(\psi,\kappa_{\pm}(\phi,\chi))
			\label{eq:ver-datum-1}\ ,\\
			[\phi, \kappa_{0}(\psi,\chi)]=& 
			\kappa_{0}(\kappa_+(\phi,\psi),\chi)+\kappa_{0}(\psi,\kappa_{-}(\phi,\chi)) \ ,
			\label{eq:ver-datum-2}
		\end{align}
		where $\phi,\psi,\chi\in \vaa$.
	\end{lemma}
	
	Let $(\vaa,\kappa_{\epsilon})$ be an associative algebra endowed with three bilinear maps
	$\kappa_{\epsilon}\colon \vaa\ten\vaa\to\vaa$, $\epsilon=\pm,0$, satisfying \eqref{eq:ver-datum-1}, 
	\eqref{eq:ver-datum-2}. We denote by $\wt{\g}(\vaa,\kappa_{\epsilon})$ the Lie algebra freely generated by $L$.
	Note that $\wt{\g}(\vaa,\kappa_{\epsilon})$ is $\Z$--graded, \ie
	\begin{align}
		\wt{\g}(\vaa,\kappa_{\epsilon})=\bigoplus_{n\in\Z}\, \wt{\g}_n\ ,
	\end{align}
	with homogeneous components $\wt{\g}_0\coloneqq L_0$, and
	\begin{align}
		\wt{\g}_n\coloneqq \begin{cases}
			[\wt{\g}_{n-1},L_+] & \text{if $n>0$} \ ,\\[4pt]
			[\wt{\g}_{n+1},L_-] & \text{if $n<0$}\ .
		\end{cases}
	\end{align}
	
	\begin{definition}[\cite{saveliev-vershik-91}]\label{def:sav-ver}
		The \textit{Saveliev--Vershik algebra} associated to the datum 
		$(\vaa,\kappa_{\epsilon})$ is the Lie algebra  
		\begin{align}
			\g(\vaa,\kappa_{\epsilon})\coloneqq \wt{\g}(\vaa,\kappa_{\epsilon})/\r \ ,
		\end{align}
		where $\r$ is the sum of all two--sided homogeneous ideals in $\wt{\g}(\vaa,\kappa_{\epsilon})$ having trivial 
		intersection  with $L_0$. In particular, $\g(\vaa,\kappa_{\epsilon})$ is $\Z$--graded with $\g_0=L_0$.
	\end{definition}
	
	The examples given in \cite{saveliev-vershik-90-2, saveliev-vershik-92, vershik-92, vershik-02} 
	belong to a special case of this formulation, where $\vaa$ is a commutative algebra, 
	$\kappa, s\colon \vaa\to \vaa$ are distinguished linear mappings, and
	\begin{align}\label{eq:cont-loc-lie-simple}
		\kappa_{\pm}(\phi,\psi)\coloneqq \pm\psi\cdot\kappa(\phi)
		\quad\text{and}\quad
		\kappa_0(\phi,\psi)\coloneqq s(\phi\cdot\psi)
	\end{align}
	Note that in this case the conditions \eqref{eq:ver-datum-1} 
	and \eqref{eq:ver-datum-2} are automatically satisfied. 
	
	\begin{remark}
		Definition~\ref{def:sav-ver} provides a straightforward generalization of \textit{derived} 
		Kac--Moody algebras as described in Remark~\ref{sss:der-km} in terms of the usual $3|\bfI|$ Chevalley generators. 
		Namely, let $\bfI$ and $\gcm$ be as in Section~\ref{ss:km}.  Then, one sees immediately that
		$\g(\rls{})'=\g(\vaa,\kappa_{\epsilon})$, where $\vaa$ is the commutative algebra $\bsfld^{|\bfI|}$ endowed with coordinate multiplication, 
		$(e_i,h_i,f_i)=(X_+(v_i),X_0(v_i),X_-(v_i))$, $i\in\bfI$, where $v_i\in\bsfld^{|\bfI|}$ are the standard 
		unit vectors,  and the linear maps $\kappa_{\epsilon}$ are defined as in \eqref{eq:cont-loc-lie-simple} with $\kappa(v)
		\coloneqq\gcm v$ and $s(v)\coloneqq v$ for any $v\in\vaa$.
		
		It is notable that, while the Kac--Moody algebra associated to the minimal realization does not fit
		in this formalism, the \textit{canonical} Kac--Moody algebra $\g(\crls{\gcm})$ does, up to a minor change. Namely, it is enough to consider a local Lie algebra $L$ with $L_0=\vaa\oplus\vaa$, 
		whose additional elements are denoted $X_0^\vee(\phi)$, $\phi\in\vaa$ 
		and satisfy $[X_0^\vee(\phi),X_0^\vee(\psi)]=0=[X_0(\phi),X_0^\vee(\psi)] $ and
		$[X_0^\vee(\phi),X_{\pm}(\psi)]=\pm X_{\pm}(\phi\cdot\psi)$.
	\end{remark}
	
	\subsection{The Lie algebra of the line}\label{ss:semigroup-line}
	
	In \cite{sala-schiffmann-17}, the last-two-named authors introduced quantum groups $\bfU_\upsilon(\sl(\K))$
	with $\KZQR$. Since the cases of $\KZQ$ are easily deduced from that of $\K=\R$, for simplicity, we focus 
	only on the Lie algebra $\sl(\R)$.

	\subsubsection{Intervals}\label{sss:definition-sl(K)}
	
	Roughly speaking, $\sl(\R)$ is a Lie algebra generated by elements labeled by \textit{intervals} of $\R$, 
	subject to some quadratic relations, whose coefficients depend upon a bilinear form on the space of \textit{characteristic functions} of such intervals. The values of such bilinear form play the same role as the coefficient of a Cartan matrix. 
	
	In order to give the precise definition of $\sl(\R)$, we need to introduce some notation. 
	First, we say that a subset $J\subset \R$ is an \textit{interval} 
	if it is an open--closed interval of the form $J=(a,b]\coloneqq \{x\in \R\;\vert\; a<x\leqslant b\}$.  
	
	Let $\intsf{\R}{}$ be the set of all intervals in $\R$ and define two partial maps 
	$\sgpp,\sgpm\colon\intsf{\R}{}\times\intsf{\R}{}\to\intsf{\R}{}$ as follows:
	\begin{align}
		J\sgpp J'\coloneqq &
		\begin{cases}
			J\cup J' & \text{if } J\cap J'=\emptyset\text{ and }J\cup J'\text{ is connected}\ ,\\[4pt]
			\text{n.d.} & \text{otherwise}\ ,
		\end{cases}
		\label{eq:sgp-intv-add}\\
		J\sgpm J'\coloneqq &
		\begin{cases}
			J\setminus J' & \text{if } J\cap J'=J'\text{ and }J\setminus J'\text{ is connected}\ ,\\[4pt]
			\text{n.d.} & \text{otherwise}\ .
		\end{cases}\label{eq:sgp-intv-sub}
	\end{align}
	Denote by $\intsf{\R}{}^{(2)}_{\sgpp}$ the set of pairs $(J, J')$ such that $J\sgpp J'$ is defined. Similarly, let 
	$\intsf{\R}{}^{(2)}_{\sgpm}$ be the set of pairs $(J, J')$ that $J\sgpm J'$ is defined.  Set 
	\begin{align}
		\delta_{J\sgpp J'}\coloneqq 
		\begin{cases}
			1 & \text{if } (J, J')\in \intsf{\R}{}^{(2)}_{\sgpp}\ , \\
			0 & \text{otherwise} \ ,
		\end{cases}
		\quad\text{and}\quad
		\delta_{J\sgpm J'}\coloneqq 
		\begin{cases}
			1 & \text{if } (J, J')\in \intsf{\R}{}^{(2)}_{\sgpm}\ , \\
			0 & \text{otherwise} \ .
		\end{cases}
	\end{align}
	
	We adopt the following notation, distinguishing all relative positions of two intervals. For any two 
	intervals $J=(a,b]$ and $J'=(a',b']$, we write
	\begin{itemize}\itemsep0.3cm
		\item $J\intnext J'$ if $b=a'$ (\textit{adjacent}) 
		\item $J\perp J'$ if $b<a'$ or $b'<a$ (\textit{disjoints}) 
		\item $J\rsub J'$ if $a=a'$ and $b<b'$ (\textit{closed subinterval})
		\item $J\lsub J'$ if $a'<a$ and $b=b'$ (\textit{open subinterval})
		\footnote{The symbol $\rsub$ (resp. $\lsub$) should be read as \textit{$J$ is a proper subinterval in 
				$J'$ starting from the left (resp. right) endpoint}.}
		\item $J<J'$ if $a'<a<b<b'$ (\textit{strict subinterval})
		\item $J\intcap J'$ if $a<a'<b<b'$ (\textit{overlapping})
	\end{itemize}
	
	In particular we have that
	\begin{align}
		(J,J')\in\intsf{\R}{}^{(2)}_{\sgpp} &\quad\text{if and only if}\quad J\intnext J' \text{ or } J'\intnext J\ ,\\[2pt]
		(J,J')\in\intsf{\R}{}^{(2)}_{\sgpm} &\quad\text{if and only if}\quad J'\rsub J \text{ or } J'\lsub J\ .
	\end{align}
	
	\subsubsection{Euler form}
	We denote by $\fun{\R}{}$ the algebra of piecewise constant, left--continuous 
	functions $f\colon \R\to \R$, with bounded support and finitely many points of discontinuity. 
	More explicitly, $f\in\fun{\R}{}$ if and only if $f$ is a linear
	combination of finitely many characteristic functions $\cf{J}$ with $J\in\intsf{\R}{}$.
	
	For any $f,g\in\fun{\R}{}$, we set
	\begin{align}\label{eq:bilinear-form-interval}
		\abf{f}{g}\coloneqq \sum_xf_-(x)(g_-(x)-g_+(x)) \quad\text{and}\quad \rbf{f}{g}\coloneqq\abf{f}{g}+\abf{g}{f}
	\end{align}
	where $h_{\pm}(x)=\lim_{t\to0^+}h(x\pm t)$. We have
	\begin{align}\label{eq:abf-line}
		\abf{\cf{J}}{\cf{J'}}=
		\begin{cases}
			\phantom{+}1 & \text{if } J=J',\, J'\rsub J,\, J\lsub J',\, J'\intcap J\ ,\\[4pt]
			\phantom{+}0 & \text{if } J\perp J',\, J'\intnext J,\, J\rsub J',\, J'\lsub J,\, J<J',\, J'<J\ ,\\[4pt]
			-1 & \text{if } J\intnext J',\, J\intcap J'\ .
		\end{cases}
	\end{align}
	Thus, 
	\begin{align}
		\rbf{\cf{J}}{\cf{J'}}=
		\begin{cases}
			{\phantom{+}}2 & \text{if } J=J'\ ,\\[4pt]
			{\phantom{+}}1 & \text{if } (J,J')\in\intsf{\R}{}^{(2)}_{\sgpm} \text{ or } (J',J)\in\intsf{\R}{}^{(2)}_{\sgpm}\ ,\\[4pt]
			-1 & \text{if } (J,J')\in\intsf{\R}{}^{(2)}_{\sgpp}\ ,\\[4pt]
			{\phantom{+}}0 & \text{otherwise}\ .
		\end{cases}
	\end{align}

	\subsubsection{The first definition}  We are ready to give the definition of $\sl(\R)$.
	\begin{definition}\label{def:sl(K)}
		Let $\sl(\R)$ be the Lie algebra generated by elements $e_J, f_J, h_J$, with $J\in\intsf{\R}{}$, 
		modulo the following set of relations:
		\begin{itemize}\itemsep0.3cm
			\item \textit{Kac--Moody type relations:} for any two intervals $J_1, J_2$,
			\begin{align}\label{eq:kac-rel-1}
				\begin{aligned}
					[h_{J_1},h_{J_2}] &=0\ ,\\[3pt]
					[h_{J_1}, e_{J_2}] &=( \cf{J_1},\cf{J_2})\, e_{J_2}\ , \\[3pt]
					[h_{J_1}, f_{J_2}] &=-( \cf{J_1},\cf{J_2})\, f_{J_2}\ , 
				\end{aligned}
			\end{align}
			\begin{align}\label{eq:kac-rel-2}
				[e_{J_1},f_{J_2}]=
				\begin{cases}
					h_{J_1} & \text{if }  J_1=J_2\ ,\\[3pt]
					0 & \text{if }  J_1\perp J_2, J_1\intnext J_2, \text{ or } J_2\intnext J_1\ ,
				\end{cases}
			\end{align}
			
			\item \textit{join relations:} for any two intervals $J_1, J_2$ with $(J_1,J_2)\in\intsf{\R}{\K}^{(2)}_{\sgpp}$,
			\begin{align}\label{eq:Joining-h-finite}
				h_{J_1\sgpp J_2}&=h_{J_1}+ h_{J_2}\ ,\\[3pt]
				e_{J_1\sgpp J_2}&=(-1)^{\abf{\cf{J_2}}{\cf{J_1}}}
				[e_{J_1}, e_{J_2}]\ ,\label{eq:Joining-e-finite}\\[3pt]
				f_{J_1\sgpp J_2}&= (-1)^{\abf{\cf{J_1}}{\cf{J_2}}}
				[f_{J_1}, f_{J_2}]\ ,\label{eq:Joining-f-finite}
			\end{align}
			
			\item \textit{nest relations:} for any nested $J_1,J_2\in\intsf{\R}{\K}$ (that is, such that	$J_1=J_2$, 
			$J_1\perp J_2$, $J_1<J_2$, $J_2<J_1$,$J_1\rsub J_2$, $J_1\lsub J_2$, $J_2\rsub J_1$, or $J_2\lsub J_1$),
			\begin{align}\label{eq:nest-finite}
				[e_{J_1},e_{J_2}]=0\quad\text{and}\quad [f_{J_1}, f_{J_2}]=0\ .
			\end{align}
			
		\end{itemize}
	\end{definition}
	\begin{remark}
		It is easy to check that the bracket is anti--symmetric and satisfies the Jacobi identity.
		Note that the join relations are consistent with anti--symmetry, since, whenever $J\sgpp J'$ is defined, 
		$(-1)^{\abfcf{J}{J'}}=-(-1)^{\abfcf{J'}{J}}$. Moreover, the combination of join and nest relations yields 
		the \textit{(type $\sfA$) Serre relations}  ($J\neq J'$)
		\begin{align}\label{eq:serre-sl(K)}
			\begin{aligned}
				\begin{array}{ll}
					[e_J, [e_J, e_{J'}]]=0=[f_J, [f_J, f_{J'}]] & \text{if } \rbfcf{J}{J'}=-1 \ ,  \\[4pt]
					[e_J, e_{J'}]=0=[f_J, f_{J'}] & \text{if } \rbfcf{J}{J'}=0\ .
				\end{array}
			\end{aligned}
		\end{align}
	\end{remark}

	\subsubsection{A new presentation}
	The datum $(\intsf{\R}{}, \sgpp, \sgpm)$ has the role of a \textit{continuum root system} of 
	$\sl(\R)$. It is therefore natural to ask whether $\sl(\R)$ is an example of a Saveliev--Vershik algebra. 
	Namely, set $\calA\coloneqq\fun{\R}{}$, and define the maps $\kappa_{0}, \kappa_{\pm}
	\colon \calA\times \calA\to \calA$ by setting
	\begin{align}
		\kappa_{0}(\cf{J},\cf{J'})\coloneqq \drc{J,J'}\cf{J'}\quad\text{and}\quad
		\kappa_{\pm}(\cf{J},\cf{J'})\coloneqq \pm\rbf{\cf{J}}{\cf{J'}}\cf{J}\ .
	\end{align}
	The maps $\kappa_0$ and $\kappa_\pm$ satisfy the relations \eqref{eq:ver-datum-1} and 
	\eqref{eq:ver-datum-2}. Thus there exists a continuum Lie algebra $\g(\fun{\R}{}, \kappa_{0}, \kappa_{\pm})$. 
	Nonetheless, we shall show in this section that $\g(\fun{\R}{}, \kappa_{0}, \kappa_{\pm})$ and $\sl(\R)$ are not 
	the same. More generally, the latter cannot be a Saveliev--Vershik algebra. 
	The first result we need is the following.
	\begin{proposition}\label{prop:relationssl}
		The relations \eqref{eq:kac-rel-2}, \eqref{eq:Joining-e-finite}, \eqref{eq:Joining-f-finite}, 
		and \eqref{eq:nest-finite} can be replaced by
		\begin{align}
			[e_{J_1},f_{J_2}]&=\drc{J_1,J_2}h_{J_1}+(-1)^{\abf{\cf{J_1}}{\cf{J_2}}}\rbf{\cf{J_1}}{\cf{J_2}}
			\left(e_{J_1\sgpm J_2}-f_{J_2\sgpm J_1}\right)\ ,\label{eq:kac-rel-3-new}
		\end{align}
		and
		\begin{align}\label{eq:serre-new}
			\begin{aligned}
				[e_{J_1}, e_{J_2}]&=(-1)^{\abf{\cf{J_2}}{\cf{J_1}}}e_{J_1\sgpp J_2}\ ,\\
				[f_{J_1}, f_{J_2}]&=(-1)^{\abf{\cf{J_1}}{\cf{J_2}}}f_{J_1\sgpp J_2}\ .
			\end{aligned}
		\end{align}
	\end{proposition}
	
	\begin{proof}
		It is clear that \eqref{eq:kac-rel-3-new} reduces to \eqref{eq:kac-rel-2}, while 
		\eqref{eq:serre-new} reduces to \eqref{eq:Joining-e-finite}, \eqref{eq:Joining-f-finite}, 
		and \eqref{eq:nest-finite}, since $J_1\sgpp J_2$ is not defined 
		whenever $J_1$ and $J_2$ are nested, and thus the RHS of \eqref{eq:serre-new} equals zero. Conversely, we shall prove that \eqref{eq:kac-rel-3-new} and \eqref{eq:serre-new} hold in $\sl(\R)$. The proof is based on a case--by--case inspection described in Appendix~\ref{app:relationssl}.
	\end{proof}
	
	We are now able to give a more efficient presentation of $\sl(\R)$. In order to stress the analogy 
	with the definition provided in Section~\ref{s:topological-quiver}, we adopt a slightly different
	notation. Set $\xz{J}\coloneqq h_J$, $\xp{J}\coloneqq e_J$, $\xm{J}\coloneqq f_J$ for any 
	inteval $J$.
	
	\begin{corollary}\label{cor:slR-presentation}
		$\sl(\R)$ is presented on the generators $\xpm{J}$, $\xz{J}$, $J\in\intsf{\R}{}$, with relations
		\begin{align}
			\xz{J\sgpp J'} & = \drc{J\sgpp J'}(\xz{J}+\xz{J'})\ ,\\[2pt]
			[\xz{J},\xz{J'}] & = 0\ ,\\[2pt]
			[\xz{J},\xpm{J'}] & = \pm\rbf{\cf{J}}{\cf{J'}}\xpm{J'}\ ,\\[2pt] \label{eq:non-local}
			[\xp{J},\xm{J'}] & = \drc{J,J'}\xz{J}+(-1)^{\abf{\cf{J}}{\cf{J'}}}\rbf{\cf{J}}{\cf{J'}}
			\left(\xp{J\sgpm J'}-\xm{J'\sgpm J}\right)\ ,\\[2pt]
			[\xpm{J},\xpm{J'}] & = \pm(-1)^{\abf{\cf{J'}}{\cf{J}}}\xpm{J\sgpp J'}\ ,
		\end{align}
		where we assume that $\xe{J_1\odot J_2}=0$ whenever $J_1\odot J_2$ 
		is not defined, for $\odot=\sgpp,\sgpm$ and $\epsilon=\pm,0$. 
	\end{corollary}
	The above presentation of $\sl(\R)$ makes clear that the main reason for which 
	$\sl(\R)$ cannot coincide with $\g(\fun{\R}{}, \kappa_0, \kappa_\pm)$ is the relation 
	\eqref{eq:non-local}. Indeed, the latter Lie algebra is, by definition, generated by a local 
	Lie algebra $L$ (cf.\ Section~\ref{ss:cont-lie}), while this is not possible $\sl(\R)$.
	As we shall see in the following sections, we need to consider a wider class of Lie algebras,  containing both 
	Saveliev--Vershik Lie algebras (hence, also Kac--Moody algebras) and $\sl(\R)$ by relaxing the 
	condition of locality in the spirit of  relation \eqref{eq:non-local}.

	\begin{remark}\label{rmk:slK}
		As we pointed out at the beginning of this section, the description of the Lie algebras $\sl(\Z)$ and $\sl(\Q)$
		are easily deduced from that of $\sl(\R)$. More precisely, one can consider the subset $\intsf{\R}{\K}\subset\intsf{\R}{}$
		consisting of intervals with boundaries in $\K=\Z,\Q$. The Lie algebra $\sl(\K)$, $\KZQ$, is then realized
		as the subalgebra in $\sl(\R)$ generated by $\xpm{J}$ and $\xz{J}$ with $J\in\intsf{\R}{\K}$. In particular, there
		is a canonical chain of embeddings $\sl(\Z)\subseteq\sl(\Q)\subseteq\sl(\R)$.
	\end{remark}
	
	\begin{remark}
		We conclude this section by observing that $\sl(\R)$ should be rather thought of as
		a continuum analogue of the familiar Lie algebra $\sl(\infty)$.
		Let
		\begin{align}
			\begin{tikzpicture}[xscale=1.5,yscale=-0.5]
				\node (A0_0) at (0, 0) {$1$};
				\node (A0_1) at (1, 0) {$2$};
				\node (A0_2) at (2, 0) {$3$};
				\node (A0_3) at (3, 0) {$4$};
				\node (A0_5) at (5, 0) {$n-2$};
				\node (A0_6) at (6, 0) {$n-1$};
				\node (A1_0) at (0, 1) {$\bullet$};
				\node (A1_1) at (1, 1) {$\bullet$};
				\node (A1_2) at (2, 1) {$\bullet$};
				\node (A1_3) at (3, 1) {$\bullet$};
				\node (A1_5) at (5, 1) {$\bullet$};
				\node (A1_6) at (6, 1) {$\bullet$};
				\path (A1_0) edge [-]node [auto] {$\scriptstyle{}$} (A1_1);
				\path (A1_1) edge [-]node [auto] {$\scriptstyle{}$} (A1_2);
				\path (A1_2) edge [-]node [auto] {$\scriptstyle{}$} (A1_3);
				\path (A1_3) edge [dashed, -]node [auto] {$\scriptstyle{}$} (A1_5);
				\path (A1_5) edge [-]node [auto] {$\scriptstyle{}$} (A1_6);
			\end{tikzpicture} 
		\end{align}
		be the Dynkin diagram of type $A_n$. We can consider two different limits for $n\to +\infty$: the infinite Dynkin diagram
		\begin{align}
			\begin{tikzpicture}[xscale=1.5,yscale=-0.5]
				\node (A1_0) at (0, 1) {$\bullet$};
				\node (A1_1) at (1, 1) {$\bullet$};
				\node (A1_2) at (2, 1) {$\bullet$};
				\node (A1_3) at (3, 1) {$\bullet$};
				\node (A1_5) at (5, 1) {$ $};
				\path (A1_0) edge [-]node [auto] {$\scriptstyle{}$} (A1_1);
				\path (A1_1) edge [-]node [auto] {$\scriptstyle{}$} (A1_2);
				\path (A1_2) edge [-]node [auto] {$\scriptstyle{}$} (A1_3);
				\path (A1_3) edge [dashed, -]node [auto] {$\scriptstyle{}$} (A1_5);
			\end{tikzpicture} 
		\end{align}
		which gives rise to the infinite--dimensional Lie algebra $\sl(+\infty)$, and the infinite Dynkin diagram
		\begin{align}\label{eq:limit-II}
			\begin{tikzpicture}[xscale=1.5,yscale=-0.5]
				\node (A1_0') at (-2, 1) {$ $};
				\node (A1_0) at (0, 1) {$\bullet$};
				\node (A1_1) at (1, 1) {$\bullet$};
				\node (A1_2) at (2, 1) {$\bullet$};
				\node (A1_3) at (3, 1) {$\bullet$};
				\node (A1_5) at (4, 1) {$\bullet$};
				\node (A1_6) at (6, 1) {$ $};
				\path (A1_0) edge [-]node [auto] {$\scriptstyle{}$} (A1_1);
				\path (A1_1) edge [-]node [auto] {$\scriptstyle{}$} (A1_2);
				\path (A1_2) edge [-]node [auto] {$\scriptstyle{}$} (A1_3);
				\path (A1_3) edge [-]node [auto] {$\scriptstyle{}$} (A1_5);
				\path (A1_5) edge [dashed, -]node [auto] {$\scriptstyle{}$} (A1_6);
				\path (A1_0') edge [dashed, -]node [auto] {$\scriptstyle{}$} (A1_0); 
			\end{tikzpicture} 
		\end{align}
		which corresponds to the infinite dimensional Lie algebra $\sl(\infty)$. 
		One can show that $\sl(+\infty)$ coincides with $\sl(\infty)$ (cf.\ \cite[Section 2]{dimitrov-penkov-99}). 
		The Lie algebra $\sl(\infty)$ admits a Kac--Moody algebra type description with generators $e_i, f_i, h_i$ 
		with $i\in \Z$ and the infinite Cartan matrix $\gcm = (a_{ij})_{i, j\in \Z}$, with $a_{ii}=2$, $a_{ij}=-1$, if 
		$\vert i-j\vert =1$, and $a_{ij}=0$ otherwise. 
		
		In view of the Serre relations \eqref{eq:serre-sl(K)}, there is a canonical embedding 
		$\sl(\infty)\to \sl(\Z)$ given by
		\begin{align}
			e_i \mapsto e_{(i, \, i+1]}\ , \quad 	f_i \mapsto f_{(i, \, i+1]} \ , \quad 	h_i \mapsto h_{(i, \, i+1]}\ . 
		\end{align}
		Moreover, since $\{e_{(i, \, i+1]}, f_{(i,\,  i+1]}, h_{(i,\,  i+1]}\; \vert \; \forall\, i\in \Z\}$
		is a minimal set of generators, this gives a canonical isomorphism $\sl(\infty)\simeq \sl(\Z)$.
	\end{remark}
	
	\bigskip\section{Semigroup Lie algebras}\label{s:sgp-lie}
	
	In this section, we introduce a \textit{semi--local} version of contragredient Lie algebras \cite{feigin-zelevinsky-85},
	as a straightforward generalization of the new presentation of $\sl(\R)$ given in Corollary 
	\ref{cor:slR-presentation}.
	We describe a special class of examples, whose combinatorics depends upon the choice of a
	partial semigroups. 
	
	\subsection{Semi--local continuum Lie algebras}\label{ss:non-loc-cont-lie}
	
	It is clear from the previous section that the Lie algebra $\sl(\R)$ does
	not fit in the usual Kac--Moody framework (nor in the one introduced by 
	Saveliev--Vershik). In particular, it is not generated
	by a local Lie algebra and its Cartan elements satisfy the linearity condition \eqref{eq:Joining-h-finite}. 
	This motivates the following definition and the subsequent construction.
	
	\begin{definition}
		We say that a direct sum of vector spaces $L=L_{-}\oplus L_0\oplus L_{+}$ is a \textit{semi--local} Lie algebra
		if there are bilinear maps $L_{0}\times L_{\pm}\to L_{\pm}$ and $L_+\times L_-\to L$ 
		such that the Jacobi identity holds whenever it make sense. 
	\end{definition}
	
	Let $(\vaa,\cdot)$ be an arbitrary associative $\bsfld$--algebra 
	endowed with six bilinear mappings 
	$\kappa_{\epsilon},\xi_{\epsilon}\colon\vaa\times \vaa\to \vaa$, with 
	$\epsilon= \pm,0$, satisfying 
	\begin{align}
		\kappa_{\pm}(\xi_0(\phi,\psi),\chi)=\drc{\xi_0(\phi,\psi)}\kappa_{\pm}(\phi+\psi,\chi)\label{eq:new-datum-4}
		\quad\text{and}\quad
		[\xi_0(\phi,\psi),\chi]=\drc{\xi_0(\phi,\psi)}[\phi+\psi,\chi]
	\end{align}
	where $\phi,\psi,\chi\in\vaa$ and $\drc{\xi_0(\phi,\psi)}$ is the \textit{characteristic function of the support 
		of $\xi_0$}, \ie $\drc{\xi_0(\phi,\psi)}=1$ if $\xi_0(\phi,\psi)\neq0$ and it is zero otherwise.
	
	Set $L=L_+\oplus L_0 \oplus L_-$, where $L_{\pm}$ is identified with $\vaa$ through
	a linear isomorphism $\xpm{\bullet}\colon\vaa\to L_{\pm}$ and $L_0$ is identified with the quotient $\vaa/K(\xi_0)$, where $K(\xi_0)$ is the span of vectors 
	$\xi_0(\phi,\psi)-\drc{\xi_0(\phi,\psi)}(\phi+\psi)$, $\phi,\psi\in\vaa$, through a linear isomorphism 
	$\xz{\bullet}\colon\vaa/K(\xi_0)\to L_0$. We define a bilinear mapping 
	\begin{align}
		[\cdot,\cdot]\colon L\ten L\to L
	\end{align}
	as follows:
	\begin{align}
		[\xz{\phi},\xz{\psi}] &= \xz{[\phi,\psi]}\ ,
		\label{eq:new-la-1}\\
		[\xz{\phi},\xpm{\psi}] &= \xpm{\kappa_{\pm}(\phi,\psi)}\ ,
		\label{eq:new-la-2}\\
		[\xp{\phi},\xm{\psi}] &= \xp{\xi_+(\phi,\psi)}+\xz{\kappa_{0}(\phi,\psi)}+\xm{\xi_-(\phi,\psi)}\ ,
		\label{eq:new-la-3}
	\end{align}
	where $\phi,\psi\in\vaa$.  The following is straightforward.
	
	\begin{lemma}\label{lem:Jacobi-2}
		The map $[\cdot,\cdot]$ is well--defined. In addition, the vector space $L$ is a semi--local Lie algebra with 
		bracket $[\cdot,\cdot]$ if and only if 
		the map $\kappa_{\epsilon}$ satisfies the relations \eqref{eq:ver-datum-1}, 
		\eqref{eq:ver-datum-2}, and $\kappa_{\epsilon}, \xi_{\epsilon}$ satisfy
		\begin{align}\label{eq:new-datum-3}
			\kappa_{\pm}(\phi,\xi_{\pm}(\psi,\chi))=\xi_{\pm}(\kappa_+(\phi,\psi),\chi)+\xi_{\pm}(\psi,\kappa_-(\phi,\chi))\ ,
		\end{align}
		where $\phi,\psi,\chi\in\vaa$.
	\end{lemma}

	We shall think of the tuple $(\vaa,\kappa_{\epsilon}, \xi_{\epsilon})$ as a Cartan datum and
	develop its corresponding Kac--Moody theory.
	
	\begin{definition}\label{def:semicontLiealgebra}
		The tuple $(\vaa,\kappa_{\epsilon}, \xi_{\epsilon})$ is a \textit{continuum Cartan datum}
		if $\kappa_{\epsilon}$ satisfies \eqref{eq:ver-datum-1}, \eqref{eq:ver-datum-2}, and $\kappa_{\epsilon}, \xi_{\epsilon}$ satisfy
		\eqref{eq:new-datum-4}, \eqref{eq:new-datum-3}.
		Then, the \textit{(semi--local) continuum Lie algebra} associated to 
		$(\vaa,\kappa_{\epsilon}, \xi_{\epsilon})$ is the Lie algebra
		\begin{align}
			\g(\vaa,\kappa_{\epsilon},\xi_{\epsilon})\coloneqq \wt{\g}(\vaa,\kappa_{\epsilon},\xi_{\epsilon})/\r \ ,
		\end{align}
		where $\wt{\g}(\vaa,\kappa_{\epsilon},\xi_{\epsilon})$ is the Lie algebra freely generated by the
		semi--local Lie algebra $L$ and $\r\subset\wt{\g}(\vaa,\kappa_{\epsilon}, \xi_{\epsilon})$ is the sum of 
		all two--sided ideals having trivial intersection with $L_0$. 
	\end{definition}
	
	Let $\n_{\pm}\subset\g(\vaa,\kappa_{\epsilon},\xi_{\epsilon})$ be the Lie subalgebra generated by 
	$\{\xpm{\phi}\;\vert \; \phi\in\vaa\}$. We shall make use of the following standard result.
	
	\begin{lemma}\label{lem:vanishing}
		Let $S$ be an index set, let $\{X_a\}$ be a collection of elements of $\n_{\pm}$ indexed by $a\in S$, and let
		$X_S=\opspan\{X_a\;\vert\; a\in S\}$. 
		If $[X_S, \xmp{\phi}]\subseteq X_S$ for any $\phi\in\vaa$, then $X_S=0$.
	\end{lemma} 
	
	\begin{proof}
		Assume $X_S\subseteq\n_+$ and
		set $\r_X\coloneqq \sum_{i,j}\sfad(L_+)^i\sfad(L_0)^jX_S$.
		$\r_X$ is a subspace of $\n_+$, which is clearly invariant 
		under $\sfad(L_0)$ and $\sfad(L_+)$. Moreover, since $\sfad(L_-)X_S\subseteq X_S$, 
		one also has $\sfad(L_-)\r_X\subset\r_X$. In particular, whenever $X_S\neq\{0\}$, the subspace 
		$\r_X$ is a non--zero ideal trivially intersecting $L_0$. Therefore, necessarily, $\r_X=0$ and 
		$X_S=0$. The case $X_S\subseteq\n_-$ is similar.
	\end{proof}

	\subsection{Cartan semigroups}
	
	We shall give a combinatorial description of certain semi--local continuum Lie algebras, 
	whose defining relations \eqref{eq:new-la-1}, \eqref{eq:new-la-2}, \eqref{eq:new-la-3} are 
	controlled by a class of partial semigroups, which we refer to as \textit{Cartan semigroups}. 
	
	Henceforth, with a slight abuse of terminology, by a \textit{semigroup} we mean a positive, commutative, partial semigroup with a maximal cancellation law (cf. Appendix~\ref{app:sgp}).
	
	Let $\sgp$ be a semigroup with commutative product $\sgpp$ and cancellation law $\sgpm$.
	Let $\kappa\colon \sgp\times\sgp\to\bsfld$ be a function such that
	\begin{align}
		\kappa(\alpha\sgpp\beta,\gamma)=\drc{\alpha\sgpp\beta}\left(\kappa(\alpha,\gamma)+\kappa(\beta,
		\gamma)\right)\label{eq:sgp-datum-1}\ ,
	\end{align}
	where $\drc{\alpha\sgpp\beta}$ is the \textit{characteristic function\footnote{This means that $\drc{\alpha\sgpp\beta}$ takes value one if $(\alpha, \beta)\in \dsgp_{\sgpp}$, and it takes value zero otherwise.} of $\dsgp_{\sgpp}$} and, by convention, $\kappa(\alpha\sgpp\beta,\gamma)=0=\kappa(\alpha,\beta\sgpp\gamma)$ whenever $\alpha\sgpp\beta$ and $\beta\sgpp\gamma$ are not defined.
	Set $L^{\sgp}=L^{\sgp}_+\oplus L^{\sgp}_0 \oplus L^{\sgp}_-$,
	where
	\begin{align}
		L^{\sgp}_{\pm}\coloneqq \bigoplus_{\alpha\in\sgp}\bsfld\cdot \xpm{\ia}
		\quad\text{and}\quad
		L^{\sgp}_0\coloneqq \left(\bigoplus_{\alpha\in\sgp}\bsfld\cdot \xz{\ia}\right)/N^{\sgp}\ ,
	\end{align}
	and $N^{\sgp}$ is the subspace spanned by the elements of the form
	\begin{align}
		\xz{\ia\sgpp\ib}-\drc{\alpha\sgpp\beta}\left(\xz{\ia}+\xz{\ib}\right)\ ,
	\end{align}
	where we assume that $\xpm{\alpha\sgpp\beta}=0=\xz{\alpha\sgpp\beta}$ if $(\alpha,\beta)\not\in\dsgp_{\sgpp}$. By a slight abuse of notation, we will denote the class of $\xz{\ia}$ in $L_0^{\sgp}$ by the same symbol.
	Given two functions $\xi_{\pm}\colon \sgp\times\sgp\to\bsfld$, we define a bilinear map $[\cdot,\cdot]\colon L^{\sgp}\times L^{\sgp}\to L^{\sgp}$ (whose dependence by $\xi_{\pm}$ 
	is omitted) by
	\begin{align}
		[\xz{\ia},\xz{\ib}]=& 0\ ,\label{eq:sgp-lie-2}\\
		[\xz{\ia},\xpm{\ib}]=& \pm\kappa(\alpha,\beta)\xpm{\ib}\ ,\label{eq:sgp-lie-3}\\
		[\xp{\ia},\xm{\ib}]=& \drc{\alpha,\beta}\xz{\ia}+
		\xi_+(\alpha,\beta)\xp{\ia\sgpm\ib}-\xi_-(\beta,\alpha)\xm{\ib\sgpm\ia}\ ,\label{eq:sgp-lie-4}
	\end{align}
	where we assume that $\xpm{\ia\sgpm\ib}=0=\xz{\ia\sgpm\ib}$ if $(\alpha,\beta)\not\in\dsgp_{\sgpm}$.
	Note that the condition \eqref{eq:sgp-datum-1} is equivalent to require $[N^{\sgp}, L^{\sgp}]=0$. 
	Therefore, the map $[\cdot,\cdot]$ is well--defined. The following is straightforward 
	(cf. Lemmas ~\ref{lem:Jacobi-1} and \ref{lem:Jacobi-2}).
	
	\begin{lemma}
		The vector space $L^{\sgp}$ is a semi--local Lie algebra with bracket $[\cdot,\cdot]$ if
		and only if
		\begin{align}\label{eq:sgp-datum-2}
			\xi_{\pm}(\beta,\gamma)\kappa(\alpha,\beta\sgpm\gamma)&=\drc{\beta\sgpm\gamma}\,
			\xi_{\pm}(\beta,\gamma)\left(\kappa(\alpha,\beta)
			-\kappa(\alpha,\gamma)\right)
		\end{align}
		where we assume that $\kappa(\alpha,\beta\sgpm\gamma)=0$ if $(\beta,\gamma)\not\in\dsgp_{\sgpm}$.
	\end{lemma}

	\begin{remark}
		Note that if $\kappa\colon\sgp\times\sgp\to\bsfld$ is symmetric or satisfies
		\begin{align}\label{eq:kappa-right-dec}
			\kappa(\alpha,\beta\sgpp\gamma)=\drc{\beta\sgpp\gamma}\left(\kappa(\alpha,\beta)+\kappa(\alpha,\gamma)\right)
		\end{align}
		then equation \eqref{eq:sgp-datum-2} is automatically satisfied for any choice of $\xi_{\pm}$. 
	\end{remark}
	
	We shall regard the tuple $(\sgp,\kappa,\xi_{\pm})$ as a generalization of the usual Cartan datum.
	
	\begin{definition}\label{def:Cartan-sgp}
		A \textit{Cartan semigroup} is a tuple $\csgp=(\sgp,\kappa,\xi_{\pm})$, where $\sgp$ 
		is a semigroup, $\kappa\colon \sgp\times\sgp\to\bsfld$ is a function satisfying \eqref{eq:sgp-datum-1} and \eqref{eq:kappa-right-dec}, and $\xi_{\pm}\colon \sgp\times\sgp\to\bsfld$ are two arbitrary functions. We denote by $\wt{\g}(\csgp)$ 
		the Lie algebra freely generated by $L^{\sgp}$. 
	\end{definition}

	\subsection{Semigroup Lie algebras}\label{ss:sgp-lie}
	We denote by $\sgprt{\alpha}$ the element of the standard basis of $\Z^{\sgp}$ for $\alpha\in\sgp$. Then, we set $\rtl^{\sgp}=\Z^{\sgp}/\sim$, where
	$\sim$ is the relation $\sgprt{\alpha\sgpp\beta}=\delta_{\alpha\sgpp\beta}(\sgprt{\alpha}+\sgprt{\beta})$, and
	\begin{align}
		\rtl^{\sgp}_+\coloneqq\mathsf{span}_{\Z_{\geqslant0}}\{\sgprt{\alpha}\;\vert\;\alpha\in\sgp\}\subset\rtl^{\sgp}\ .
	\end{align}
	Set $\rtl^{\sgp}_-\coloneqq-\rtl^{\sgp}_+$, so that $\rtl^{\sgp}=\rtl_+^{\sgp}\oplus\rtl_-^{\sgp}$. For 
	$\lambda,\mu\in\rtl^{\sgp}$, we say that $\mu\preccurlyeq\lambda$ if and only if $\lambda-\mu\in\rtl^{\sgp}_+$. The following is standard.

	\begin{proposition}\label{prop:triang-dec-1}
		\hfill
		\begin{enumerate}
			\item As vector spaces, $\wt{\g}(\csgp)=\wt{\n}_+\oplus L_0^{\sgp}\oplus\wt{\n}_-$, where
			$\wt{\n}_{\pm}$ is the subalgebra  generated by the elements $\xpm{\ia}$, $\alpha\in\sgp$. Moreover,
			$\wt{\n}_{\pm}$ is freely generated.
			\item There is a natural $\rtl^{\sgp}$--gradation on $\wt{\g}(\csgp)$ given by $\deg(\xpm{\ia})=\pm\sgprt{\alpha}$ and $\deg(\xz{\ia})=0$. In particular, 
			\begin{align}\label{eq:triang-dec-sgp}
				\wt{\g}(\csgp)
				=\left(\bigoplus_{\mu\in\rtl_+^{\sgp}\setminus\{0\}}\wt{\g}_{\mu}\right)
				\oplus L_0^{\sgp}\oplus\left(\bigoplus_{\mu\in\rtl_+^{\sgp}\setminus\{0\}}\wt{\g}_{-\mu}\right)
			\end{align}
			and $\wt{\g}_{\pm\mu}\subseteq\wt{\n}_{\pm}$.
		\end{enumerate}
	\end{proposition}
	
	\begin{defin}\label{def:semigroupLiealgebra}
		The \textit{semigroup Lie algebra with Cartan datum} $\csgp$ 
		is the Lie algebra 
		\begin{align}
			\g(\csgp)\coloneqq \wt{\g}(\csgp, \kappa,\xi_{\pm})/\r\ ,
		\end{align}
		where $\r$ is the sum of all two--sided $\rtl^{\sgp}$--graded ideals in $\wt{\g}(\csgp)$ 
		having trivial intersection with $L^{\sgp}_0$. \hfill$\oslash$
	\end{defin}
	
	In particular, it follows immediately from Proposition~\ref{prop:triang-dec-1} that $\g(\csgp)$ inherits the triangular decomposition $\g(\csgp)=\n_+\oplus L_0^{\sgp}\oplus\n_-$, where $\n_{\pm}\subset\g(\csgp)$ denotes the subalgebra generated by $\xpm{\ia}$, $\alpha\in\sgp$, and the $\rtl^{\sgp}$--gradation
	\begin{align}
		{\g}(\csgp)=\left(\bigoplus_{\mu\in\rtl_+^{\sgp}\setminus\{0\}}{\g}_{\mu}\right)
		\oplus L_0^{\sgp}\oplus\left(\bigoplus_{\mu\in\rtl_+^{\sgp}\setminus\{0\}}{\g}_{-\mu}\right)
	\end{align}
	where $\g_{\pm\mu}\subseteq\n_{\pm}$. 
	
	\begin{definition}
		We call \textit{root} an element $\mu\in \sfQ^{\sgp}\setminus\{0\}$ such that $\g_\mu\neq 0$. We say that a root $\mu$ is \textit{positive} (resp.\ \textit{negative}) if $\mu\succ0$ (resp.\ $\mu\prec0$). The set of roots (resp. positive, negative roots) is denoted by $\rts^{\sgp}$ (resp. $\rts^{\sgp}_+$, $\rts^{\sgp}_-$).
	\end{definition}
	
	\begin{remark}
		Let $\csgp$ be a Cartan semigroup and let $f\colon \sgp \to (\bsfld, +, -)$ be a homomorphism of partial semigroups. Then, we may define a gradation with respect to $f$ by setting
		\begin{align}
			\deg \xpm{\ia} \coloneqq \pm f(\ia) \quad\text{and}\quad \deg \xz{\ia} \coloneqq 0\ .
		\end{align}  
	\end{remark}	
	
	\begin{remark}
		As expected, a semigroup Lie algebra is a special cases of a semi--local 
		continuum Lie algebra as in Definition~\ref{def:semicontLiealgebra}. 
		Namely, let $\csgp$ be a Cartan semigroup,
		$\vaa_{\csgp}=\bsfld[\sgp]$ the algebra of regular functions over the set $\sgp$, and 
		$\cf{\alpha}\in\vaa_{\csgp}$ the characteristic function of 
		$\alpha\in\sgp$. For any $\alpha,\beta\in\sgp$, define
		\begin{align}
			\kappa_{\csgp,0}(\cf{\alpha},\cf{\beta})=\drc{\alpha,\beta}\cf{\alpha}\quad\text{and}\quad
			\kappa_{\csgp,\pm}(\cf{\alpha},\cf{\beta})=\pm\kappa(\alpha,\beta)\cf{\beta}
		\end{align}
		and
		\begin{align}
			\xi_{\csgp,0}(\cf{\alpha},\cf{\beta})&=\drc{\alpha\sgpp\beta}\cf{\alpha\sgpp\beta}\\
			\xi_{\csgp,+}(\cf{\alpha},\cf{\beta})&=\drc{\alpha\sgpm\beta}\xi_+(\alpha,\beta)\cf{\alpha\sgpm\beta}\\
			\xi_{\csgp,-}(\cf{\alpha},\cf{\beta})&=-\drc{\beta\sgpm\alpha}\xi_-(\beta,\alpha)\cf{\beta\sgpm\alpha}
		\end{align}
		where we assume that $\cf{\alpha\odot\beta}=0$ whenever $(\alpha,\beta)\not\in\dsgp_{\odot}$,
		$\odot=\sgpp,\sgpm$. Then, the assignment $\xpm{\ia}\to X_{\pm}(\cf{\alpha})$ and $\xz{\ia}\to X_{0}(\cf{\alpha})$ gives a
		Lie algebra isomorphism
		\begin{align}
			\g(\csgp)\simeq\wt{\g}(\vaa_{\csgp},\kappa_{\csgp,\epsilon},\xi_{\csgp,\epsilon})/\wt{\r}
		\end{align}
		where $\wt{\r}$ is the sum of all two--sided graded ideals in $\wt{\g}(\vaa_{\csgp},\kappa_{\csgp,\epsilon},\xi_{\csgp,\epsilon})$
		having trivial intersection with its Cartan subalgebra. In particular, $\g(\csgp)$ 
		is a graded semi--local continuum Lie algebra.\hspace{-0.5cm} 
	\end{remark}
	
	\subsection{Derived Kac--Moody algebras and semigroups}
	Derived Kac--Moody algebras are easily realized as degenerate examples of semigroup 
	Lie algebras. We use the notation from Section~\ref{s:km-vkm}. Let $\sgp$ be the trivial semigroup 
	with underlying set $\Pi=\{\alpha_i\,|\,i\in\bfI\}$ and $\dsgp_{\sgpp}=\emptyset=\dsgp_{\sgpm}$.
	Then, $\csgp=(\sgp,\kappa_{\gcm},0)$, with $\kappa_{\gcm}(\alpha_i,\alpha_j)=a_{ij}$, $\xi(i,j)=0$, is a 
	Cartan semigroup and the assignment 
	\begin{align}
		(\xp{\ia_i}, \xz{\ia_i}, \xm{\ia_i})\mapsto(e_i,h_i,f_i)
	\end{align} 
	defines a Lie algebra isomorphism $\g(\csgp)\simeq\g(\gcm)'$.
	
	Symmetric Borcherds--Kac--Moody algebras can also be described in terms of a more 
	interesting semigroup structure. More precisely, we show in Proposition~\ref{prop:bkm-sgp} and Theorem~\ref{thm:main}
	that Borcherds--Kac--Moody algebras corresponding to quivers 
	with at most one loop on each vertex and at most two arrows between any two vertices, can 
	be realised as Lie subalgebras of semigroup Lie algebras $\g(\sgp)$ for some non--trivial
	semigroups of topological origin.
	
	\bigskip\section{Semigroup Serre relations}\label{s:sgp-serre-relations}
	
	In this section, we study the necessary and sufficient conditions for the occurrence of
	distinguished quadratic relations in a semigroup Lie algebra $\g(\csgp)$. 
	Such relations, which we refer to as \emph{semigroup Serre relations}, are clearly inspired 
	by the case of $\sl(\R)$ as described in Corollary~\ref{cor:slR-presentation}. 
	
	\subsection{Serre relations}\label{ss:sgp-serre-rel}
	
	In analogy with the case of Kac--Moody algebras and $\sl(\R)$, it is desirable to have in $\g(\csgp)$ 
	certain \textit{quadratic Serre relations} of the form
	\begin{align}
		[\xpm{\ia}, \xpm{\ib}]=\mu_{\pm}(\alpha,\beta)\cdot \xpm{\ia\sgpp\ib}\ ,
	\end{align}
	for some $\mu_{\pm}\colon\sgp\times\sgp\to\bsfld$. The next result describes the necessary and
	sufficient conditions for such relations to hold. To this end, we shall define recursively the set $\sgp_{\leqslant\alpha}\subseteq\sgp$ 
	of \textit{partitions of $\alpha\in\sgp$} as follows. We set 
	\begin{align}
		\sgp_{\leqslant\alpha}^{(0)} & \coloneqq \{\alpha\} \ ,\\[2pt]
		\sgp_{\leqslant\alpha}^{(n)} &\coloneqq  
		\{\beta\sgpm\gamma\;\vert \;\gamma\in\sgp,\beta\in\sgp_{\leqslant\alpha}^{(n-1)}\}\quad\text{ for $n\geqslant 1$}\ ,
	\end{align}
	and  
	$\displaystyle
	\sgp_{\leqslant\alpha}\coloneqq \bigcup_{n\geqslant0}\sgp_{\leqslant\alpha}^{(n)}$. More precisely, $\alpha'\in \sgp_{\leqslant\alpha}$ if and only if there exist a sequence 
	\begin{align}
		\alpha=\alpha_1, \alpha_2, \ldots,  \alpha_{n}, \alpha_{n+1}=\alpha'
	\end{align} 
	such that $(\alpha_i, \alpha_{i+1})\in \dsgp_{\sgpm}$ for any $1\leqslant i\leqslant n$, so that
	\begin{align}\label{eq:alpha-partition}
		\alpha=(\beta_1\sgpp(\beta_2\sgpp\cdots\sgpp(\beta_n\sgpp\alpha')\cdots)
	\end{align} 
	where $\beta_i=\alpha_i\sgpm\alpha_{i+1}$.
	We shall call such a sequence an \textit{partition of $\alpha$ at $\alpha'$} and we write $\alpha'\leqslant\alpha$ if $\alpha'\in\sgp_{\leqslant\alpha}$. Finally, we denote by $\csgp_{\leqslant\alpha}^{\pm}$ the subset of all elements $\alpha'$ in  $\sgp_{\leqslant\alpha}$ for which there exists a partition $\alpha=\alpha_1, \alpha_2,\ldots, \alpha_{n}, \alpha_{n+1}=\alpha'$ such that, if $n>1$ we have 
	$\xi_{\pm}(\alpha_i, \alpha_{i+1})\neq 0$ for any $1\leqslant i\leqslant n$. 
	
	\begin{proposition}\label{pr:sgp-serre-general}
		Let $\mu_{\pm}\colon\sgp\times\sgp\to\bsfld$ be two functions and $\alpha,\beta\in\sgp$, $\alpha\neq\beta$.
		Then,
		\begin{align}\label{eq:sgp-serre}
			[\xpm{\ia}, \xpm{\ib}]=\mu_{\pm}(\alpha,\beta)\, \xpm{\ia\sgpp\ib}
		\end{align}
		holds in $\g(\csgp)$ if and only if the following relations hold.
		\begin{enumerate}\itemsep0.2cm
			\item For any $\oa\in \csgp_{\leqslant\alpha}^{\pm}$, $\ob\in \csgp_{\leqslant\beta}^{\pm}$,
			\begin{align}
				\label{eq:sgp-serre-general-1} 
				\xi_{\pm}(\oa\sgpp\ob,\oa)=\mu_{\mp}(\oa,\ob)=-\xi_{\pm}(\oa\sgpp\ob,\ob) \ ,
			\end{align}
			
			\begin{align}
				\label{eq:sgp-serre-general-2} 
				\kappa(\oa,\ob)=&
				\xi_+(\oa\sgpp\ob,\oa)\xi_-(\oa\sgpp\ob,\oa)
				-(\drc{\oa\sgpm\ob}+\drc{\ob\sgpm\oa})\xi_{+}(\ob,\oa)\xi_-(\ob,\oa) 
				=\kappa(\ob,\oa) \ .
			\end{align}
			\item For any $\oa\in \csgp_{\leqslant\alpha}^{\pm}$, $\ob\in \csgp_{\leqslant\beta}^{\pm}$, $\oc\in\sgp$,
			\begin{multline}\label{eq:sgp-serre-general-3}
				\drc{\oc\sgpm(\oa\sgpp\ob)}
				\xi_{\pm}(\oa\sgpp\ob,\oa)\xi_{\pm}(\oc,\oa\sgpp\ob)=\\[3pt]
				=\drc{(\oc\sgpm\oa)\sgpm\ob}\xi_{\pm}(\oc,\oa)\xi_{\pm}(\oc\sgpm\oa,\ob)
				-\drc{(\oc\sgpm\ob)\sgpm\oa}\xi_{\pm}(\oc,\ob)\xi_{\pm}(\oc\sgpm\ob,\oa) \ .
			\end{multline}
			\item For any $\oa\in \csgp_{\leqslant\alpha}^{\pm}$, $\ob\in \csgp_{\leqslant\beta}^{\pm}$, 
			$\oc\in\sgp$, $\oc\neq\oa,\ob$,
			\begin{multline}\label{eq:sgp-serre-general-4}
				\xi_{\pm}(\oa,\oc)\xi_{\mp}((\oa\sgpm\oc)\sgpp\ob,\ob)
				-\xi_{\pm}(\ob,\oc)\xi_{\mp}(\oa\sgpp(\ob\sgpm\oc),\oa)=\\[3pt]
				\drc{\ob\sgpm(\oc\sgpm\oa)}\xi_{\mp}(\oc,\oa)\xi_{\pm}(\ob,\oc\sgpm\oa)
				-\drc{\oa\sgpm(\oc\sgpm\ob)}\xi_{\mp}(\oc,\ob)\xi_{\pm}(\oa,\oc\sgpm\ob)\\[3pt]
				\phantom{\drc{\ob\sgpm(\oc\sgpm\oa)}\xi_{\mp}(\oc,\oa)\xi_{\pm}(\ob\sgpm\oc,\oa)\qquad}
				-\drc{(\oa\sgpp\ob)\sgpm\oc}\xi_{\mp}(\oa\sgpp\ob,\oa)\xi_{\pm}(\oa\sgpp\ob,\oc)
				\ .
			\end{multline}
		\end{enumerate}
	\end{proposition}
	
	\begin{proof}
		The result follows as a straightforward application of Lemma~\ref{lem:vanishing}.
		Namely, let $X_{\pm,\, \alpha, \, \beta}$ be the element defined by the  equation \eqref{eq:sgp-serre} for the pair 
		$(\alpha,\beta)$. We shall prove that $X_{\pm,\, \alpha, \, \beta}$ generates an ideal in $\wt{\n}_{\pm}$, 
		having trivial intersection with ${L}_0^{\sgp}$. Then, $X_{\pm,\, \alpha, \, \beta}=0$ in $\g(\csgp)$.
		By direct inspection, one sees easily that this holds if and only if the elements $X_{\pm,\, \oa,\, \ob}$, for any 
		$(\oa,\ob)\in\csgp_{\leqslant\alpha}^{\pm}\times\csgp_{\leqslant\beta}^{\pm}$, are also added to the generating set
		of the ideal and the functions $\kappa$ and $\xi_{\pm}$ satisfy the relations listed above. In Appendix~\ref{app:sgp-serre-general},
		we carry out this computation in full details.
		
	\end{proof}

	In Sections~\ref{ss:perp} and \ref{ss:deg} below, we study two special cases 
	for which the conditions \eqref{eq:sgp-serre-general-1}--\eqref{eq:sgp-serre-general-4} 
	are particularly simple to describe.
	
	\subsection{Orthogonality}\label{ss:perp}
	
	Recall that, in a Kac--Moody algebra, generators corresponding to \textit{orthogonal vertices} 
	in the Dynkin diagram (\ie $a_{ij}=0$) commute. In the case of \textit{trivial} Cartan semigroup 
	(\ie with $\xi_{\pm}=0$), Proposition~\ref{pr:sgp-serre-general} reduces precisely to an analogue 
	of the Serre relations for $a_{ij}=0$ in a Kac--Moody algebra.
	
	\begin{corollary}\label{cor:deg-case-km}
		If the functions $\xi_{\pm}$ are trivial, then 
		\begin{align}
			[\xpm{\ia},\xpm{\ib}]=0 \qquad\text{if and only if}\qquad\kappa(\alpha,\beta)=0=\kappa(\beta,\alpha)\ ,
		\end{align}
		where $\alpha,\beta\in\sgp$.
	\end{corollary}
	
	More in general, Corollary~\ref{cor:deg-case-km} suggests to introduce a suitable notion of
	orthogonal subsemigroups, so that the corresponding generators automatically commute.
	
	\begin{definition}\label{def:perp}
		Let $\sgp',\sgp''\subseteq\sgp$ be two saturated\footnote{The notion of saturated sub--semigroup is given in Definition~\ref{def:saturated}.} sub--semigroups. We say that $\sgp'$ and $\sgp''$ are \textit{orthogonal},
		and we write $\sgp'\perp\sgp''$, if they satisfy the following property: for any $\oa\in\sgp'$,
		$\ob\in\sgp''$ and $\oc\in\sgp$,
		\begin{enumerate}\itemsep0.2cm
			\item \label{def:perp-(1)} $\oa\sgpp\ob$ is not defined;
			\item \label{def:perp-(2)} $\kappa(\oa,\ob)=0=\kappa(\ob,\oa)$;
			\item \label{def:perp-(3)} $\xi_{\pm}(\oc,\ob)=\xi_{\pm}(\oc\sgpm\oa,\ob)$ and $\xi_{\pm}(\ob,\oc)=\xi_{\pm}(\ob,\oc\sgpm\oa)$ whenever all terms are defined. 
		\end{enumerate}
		If $\sgp'\perp\sgp''$, $\oa\in\sgp'$, and $\ob\in\sgp''$, we also write $\oa\perp\ob$.
	\end{definition}
	
	\begin{remark}\label{rem:perp}
		\hfill
		\begin{enumerate}[label=(\roman*), leftmargin=1.8em]\itemsep=0.2cm
			\item \label{rem:perp-(i)} Note that if $\oa\perp\ob$, then the elements $\oa\sgpm\ob$, $\ob\sgpm\oa$ 
			are never defined. Indeed, since $\sgp'$ is saturated and $\oa\in\sgp'$, then $\oa\sgpm\ob$ should belong to $\sgp'$.
			Therefore the sum $(\oa\sgpm\ob)\sgpp\ob=\oa$ would be defined, contradicting
			condition \eqref{def:perp-(1)} above.
			
			\item \label{rem:perp-(ii)} Note also that, by equation \eqref{eq:sgp-datum-1}, $\kappa$ already satisfies the analogue of condition \eqref{def:perp-(3)} above. 
			Namely, if $\oa\perp\ob$, then it follows from condition \eqref{def:perp-(2)} above that
			\begin{align}
				\kappa(\oc,\oa)=\kappa((\oc\sgpm\ob)\sgpp\ob, \oa)=
				\kappa(\oc\sgpm\ob,\oa)+\kappa(\ob,\oa)=\kappa(\oc\sgpm\ob,\oa)\ ,\\
				\kappa(\oa,\oc)=\kappa(\oa, (\oc\sgpm\ob)\sgpp\ob)=
				\kappa(\oa, \oc\sgpm\ob)+\kappa(\oa, \ob)=\kappa(\oa, \oc\sgpm\ob)\ ,
			\end{align}
			whenever $\oc\sgpm\ob$ is defined.
		\end{enumerate}
	\end{remark}
	
	This notion of orthogonality produces the desired outcome.
	
	\begin{corollary}\label{cor:perp-comm}	
		Let $\alpha,\beta\in\sgp$ be such that $\alpha\perp\beta$. Then $[\xpm{\ia},\xpm{\ib}]=0$.
	\end{corollary}
	
	\begin{proof}
		Let $\sgp', \sgp''$ be the smallest orthogonal saturated sub--semigroups containing $\alpha,\beta$, respectively.
		Note that $\csgp_{\leqslant\alpha}^{\pm}\subseteq\sgp'$ and $\csgp_{\leqslant\beta}^{\pm}\subseteq\sgp''$.
		Thus, \eqref{eq:sgp-serre-general-1} and \eqref{eq:sgp-serre-general-2} are automatically satisfied in view of
		Remark~\ref{rem:perp}--\ref{rem:perp-(i)} and Definition~\ref{def:perp}--\eqref{def:perp-(1)} and \eqref{def:perp-(2)}. The condition \eqref{eq:sgp-serre-general-3} follows immediately from 
		Definition~\ref{def:perp}--\eqref{def:perp-(3)}. Finally,  \eqref{eq:sgp-serre-general-4} follows from Remark~\ref{rem:perp}--\ref{rem:perp-(i)}.
	\end{proof}
	
	\subsection{Degenerate elements and Serre relations}\label{ss:deg}
	
	Another particularly simple case is given by elements which, roughly, do not interact
	with any other element. 
	
	\begin{definition}
		We say that $\alpha\in\sgp$ is 
		\begin{itemize}\itemsep0.3cm
			\item  \textit{degenerate} if, for any $\beta\in\sgp$, one of the following holds:
			
			\begin{enumerate}\itemsep0.2cm
				\item $\alpha\sgpp\beta$, $\alpha\sgpm\beta$, $\beta\sgpm\alpha$ are not defined;
				\item if $(1)$ does not hold, then $\xi_{\pm}(\alpha,\beta)$, $\xi_{\pm}(\beta,\alpha)$, 
				$\xi_{\pm}(\alpha\sgpp\beta,\beta)$, $\xi_{\pm}(\beta,\alpha\sgpp\beta)$ vanish.
			\end{enumerate}
			\item \textit{locally} degenerate if there exists a saturated sub--semigroup $\sgp'$ such that $\ia$ is degenerate in $\sgp'$.
			
		\end{itemize}
		We denote by $\degp(\sgp)$ (resp.\  $\ldegp(\sgp)$) the subset of  degenerate (resp.\ locally degenerate) elements in $\sgp$.
	\end{definition}
	\begin{remark}\label{rem:degenerate}
		Let $\ia\in\ldegp(\sgp)$ and $\sgp'$ a saturated sub--semigroup such that $\alpha\in \degp(\sgp')$. 
		Then, $\sgp_{\leqslant\ia}\subseteq \sgp'$ and $\csgp_{\leqslant\ia}^\pm=\{\ia\}$.
	\end{remark}
	
	The generators $\xpm{\ia}$, for $\ia\in\degp(\sgp)$, satisfy simple commutation relations in $\g(\csgp)$.
	\begin{proposition}\label{prop:deg-elem-rel}
		\hfill
		\begin{enumerate}\itemsep0.2cm
			\item \label{prop:deg-elem-rel-(1)} Let $\alpha\in\degp(\sgp)$ and $\beta\in\sgp$, $\beta\neq\alpha$. Then 
			\begin{align}
				[\xpm{\ia}, \xpm{\ib}]=0
			\end{align}
			if and only if $\kappa(\alpha,\gamma)$, $\kappa(\gamma,\alpha)$ vanish for any $\gamma\in\csgp_{\leqslant\beta}^{\pm}$.
			\item \label{prop:deg-elem-rel-(2)}
			For any $\alpha,\beta\in\degp(\sgp)$, $\alpha\neq\beta$, 
			such that $\kappa(\alpha,\alpha)=2$, $\kappa(\alpha,\beta)
			\in\Z_{\leqslant0}$, and $\kappa(\alpha,\beta)=0$ if and only if $\kappa(\beta,\alpha)=0$.
			Then, 
			\begin{align}
				\sfad(\xpm{\ia})^{1-\kappa(\alpha,\beta)}(\xpm{\ib})=0
			\end{align}
		\end{enumerate}
	\end{proposition}
	
	\begin{proof}
		\eqref{prop:deg-elem-rel-(1)} We proceed as in Proposition~\ref{pr:sgp-serre-general}.
		Let $S_{\pm}^{\alpha\beta}$ be the subspace spanned by $[\xpm{\ia}, \xpm{\ib}]$ 
		and all elements of the form $[\xpm{\ia}, \xpm{\ib\sgpm\ic}]$ for some $\gamma\in\sgp$.
		Then, one readily sees that, since $\alpha\in\degp(\sgp)$, 
		the identities \eqref{eq:sgp-serre-h}, \eqref{eq:sgp-serre-f} are trivial
		and, since $\kappa(\alpha,\beta)=0=\kappa(\beta,\alpha)$, 
		\eqref{eq:sgp-serre-e} reduces to $\xi_{\pm}(\beta,\gamma)[\xpm{\ia}, 
		\xpm{\ib\sgpm\ic}]$. The result follows as usual from Lemma~\ref{lem:vanishing}.
		
		\eqref{prop:deg-elem-rel-(2)} The proof follows closely \cite[Section 3.3]{kac-90}. We first observe that 
		$(\xp{\ia}, \xz{\ia}, \xm{\ia})$ is an $\sl(2)$--triple, since $\kappa(\alpha,\alpha)=2$,
		which acts on $\g(\csgp)$ by restriction. 
		
		Set $v_k=(\xm{\ia})^k\cdot \xm{\ib}$, $k\geqslant 0$. 
		Then, $\xz{\ia}\cdot v_k=(-\kappa(\alpha,\beta)-2k)v_k$, and
		\begin{align}\label{eq:sl(2)-rep}
			\xp{\ia}\cdot v_k=k(-\kappa(\alpha,\beta)-k+1)v_k\ .
		\end{align} 
		The last relation follows from the fact that $\alpha\in\degp(\sgp)$ and therefore 
		$[\xp{\ia}, \xm{\ib}]=0$. 
		
		Set $\theta_{\alpha\beta}=(\xm{\ia})^{1-\kappa(\alpha,\beta)}\cdot \xm{\ib}$.
		By \eqref{eq:sl(2)-rep}, $[\xp{\ia}, \theta_{\alpha\beta}]=0$. 
		Note that, for any $k\geqslant 1$, one has
		\begin{align}\label{eq:sl(2)-aux}
			\begin{aligned}
				[\xp{\ib}, &(\xm{\ia})^k\cdot \xm{\ib}]=
				\xm{\ia}\cdot [\xp{\ib}, (\xm{\ia})^{k-1}\cdot \xm{\ib}]=\\
				=&(\xm{\ia})^k\cdot [\xp{\ib}, \xm{\ib}]=
				(\xm{\ia})^k\cdot \xz{\ib}=\kappa(\beta,\alpha)(\xm{\ia})^{k-1}\cdot \xm{\ia}
				\ .
			\end{aligned}
		\end{align}
		For $k=1-\kappa(\alpha,\beta)$, this implies $[\xp{\ib}, \theta_{\alpha\beta}]=0$,
		since if $\kappa(\alpha,\beta)=0$ then $\kappa(\beta,\alpha)=0$ and
		if $\kappa(\alpha,\beta)\neq0$ then $k>1$.
		
		Finally, since $\alpha,\beta\in\degp(\sgp)$, for any $\gamma\neq\alpha,\beta$, one has
		$[\xp{\ic}, \xm{\ia}]=0=[\xp{\ic}, \xm{\ib}]$ and therefore
		$[\xp{\ic},\theta_{\alpha\beta}]=0$. By Lemma~\ref{lem:vanishing}, $\theta_{\alpha\beta}$=0.
	\end{proof}
	
	Proposition~\ref{prop:deg-elem-rel}--\eqref{prop:deg-elem-rel-(1)} 
	is easily generalized at the level of saturated sub--semigroups and we get the following
	corollary. 
	
	\begin{corollary}\label{cor:commutation-degenerate}
		Let $\alpha$ be degenerate in a saturated sub--semigroup $\sgp'$, for which $\kappa(\alpha, \gamma)=0=\kappa(\gamma,\alpha)$ for any $\gamma\in \sgp'$. 
		Then  $[\xpm{\ia},\xpm{\ib}]=0$ for any $\beta\in\sgp_{\leqslant \alpha}$.
	\end{corollary}

	\subsection{Total semigroups}\label{s:sgp-example}
	We conclude this section by showing that the situation becomes particularly simple
	for an \emph{actual} semigroup with operations which are always defined.
	Let $\csgp=(\sgp,\kappa,\xi_{\pm})$ be a Cartan semigroup satisfying the following additional 
	conditions:
	\begin{enumerate}\itemsep0.2cm
		\item \label{condition-(1)} $\csgp$ is \emph{total} with respect to $\sgpp$, \ie $\dsgp_{\sgpp}=\sgp\times\sgp$;
		\item \label{condition-(2)} 
		$\csgp$ is \emph{total} with respect to $\sgpm$, \ie
		for any $\ia, \ib\in \sgp$, $\ia\neq \ib$, either $(\ia,\ib)$  or $(\ib, \ia)$ is in $\dsgp_{\sgpm}$;
		\item \label{condition-(3)} $\kappa =0$;
		\item \label{condition-(4)} $\xi_{+}=\xi_{-}\eqqcolon \xi$ is antisymmetric and $\sgp$-linear. 
	\end{enumerate}
	In this case, the quadratic Serre relations hold globally for every pair of elements in $\sgp$.
	
	\begin{proposition}\label{pr:sgp-serre-example} 
		For any $\ia, \ib\in \sgp$, it holds
		$[\xpm{\ia}, \xpm{\ib}]=\xi(\ib, \ia)\cdot \xpm{\ia\sgpp\ib}\, $
		in $\g(\csgp)$.
	\end{proposition}
	
	\begin{proof}
		The result follows from a straightforward application of Proposition~\ref{pr:sgp-serre-general},
		based on a case--by--case analysis. 
		Note that there is a standard strict partial order $<$ on $\sgp$ given by
		\begin{align}
			\ia < \ib\quad \text{if and only if} \quad \ib \sgpm \ia \in \sgp\ ,
		\end{align}
		for $\ia, \ib\in \sgp$. First, note that equation \eqref{eq:sgp-serre-general-1} follows from the antisymmetry of $\xi$, while, because of condition \eqref{condition-(2)}, equation \eqref{eq:sgp-serre-general-2} holds if and only if $\kappa= 0$. Moreover, equation \eqref{eq:sgp-serre-general-3} is either trivial or all the terms are defined since:
		\begin{align}
			\oa\sgpp \ob < \oc \quad \Leftrightarrow \quad \oa <\oc \ \text{ and } \ \ob< \oc \sgpm \oa \quad \Leftrightarrow \quad \ob <\oc \ \text{ and } \ \oa< \oc \sgpm \ob \ .
		\end{align}
		In the latter case, a straightforward check show that equation \eqref{eq:sgp-serre-general-3} holds. 
		Finally, we have to check that relation~\eqref{eq:sgp-serre-general-4} holds in the following mutually exclusive cases:
		
		\begin{itemize}\itemsep0.2cm
			\item $\oc<\oa\sgpp \ob$, $\oc<\oa$, and $\oc<\ob$; 
			
			\item $\oc<\oa\sgpp \ob$, $\oc<\oa$, and $\ob<\oc$; 
			
			\item $\oc<\oa\sgpp \ob$, $\oa<\oc$, and $\ob<\oc$; 
			
			\item $\oc<\oa\sgpp \ob$, $\oc<\ob$, and $\oa<\oc$; 
			
			\item $\oa\sgpp \ob<\oc$. 
		\end{itemize}
		Note that in the last case, both sides of \eqref{eq:sgp-serre-general-4} are zero.
		The other cases are proved with a simple direct inspection.
	\end{proof}	
	
	\bigskip\section{Good  Cartan semigroups and Serre relations}\label{s:sgp-serre}
	
	In this section, we introduce the notion of \textit{good} Cartan semigroup. The
	definition is tailored around the conditions highlighted by Proposition~\ref{pr:sgp-serre-general}
	and aims to provide a list of simple properties a Cartan semigroup $\csgp$
	should satisfy so to induce Serre relations in the Lie algebra $\g(\csgp)$.
	
	\subsection{Good  Cartan semigroups}
	
	Let $\csgp=(\sgp,\kappa,\xi_{\pm})$ be a Cartan semigroup.
	Locally degenerate elements are semigroup analogues of imaginary roots. 
	This leads to the following definition.
	
	\begin{definition}
		We say that $\alpha\in\sgp$ is \textit{imaginary} if there exists $\alpha'\in \sgp_{\leqslant \alpha}$ which is locally degenerate; while $\alpha$ is \textit{real} if it is not imaginary. We denote by $\isgp$ (resp.\ $\rsgp$) the set of  imaginary (resp.\ real) elements of $\sgp$.
	\end{definition}
	
	\begin{remark}
		Note that, if $\alpha\in\rsgp$, then $\sgp_{\leqslant\alpha}\subseteq\rsgp$.
	\end{remark}
	
	The following notion of good Cartan semigroup lists five crucial properties concerning
	the basic semigroup operations, the real elements and the functions $\kappa$ and $\xi_{\pm}$.
	
	\begin{definition}\label{def:sym-sgp}
		A \textit{Cartan semigroup} $\csgp$ is \textit{good}
		if the following conditions hold.
		\begin{enumerate}\itemsep0.2cm
			\item \label{def:sym-sgp-1} {\bf Multiplicity free.}  For any $\alpha,\beta\in\sgp$, at most one between the elements 
			$\alpha\sgpp\beta$, $\alpha\sgpm\beta$, and $\beta\sgpm\alpha$ is defined ($\dsgp_{\sgpp}\cap\dsgp_{\sgpm}=\emptyset$).
			\item \label{def:sym-sgp-2} {\bf Locality.} The following holds:
			\begin{itemize}\itemsep0.2cm
				\item[(L1)] if $\alpha\not\perp\beta$, then $(\gamma\sgpm\alpha)\sgpm\beta$ is defined only if $\alpha\sgpp\beta$ is defined;
				\item[(L2)] if $(\alpha\sgpp\beta)\sgpm\gamma$ is defined and $\alpha\perp\gamma$, then $\beta\sgpm\gamma$ is defined.
			\end{itemize}
			\item \label{def:sym-sgp-3} {\bf Real elements.} If $\alpha\in\rsgp$ and $\gamma\sgpm\alpha$ is defined for some $\gamma\not\in\rsgp$,
			then there exists $\gamma'\in\rsgp$ such that $\gamma\sgpm\gamma'$ is defined and  $\alpha\perp(\gamma\sgpm\gamma')$.
			
			\item \label{def:sym-sgp-4} $\xi_+(\alpha,\beta)=\xi_-(\beta,\alpha)$ and satisfies the following properties ($\xi\coloneqq\xi_+$):
			\begin{itemize}\itemsep0.2cm
				\item for any $\alpha,\beta\in\sgp$,
				\begin{align}
					\xi(\alpha\sgpp\beta,\alpha)=-\xi(\alpha\sgpp\beta,\beta) \ ,\label{eq:sym-sgp-datum-3}\\
					\xi(\alpha,\alpha\sgpp\beta)=-\xi(\beta,\alpha\sgpp\beta) \ ;\label{eq:sym-sgp-datum-4}
				\end{align}
				\item for any $\alpha,\beta,\gamma\in\sgp$ such that 
				\begin{itemize}
					\item $\alpha\sgpm\beta$, $\alpha\sgpp\gamma$ and $\beta\sgpp\gamma$ are defined;
					\item if $\alpha\sgpp \gamma$ is imaginary, then $\beta\sgpp \gamma$ is not a partition of the locally degenerate part of $\alpha\sgpp \gamma$;
					\footnote{Note that this condition implies $\alpha\sgpp \gamma, \beta\sgpp \gamma\notin \ldegp(\sgp)$.}
				\end{itemize}
				one has
				\begin{align}\label{eq:sym-sgp-datum-5}
					\xi(\alpha\sgpp\gamma,\beta\sgpp\gamma)=
					\xi(\alpha,\beta)\ .
				\end{align}
			\end{itemize}
			\item \label{def:sym-sgp-5} $\kappa$ is symmetric (\ie $\kappa(\alpha,\beta)=\kappa(\beta,\alpha)$) and satisfies the following properties:
			\begin{itemize}\itemsep0.2cm
				\item for any $\alpha,\beta,\gamma\in\sgp$,
				\begin{align}\label{eq:sym-sgp-datum-1}
					\kappa(\alpha\sgpp\beta,\gamma)&=\drc{\alpha\sgpp\beta}\left(\kappa(\alpha,\gamma)+\kappa(\beta,
					\gamma)\right) \ ;
				\end{align}
				\item for any $\alpha\in\rsgp$ and $\beta\in\sgp$,
				\begin{align}\label{eq:sym-sgp-datum-2}
					\kappa(\alpha,\beta)=
					\begin{cases}
						\xi(\alpha\sgpp\beta,\alpha)\xi(\alpha,\alpha\sgpp\beta) & \text{if } (\alpha,\beta)\in\dsgp_{\sgpp} \text{ and $(\ast)$ holds}\ ,\\[4pt]
						\phantom{a}-\xi(\alpha,\beta)\xi(\beta,\alpha) & \text{if } (\alpha,\beta)\in\dsgp_{\sgpm} \text{ or }  
						(\beta,\alpha)\in\dsgp_{\sgpm} \ ,\\[4pt]
						\phantom{\xi(\alpha\sgpp\beta,\alpha}0 & \text{if } \alpha\sgpp\beta,\alpha\sgpm\beta, \text{ and }\beta\sgpm\alpha \text{ are not defined}\ ,
					\end{cases}
				\end{align}
				where $(\ast)$ is the following condition: if $\beta\in\rsgp$ then $\alpha\sgpp \beta\in \rsgp$; if $\beta\in \isgp$ then $\alpha\sgpp\beta\notin\ldegp(\sgp)$.
			\end{itemize}
		\end{enumerate}
	\end{definition}
	
	\begin{remark}\label{rem:simply-laced-sgp}
		Note that, by (L2) and (3), it follows that, if $\alpha$ is a real element in $\sgp_{\leqslant\gamma}^{(1)}$
		with $\gamma\not\in\rsgp$, then there exists a real element $\gamma'\in\sgp_{\leqslant\gamma}^{(1)}$ 
		such that $\alpha\in\sgp_{\leqslant\gamma'}^{(1)}$.
	\end{remark}
	
	\subsection{Admissible pairs and triples}
	
	We shall prove in Theorem~\ref{thm:sym-serre-rel} that in the case of a good Cartan semigroup 
	certain semigroup Serre relations occur for suitable pairs of elements in $\sgp$. 
	This leads to the following definitions of \emph{admissible pairs and triples}.
	
	\begin{definition}\label{def:admissible-triple}\label{def:admissible-pair}
		Let $\csgp$ be a good Cartan semigroup.
		\begin{enumerate}[leftmargin=2em]\itemsep0.15cm
			\item
			Let $(\ia, \ib, \ic)\in \sgp \times \sgp \times \sgp$. The triple $(\ia, \ib, \ic)$ is \textit{admissible} if $(\ia,\ib,\ic)\in\rsgp\times\rsgp\times\rsgp$ or $(\ia,\ib,\ic)\in\rsgp\times\isgp\times\isgp$, with $\ia\not\perp\ib$, then
			\begin{enumerate}[leftmargin=2em]\itemsep0.15cm
				\item \label{item:A1} either none or exactly two elements 
				among $\ic\sgpm(\ia\sgpp\ib)$, $(\ic\sgpm\ia)\sgpm\ib$, and $(\ic\sgpm\ib)\sgpm\ia$ can be simultaneously defined;
				\item \label{item:A2} either none or exactly two elements among $(\ia\sgpp\ib)\sgpm\ic$, $(\ia\sgpm\ic)\sgpp\ib$, $\ia\sgpp(\ib\sgpm\ic)$,
				$\ia\sgpm(\ic\sgpm\ib)$, and $\ib\sgpm(\ic\sgpm\ia)$ are simultaneously defined.
			\end{enumerate}
			\item
			Let $(\ia, \ib)$ be an unordered pair of elements in $\sgp$ with $\alpha\in\rsgp$. We 
			say that $(\ia, \ib)$ is a \textit{admissible pair} if for any $\oa\in\csgp^{\pm}_{\leqslant\alpha}, \ob\in\csgp^{\pm}_{\leqslant\beta}$ we have that
			\begin{enumerate}[leftmargin=2em]\itemsep0.2cm
				\item for any $\oc\in \sgp$, the triple $(\oa, \ob, \oc)$ is admissible;
				\item either $(\oa,\ob)\not\in\dsgp_{\sgpp}$ or $\oa\sgpp\ob\notin\ldegp(\sgp)$.
			\end{enumerate}
			We denote by $\serreadm{\csgp}$ the set of admissible pairs.
		\end{enumerate}
	\end{definition}
	
	\begin{remark}\label{rem:serre-induction}
		Note that, if $(\alpha,\beta)\in\serreadm{\csgp}$, then $\csgp^{\pm}_{\leqslant\alpha}\times\csgp^{\pm}_{\leqslant\beta}\subseteq\serreadm{\csgp}$.
		Moreover, if $\alpha,\beta\in\rsgp$, the element $\alpha\sgpp\beta$ is necessarily real, whenever defined. Finally, if $(\alpha, \beta)\in \serreadm{\csgp}$, then $(\alpha, \alpha\sgpp \beta)\in \serreadm{\csgp}$ 
		whenever $\alpha\sgpp \beta$ is defined.
	\end{remark}
	
	\subsection{Serre relations}
	Let $\serre{\sgp}$ be the union of $\serreadm{\csgp}$ and the set of orthogonal pairs.
	
	\begin{theorem}\label{thm:sym-serre-rel}
		For any $(\alpha,\beta)\in\serre{\sgp}$,   
		$[\xpm{\ia}, \xpm{\ib}]=\xi_{\mp}(\alpha\sgpp\beta,\alpha)\cdot \xpm{\ia\sgpp\ib}$.
	\end{theorem}
	
	\begin{proof}
		It is clear that, if $\ia\perp\ib$, then Corollary~\ref{cor:perp-comm} implies the result.
		Therefore, we have to prove the Serre relations for any pair of elements in $\serreadm{\csgp}$.
		As before, we proceed by showing that for any such pair the conditions 
		\eqref{eq:sgp-serre-general-1}, 
		\eqref{eq:sgp-serre-general-2}, 
		\eqref{eq:sgp-serre-general-3}, and 
		\eqref{eq:sgp-serre-general-4}
		from Proposition~\ref{pr:sgp-serre-general} hold. 
		The first two conditions follow immediately from the properties 
		\eqref{eq:sym-sgp-datum-3}, \eqref{eq:sym-sgp-datum-4}, and \eqref{eq:sym-sgp-datum-2}
		of good Cartan semigroups. This simple observation can also be thought of as the main 
		motivation behind these properties. The proof of the remaining two conditions is longer 
		and more technical, but it can be easily carried out on a case--by--case analysis, 
		whose detailed computations motivate the other properties featured 
		in Definition~\ref{def:sym-sgp}. The full proof is given in Appendix~\ref{app:sym-serre-rel}.
	\end{proof}
	
	From Corollary~\ref{cor:perp-comm}, Remark~\ref{rem:serre-induction} and Theorem~\ref{thm:sym-serre-rel}, we obtain the following result, which can be thought of
	as a semigroup analogue of the Serre relations for Borcherds-Kac-Moody algebras 
	in the case $a_{ij}=-1,0$ respectively (cf.\ Equations~\eqref{eq:Serrerel-BKM-1} 
	and \eqref{eq:Serrerel-BKM-2}). 
	\begin{corollary}\label{cor:serre-rel}
		Let $\ia,\ib \in \sgp$. 
		\begin{enumerate}\itemsep0.2cm
			\item If $(\ia, \ib)\in \serreadm{\csgp}$, then $[\xpm{\ia}, [\xpm{\ia}, \xpm{\ib}]]=0$.
			\item If $\ia\perp \ib$, then $[\xpm{\ia}, \xpm{\ib}]=0$.
		\end{enumerate}
	\end{corollary}
	
	\subsection{Locally nilpotent adjoint actions}
	
	We conclude this section by proving that in the case of a good 
	Cartan semigroup the generators $\xpm{\ia}$ associated to real 
	elements  act locally nilpotently on $\g(\csgp,\kappa, \xi)$. 
	
	\begin{proposition}\label{prop:locallynilpotent}
		Let $\alpha\in \rsgp$, $\beta\in \sgp$ be such that $(\alpha, \beta)\notin \serreadm{\csgp}$, $\alpha\not\perp \beta$, and there exists a sequence of elements $\beta_i$, $i=1,\dots, k$, such that
		\begin{enumerate}
			\item $\beta=(\cdots (\beta_1\sgpp \beta_2)\sgpp \cdots \sgpp \beta_k)$;
			\item $\big((\cdots (\beta_1\sgpp \beta_2)\sgpp \cdots \sgpp \beta_i), \beta_{i+1}\big)\in \serreadm{\csgp}$ for any $i=1, \ldots, k-1$;
			\item either $\alpha \perp \beta_i$ or $(\alpha, \beta_i)\in \serreadm{\csgp}$ for any $i$.
		\end{enumerate}
		Then  $\sfad(\xpm{\ia})^{k+1}(\xpm{\ib})=0$.
	\end{proposition}
	
	\begin{proof}
		By conditions (1) and (2), we can iteratively apply Theorem~\ref{thm:sym-serre-rel} to the 
		elements $\beta_1,\dots,\beta_k$, and get
		\begin{align}
			\xpm{\ib}=c\cdot[\xpm{\ib\sgpm\ib_k}, \xpm{\ib_k}]=c\cdot[\cdots[\xpm{\ib_1},\xpm{\ib_2}]\cdots],\xpm{\ib_{k-1}}],\xpm{\ib_k}]
		\end{align} 
		for some constant $c$. 
		Therefore,
		\begin{align}
			\sfad(\xpm{\ia})^{k+1}(\xpm{\ib}) =c\cdot\sfad(\xpm{\ia})^{k+1}([\xpm{\ib\sgpm\ib_k}, \xpm{\ib_k}])
			=c\cdot\sum_{j=1}^{k+1}\left[\sfad(\xpm{\ia})^{k+1-j}(\xpm{\ib\sgpm\ib_k}),\sfad(\xpm{\ia})^{j}(\xpm{\ib_k})\right]
		\end{align}
		From Corollary~\ref{cor:serre-rel}, $\sfad(\xpm{\ia})^{l}(\xpm{\ib_k})=0$ for $l\geqslant 2$. Thus,
		\begin{align}
			\sfad(\xpm{\ia})^{k+1}(\xpm{\ib}) = c\cdot[\sfad(\xpm{\ia})^{k+1}(\xpm{\ib\sgpm\ib_k}), \xpm{\ib_k}]  +c\cdot[\sfad(\xpm{\ia})^{k}(\xpm{\ib\sgpm\ib_k}), \sfad(\xpm{\ia})(\xpm{\ib_k})]\ .
		\end{align}
		Therefore, we are reduced to prove the statement for $k=2$. 
		In this case, we have 
		\begin{align}
			\sfad(\xpm{\ia})^{3}(\xpm{\ib}) 
			= c\cdot[\xpm{\beta_1}, \sfad(\xpm{\ia})^{3}(\xpm{\beta_2}))] 
			+c\cdot [\sfad(\xpm{\ia})(\xpm{\beta_1}), \sfad(\xpm{\ia})^{2}(\xpm{\beta_2})]
			=0
		\end{align}
		where the last equality follows directly from Corollary~\ref{cor:serre-rel}.
	\end{proof}
	
	\begin{corollary}
		Assume that every $\ib\in\sgp$ admits a decomposition as in Proposition~\ref{prop:locallynilpotent}.
		Then, every element $\xpm{\ia}$ with $\ia\in\rsgp$ acts locally nilpotently on 
		$\g(\csgp,\kappa, \xi)$.
	\end{corollary}
	
	\begin{proof}
		Note that, for any $\alpha,\beta\in \sgp$, 
		one has $\sfad(\xpm{\ia})^3(\xmp{\ia})=0$ and $\sfad(\xpm{\ia})^2(\xz{\ib})=0$. 
		Moreover, $\sfad(\xpm{\ia})^2(\xmp{\ib})=0$ for $\beta\neq\alpha$, since 
		the elements $\alpha\sgpm(\beta\sgpm\alpha)$ and $(\beta\sgpm\alpha)\sgpm\alpha$
		are never defined, therefore
		$\sfad(\xpm{\ia})(\xmp{\ib\sgpm\ia})=0
		=\sfad(\xpm{\ia})(\xpm{\ia\sgpm\ib})$. The result then follows from Proposition
		~\ref{prop:locallynilpotent}.
	\end{proof}
	
	\bigskip\section{Continuum Kac--Moody algebras}\label{s:topological-quiver}
	
	In this section, we introduce the notion of a \emph{continuum quiver} as a distinguished 
	semigroup associated to a one--dimensional real space. We prove that continuum 
	quivers are examples of good  Cartan semigroup. Relying on the results obtained in
	Section~\ref{s:sgp-serre}, the corresponding 
	semigroup Lie algebras, which we refer to as \textit{continuum Kac--Moody algebras}, 
	are explicitly presented in terms of generators and relations and proved to be isomorphic
	to a colimit of symmetric Borcherds--Kac--Moody algebras.
	
	\subsection{Vertex spaces}
	
	We introduce the class of topological spaces we are mainly interested in. These can 
	be thought of as \emph{smoothings} of one--dimensional real CW complexes.
	
	\begin{definition}
		Let $X$ be a Hausdorff topological space. We say that $X$ is a \textit{vertex space} 
		if for any $x\in X$, there exists a \textit{chart $(U, A, \phi)$ around $x$} such that
		\begin{enumerate}\itemsep0.15cm
			\item $U$ is an open neighborhood of $x$,
			\item $A=\{A_i\}$ is a family of closed subsets $A_i\subseteq U$ containing $x$, 
			such that $U=\bigcup_i\, A_i$, 
			\item $\phi=\{\phi_i\}$ is a family of continuous maps $\phi_i\colon A_i\to \R$ which 
			are homeomorphisms onto open intervals of $\R$, such that if the intersection between 
			$A_i$ and $A_j$ strictly contains the point $x$, then $\phi_i\vert_{A_i\cap A_j}=\phi_j
			\vert_{A_i\cap A_j}$ and $\phi_i\vert_{A_i\cap A_j}$ induces a homeomorphism between 
			$A_i\cap A_j$ and a closed interval of $\R$.
		\end{enumerate}
		We say that $x$ is 
		\begin{itemize}
			\item a \textit{regular point} if the exists a chart such that $A=\{U\}$, 
			\item a \textit{critical point} if there exists a chart such that the boundary $ \partial(A_i\cap A_j)$ of $A_i\cap A_j$, as a subset of $U$, contains $x$ for any $i,j$.
		\end{itemize}
	\end{definition}
	\begin{remark}
		Let $x$ be a critical point with a chart $(U, A, \phi)$ such that $x\in \partial(A_i\cap A_j)$ for any $i,j$. Then $x\in \partial A_i$ for any $i$.
	\end{remark}
	
	\begin{example}
		Simple examples of a vertex space, beyond $\R$ and $S^1$, are given below:
		\begin{align}
			\begin{tikzpicture}[scale=.35]
				\begin{scope}[on background layer]
					\draw [white] (0,0) rectangle (15,10);
				\end{scope}
				\coordinate (C2) at (6, 5);
				\coordinate (C3) at (9, 5);
				\coordinate (C4) at (11.25,6);
				\draw [dotted, thick] (C2) -- (4,5);
				\draw [-, thick]		 (C2) -- (C3);
				\draw [->, thick]      (C3) arc (270:360:3); 
				\draw [dotted, thick] (C3) arc (270:385:3);
				\draw [->, thick]      (C3) arc (90:0:3);
				\draw [dotted, thick] (C3) arc (90:-25:3);
				\draw [->, thick]      (C4) arc (120:75:3);
				\draw [dotted, thick] (C4) arc (120:45:3);
			\end{tikzpicture}
			\qquad
			\begin{tikzpicture}[scale=.35]
				\begin{scope}[on background layer]
					\draw [white] (0,0) rectangle (15,10);
				\end{scope}
				\draw [-,  thick] (7.5,7.5) arc (90:180:2.5);
				\draw [->,  thick] (5,5) arc (180:270:2.5);
				\draw [-,  thick] (7.5,2.5) arc (270:360:2.5);
				\draw [->,  thick] (10,5) arc (0:90:2.5);
				\draw [->,  thick] (5,5) arc (0:-90:2.5);
				\draw [dotted,  thick] (5,5) arc (0:-110:2.5);
				\draw [->,  thick] (10,5) arc (180:45:2.5);
				\draw [dotted,  thick] (10,5) arc (180:25:2.5);
				\draw [->, thick] (12.5, 7.5) arc (270:315:2.5);
				\draw [dotted, thick] (12.5, 7.5) arc (270:335:2.5);
				\draw [->, thick] (11.25, 7.15) arc (-45:15:2.5);
				\draw [dotted, thick] (11.25, 7.15) arc (-45:35:2.5);
			\end{tikzpicture}
		\end{align}
	\end{example}
	
	\subsection{Semigroup of intervals}\label{ss:topological-quiver-1}
	
	We lift the notion of open--closed interval on $\R$ to an arbitrary vertex space $X$.
	
	\begin{definition}\label{def:topological-quiver}
		Let $X$ be a vertex space and let $\ia$ be a subset of $X$. 
		We say that $\ia$ is an \textit{elementary interval} if there exists a chart $(U, A, \phi)$ for which $\ia\subset A_i$ for some $i$ and $\phi_i(\ia)$ is a open-closed interval of $\R$.
		A sequence of elementary intervals $(\ia_1,\dots, \ia_n)$, $n>0$, is \textit{admissible} if
		\begin{enumerate}[label=(\alph*)]\itemsep0.2cm
			\item \label{def:topological-quiver-a} $(\ia_1\cup\cdots \cup \ia_{i})\cap \ia_{i+1}=\emptyset$ and $(\ia_1\cup\cdots\cup \ia_{i})\cup \ia_{i+1}$ is connected for any $i=1,\dots, n-1$;
			\item \label{def:topological-quiver-b} for any $i=1,\dots, n-1$, there exist $x\in X$ and a chart $(U,A, \phi)$ around $x$ for which $U\supseteq(\ia_1\cup\cdots \cup \ia_{i})\cup \ia_{i+1}$ and 
			$\big((\ia_1\cup\cdots \cup \ia_{i})\cup \ia_{i+1}\big)\cap A_k$ is either empty or an elementary interval for any $k$. 
		\end{enumerate}
		An \textit{interval} of $X$ is a subset $\ia$ of the form $\ia_1\cup\cdots\cup \ia_n$, where $(\ia_1,\dots, \ia_n)$ is an admissible sequence of elementary intervals.
		We denote by $\intsf{X}{}$ the set of all intervals in $X$.
	\end{definition}
	
	The set of intervals $\intsf{X}{}$ carries a natural
	semigroup structure induced by the local structure on $\R$ described in
	Section~\ref{ss:semigroup-line}. For any $\ia,\ib\in\intsf{X}{}$, set
	\begin{align}
		\ia\sgpp \ib\coloneqq &
		\begin{cases}
			\ia\cup \ib & \text{if } \ia\cap \ib=\emptyset\text{ and }\ia\cup \ib\in\intsf{X}{}\ ,\\[4pt]
			\text{n.d.} & \text{otherwise}\ ,
		\end{cases}\\
		\ia\sgpm \ib\coloneqq &
		\begin{cases}
			\ia\setminus \ib & \text{if } \ia\cap \ib=\ib\text{ and }\ia\setminus \ib\in \intsf{X}{}\ ,\\[4pt]
			\text{n.d.} & \text{otherwise}\ .
		\end{cases}
	\end{align}
	We refer to $\sgpp$ and $\sgpm$ as the \textit{sum and difference of intervals}, respectively.
	The following is straightforward.
	\begin{lemma}\label{lem:int-X}\hfill
		\begin{enumerate}
			\item Every contractible interval is homeomorphic to a finite oriented tree such that any
			vertex is the target of at most one edge.
			\item Every non--contractible interval is homeomorphic to an interval of the form 
			\begin{align}
				S^1\sgpp \bigoplus_{k=1}^N T_k\coloneqq (\cdots(S^1\sgpp T_1)\sgpp T_2)\cdots \sgpp T_N)
			\end{align}
			for some pairwise disjoint contractible intervals $T_k$, with $N\geqslant 0$.
		\end{enumerate}
	\end{lemma}
	
	\begin{proof}
		We assume for simplicity that $X$ is connected. Note that if $X$ has no critical points, $X$ reduces to either $\R$ or $S^1$. In the former case, every interval is elementary, while in the latter the only non--contractible interval is $S^1$ itself. Assume that there exists at least one critical point in $X$. If $\ia$ does not contain any critical point, it is elementary. If $\ia$ contains a critical point and it is contractible, then it must be a finite sum of
		elementary intervals, hence it corresponds to oriented tree such that any vertex is the target of at most one edge. On the other hand, if $\ia$ is not contractible, by definition, it must be either homeomorphic to $S^1$ or a sum 
		of one copy of $S^1$ with at least one necessarily contractible interval. 
	\end{proof}

	\subsection{Euler forms}\label{ss:topological-quiver-2}
	We denote by $\fun{X}{}$ the $\Z$-span of the characteristic functions $\cf{\ia}$ for all interval $\ia$ of $X$. Note that $\cf{\ia\sgpp \ib}=\cf{\ia}+\cf{\ib}$ for a given $(\ia, \ib)\in \intsf{X}{}^{(2)}_{\sgpp}$. We call \textit{support} of a function $f\in \fun{X}{}$ the set $\mathsf{supp}(f)\coloneqq \{x\in X\, \vert\, f(x)\neq 0\}$. It is a disjoint union of finitely many intervals of $X$.
	
	Define a bilinear form $\abf{\cdot}{\cdot}$ on $\fun{X}{}$ in the following way. Let $f,g\in\fun{X}{}$, and assume that there exists a point $x$ with a chart $(U, A, \phi)$ for which the supports of $f$ and $g$ are contained in $A_i$ for some $i$, then we set 
	\begin{align}
		\abf{f}{g}\coloneqq \sum_{x\in A_i} f_-(x)(g_-(x)-g_+(x)) \ .
	\end{align}
	Since we can always decompose an interval into a sum of elementary subintervals (and we can do similarly with supports of functions of $\fun{X}{}$), we extend $\abf{\cdot}{\cdot}$ with respect to $\sgpp$ by imposing the condition that $\abfcf{\ia}{\ib}=0$ for two elementary intervals $\ia, \ib$ for which there does not exist a common $A_i$ containing both.
	
	As a consequence of the definition, the bilinear form $\abf{\cdot}{\cdot}$ is compatible with the concatenation of intervals, by Lemma~\ref{lem:int-X}, it is entirely determined by its values on contractible elements.
	
	\begin{remark}\label{rem:euler-form-identities-I}
		Thanks to Definition~\ref{def:topological-quiver}-\ref{def:topological-quiver-b}, one can easily verify that 
		\begin{itemize}\itemsep0.2cm
			\item if $\ib$ is a non--contractible sub--interval of $\ia$, then $\abfcf{\ia}{\ib}=\abfcf{\ia\sgpm \ib}{\ib}$
			whenever $\ia\sgpm \ib$ is defined;
			\item if $(\ia,\ib)\not\in\intsf{X}{}^{(2)}_{\sgpp}$ and $\ia\cap \ib=\emptyset$, then $\abfcf{\ia}{\ib}=0$.
		\end{itemize}
	\end{remark}
	
	\begin{remark}\label{rem:euler-form-identities-II}
		Let $\ia$ be a contractible interval given as $\ia=\ia_1\cup\cdots\cup \ia_n$ for some admissible sequence $(\ia_1,\dots, \ia_n)$ of elementary intervals. One has $\abfcf{\ia}{\ia}=1$, indeed if $n=1$, this follows from equation \eqref{eq:abf-line}. Assume that the result holds for all intervals given by admissible sequences consisting of $n-1$ elementary intervals. Let us prove for $n$: we have
		\begin{align}
			\abfcf{\ia}{\ia}= &\abf{\cf{\ia_1\cup \cdots\cup \ia_n}+\cf{\ia_{n+1}}}{\cf{\ia_1\cup \cdots\cup \ia_n}+\cf{\ia_{n+1}}}\\
			=&\abf{\cf{\ia_1\cup \cdots\cup \ia_n}}{\cf{\ia_1\cup \cdots\cup \ia_n}}+\abf{\cf{\ia_{n+1}}}{\cf{\ia_{n+1}}}+\abf{\cf{\ia_1\cup \cdots\cup \ia_n}}{\cf{\ia_{n+1}}}+\abf{\cf{\ia_{n+1}}}{\cf{\ia_1\cup \cdots\cup \ia_n}}\ .
		\end{align}
		Now, the first summand of the second formula is one by the inductive hypothesis. On the other hand, thanks to Definition~\ref{def:topological-quiver}-\ref{def:topological-quiver-b}, we can reduce the computations of the other summands to the case of elementary intervals, and rely on equation \eqref{eq:abf-line}. Hence, we get the assertion. 
		
		Let $\ia$ be a non--contractible interval. Assume that $\ia$ is of the form
		\begin{align}
			S^1\sgpp \bigoplus_{k=1}^N T_k= (\cdots(S^1\sgpp T_1)\sgpp T_2)\cdots \sgpp T_N)
		\end{align}
		for some pairwise disjoint contractible intervals $T_k$. Then by the previous remark, we get
		\begin{align}
			\abfcf{\ia}{\ia}=&\abfcf{S^1}{S^1}+\sum_{k=1}^N\, \abfcf{T_k}{T_k}+\sum_{k=1}^{N}\, \abfcf{S^1}{T_k}+\sum_{k=1}^{N}\, \abfcf{T_k}{S^1} =0\ ,
		\end{align}
		since $\abfcf{S^1}{S^1}=0$, the second summand equals $N$ by the previous computation; while, the computation of the last two summands can be obtained by reducing to the case of elementary intervals thanks to Definition~\ref{def:topological-quiver}-\ref{def:topological-quiver-b}: we get $-N$ and zero, respectively.
	\end{remark}
	
	Set $\rbf{f}{g} \coloneqq \abf{f}{g}+\abf{g}{f}$ for $f,g\in\fun{X}{}$. It follows from an easy generalization of 
	the computations carried out in Section~\ref{ss:semigroup-line} that, if $\ia,\ib\in\intsf{X}{}$ are contractible, then  
	\begin{align}
		\rbf{\cf{\ia}}{\cf{\ib}}=
		\begin{cases}
			{\phantom{+}}2 & \text{if } \ia=\ib \ ,\\[4pt]
			{\phantom{+}}1 & \text{if } (\ia,\ib)\in\intsf{X}{}^{(2)}_{\sgpm} \text{ or } (\ib,\ia)\in\intsf{X}{}^{(2)}_{\sgpm}\ ,\\[4pt]
			{\phantom{+}}0 & \text{if } (\ia,\ib)\not\in\intsf{X}{}^{(2)}_{\sgpp} \text{ and } \ia\cap \ib=\emptyset, \\[4pt]
			-1 & \text{if } (\ia,\ib)\in\intsf{X}{}^{(2)}_{\sgpp} \text{ and } \ia\sgpp \ib \text{ is contractible}\ ,\\[4pt]
			-2 & \text{if } (\ia,\ib)\in\intsf{X}{}^{(2)}_{\sgpp} \text{ and } \ia\sgpp \ib \text{ is non--contractible}\ .
		\end{cases}
	\end{align}
	All other cases follow therein. Note in particular that, if $\ia$ is non--contractible, $\rbfcf{\ia}{\ia}=0$.
	
	\subsection{Continuum quivers}\label{ss:topological-quiver-3}
	For any $\ia,\ib\in\intsf{X}{}$, we set
	\begin{align}
		\kappa_X\rbf{\ia}{\ib}\coloneqq \rbf{\cf{\ia}}{\cf{\ib}}\quad\text{and}\quad \xi_X\rbf{\ia}{\ib}\coloneqq (-1)^{\abfcf{\ia}{\ib}}\rbf{\cf{\ia}}{\cf{\ib}}\ .
	\end{align}
	One checks immediately that $\kappa_X$ satisfies the conditions \eqref{eq:sgp-datum-1} and \eqref{eq:kappa-right-dec}. Therefore,
	the datum $\cq{X}\coloneqq(\intsf{X}{}, \kappa_X, \xi_X)$ is a Cartan semigroup, which we refer to as the \textit{continuum quiver of $X$}.
	
	\begin{remark}
		Recall that, given a quiver $\calQ$ with adjacency matrix $\sfB_{\calQ}$, its Cartan matrix is
		the symmetric matrix $\sfA_{\calQ}=2\cdot\mathsf{id}-\sfB_{\calQ}-\sfB_{\calQ}^t$.
		Analogously, we think of $\intsf{X}{}$ as a set of vertices and of $\kappa_X$ as a
		generalized Cartan matrix.
	\end{remark}
	
	The description of locally degenerate, imaginary and real elements in $\cq{X}$
	is easily obtained. Specifically, we have the following.
	
	\begin{lemma}
		Let $\ia, \ib$ be two intervals. 
		\begin{enumerate}\itemsep0.2cm
			\item $\ia$ is a locally degenerate element if and only if it is homemorphic to $S^1$.
			\item $\ia$ is an imaginary (resp. real) element if and only if it is non--contractible (resp.\ contractible).
			\item $\ia$ and $\ib$ are perpendicular if and only if $(\ia,\ib)\not\in\intsf{X}{}^{(2)}_{\sgpp}$ and $\ia\cap \ib=\emptyset$.
		\end{enumerate}
	\end{lemma}
	
	\begin{remark}
		It follows immediately from Remark~\ref{rem:euler-form-identities-II} that 
		\begin{align}
			\kappa_X(\ia,\ia)=
			\begin{cases}
				2 & \text{if } \ia \text{ is real},\\
				0 &  \text{if } \ia \text{ is imaginary}.
			\end{cases}
		\end{align}
	\end{remark} 
	
	Continuum quivers provide a large class of examples of the theory developed in
	Sections~\ref{s:sgp-serre}.
	
	\begin{proposition}\label{prop:topologicalquiver}
		The continuum quiver $\cq{X}$ is a good Cartan semigroup.
	\end{proposition}
	
	\begin{proof}
		We shall show that $\cq{X}$ satisfies the conditions \eqref{def:sym-sgp-1}--\eqref{def:sym-sgp-5} from Definition~\ref{def:sym-sgp}. We first observe that $\intsf{X}{}$ satisfies \eqref{def:sym-sgp-1}, \ie at most one among
		$\ia\sgpp \ib$, $\ia\sgpm \ib$, and $\ib\sgpm \ia$ is defined. An easy check shows that the conditions \eqref{def:sym-sgp-2} and \eqref{def:sym-sgp-3} hold.
		
		It remains to prove that the functions $\xi_{X}$ and $\kappa_X$ satisfy \eqref{def:sym-sgp-4} and \eqref{def:sym-sgp-5}. Note that 
		$\kappa_X$ is symmetric and satisfies the condition \eqref{eq:sym-sgp-datum-1} by definition. 
		The proof of conditions~\eqref{eq:sym-sgp-datum-3}, \eqref{eq:sym-sgp-datum-4}, \eqref{eq:sym-sgp-datum-5}, and \eqref{eq:sym-sgp-datum-2}, which is rather tedious and
		technical, is carried out in full details in Appendix~\ref{app:topologicalquiver}.
	\end{proof}
	
	\subsection{Continuum Kac--Moody algebras and Serre relations}
	We study in greater detail the case of semigroup Lie algebras associated with 
	continuum quivers, providing an explicit description of an essential subset of 
	Serre pairs. 
	
	\begin{definition}
		Let $\cq{X}$ be a continuum quiver. 
		The \textit{continuum Kac--Moody algebra of $X$} is the semigroup Lie 
		algebra $\cqg{X}\coloneqq \g(\cq{X})$.
	\end{definition}
	
	Set $\serre{X}\coloneqq \serre{\cq{X}}$. We provide a list of
	pairs in $\serre{X}$ by studying the admissibility conditions from 
	Definition~\ref{def:admissible-pair}
	in the context of continuum quivers. Note that, in the two simplest 
	cases of the real line and the circle, one checks easily that
	\begin{align}
		\serre{\R}&=\intsf{\R}{}\times \intsf{\R}{}\\
		\serre{S^1}&=
		\intsf{S^1}{}\times\intsf{S^1}{}\setminus
		\{(\ia,\ib)\;\vert\; \ia,\ib\neq S^1,\, \ia\sgpp\ib= S^1\}
	\end{align}
	The general case is roughly the same, but it is necessary to exclude the
	appearance of $S^1$--partitions at the level of subintervals.
	
	\begin{proposition}
		Let $\cqserre{X}$ be the set of unordered pairs of intervals $(\ia,\ib)$ such that 
		either $\ia\perp\ib$
		or 
		$\ia\not\perp\ib$, $\ia$ contractible, and the following conditions hold:
		\begin{enumerate}
			\item for any subinterval $\ia'\subseteq\ia$ and $\ib'\subseteq\ib$ with $\kappa_X\rbf{\ib'}{\ib}\neq0$
			whenever $\ib'\neq\ib$, the element $\ia'\sgpp\ib'$ is either not defined or not homeomorphic
			to $S^1$;
			\item if $\ia$ is not elementary, $\ia\cap\ib$ does not contain any critical point.
		\end{enumerate}	
		Then, $\cqserre{X}\subseteq\serre{X}$.
	\end{proposition}
	
	\begin{proof}
		By definition, if $\ia\perp\ib$, then $(\ia,\ib)\in\serre{X}$. Assume $\ia\not\perp\ib$. 
		Then, since locally degenerate elements are homeomorphic to copies of $S^1$ in $X$,
		condition (1) is equivalent to condition (2b) in Definition~\ref{def:admissible-pair}.
		Condition (2) is explained as follows. We consider the following configuration of disjoint 
		intervals
		\begin{align}
			\begin{tikzpicture}[scale=.35]
				\begin{scope}[on background layer]
					\draw [white] (0,0) rectangle (15,10);
				\end{scope}
				\coordinate (C1) at (3, 5);
				\coordinate (C2) at (6, 5);
				\coordinate (C3) at (9, 5);
				\coordinate (C4) at (12, 7.5);
				\coordinate (C5) at (12, 2.5);
				\draw [->,very thick, green]	(C1) -- (C2);
				\draw [->,very thick, orange]		 (C2) -- (C3);
				\draw [->,very thick, blue]      (C3) arc (270:360:3);
				\draw [->,very thick, red]      (C3) arc (90:0:3);
				\node at (13.5, 7.5) {$z$};
				\node at (13.5, 2.5) {$y$};
				\node at (7.5, 6) {$x$};
				\node at (4.5, 6) {$w$};
			\end{tikzpicture}
		\end{align}
		Set $a\coloneqq w\cup x\cup y\cup z$, $b\coloneqq x$, and $c\coloneqq w\cup x$.
		It is easy to check that the triple $(a,b,c)$ is not admissible, since it does not satisfy the 
		property (1b) from Definition~\ref{def:admissible-triple}. In particular, $(a,b)$ is not an
		admissible pair. Condition (2) prevents this configuration to appear and implies that the
		pair $(\ia,\ib)$ satisfies the property (2a) from Definition~\ref{def:admissible-triple}. The
		result follows.
	\end{proof}
	
	\begin{remark}
		We shall not need to prove that $\cqserre{X}=\serre{X}$. In fact, in the proof of Theorem
		~\ref{thm:main}, we show that the relations indexed by $\cqserre{X}$ generates the maximal
		ideal $\r_X\subset\wt{\g}(\cq{X})$. In particular, $\cqserre{X}$ contains enough information to recover $\serre{X}$
		in any case.
	\end{remark}

	Note that, if $(\ia, \ib)\in\cqserre{X}$ and $\ia\sgpp \ib$ is defined, one has
	\begin{align}
		\xi_X(\ia\sgpp \ib, \ia)=(-1)^{\abfcf{\ia\sgpp \ib}{\ia}}\rbfcf{\ia\sgpp \ib}{\ia}=(-1)^{1+\abfcf{\ib}{\ia}}\left(2+\rbfcf{\ib}{\ia})\right)=(-1)^{\abfcf{\ia}{\ib}}\ ,
	\end{align}
	since $\ia$ is contractible and $\rbfcf{\ib}{\ia}=-1$. Similarly, $\xi_X(\ia,\ia\sgpp \ib)=(-1)^{\abfcf{\ib}{\ia}}$.
	Therefore, from Corollary~\ref{cor:perp-comm} and \ref{cor:commutation-degenerate}, and 
	Theorem~\ref{thm:sym-serre-rel}, we get the following.
	\begin{corollary}\label{prop:serre-TQ}
		Let $\ia, \ib$ be two intervals of $X$. The following relations hold in $\cqg{X}$.
		\begin{enumerate}\itemsep0.2cm
			\item \label{prop:serre-TQ-(1)} If $\ia\perp \ib$ (\ie $\ia\cap \ib=\emptyset$ and 
			$\ia\sgpp \ib$ is not defined), then $[\xpm{\ia}, \xpm{\beta}]=0$.
			\item \label{prop:serre-TQ-(2)} More in general, if $(\ia, \ib)\in \cqserre{X}$, then
			\begin{align}
				[\xp{\ia}, \xp{\ib}]= & (-1)^{\abfcf{\ib}{\ia}}\xp{\ia\sgpp\ib} \ , \\
				[\xm{\ia}, \xm{\ib}]= & (-1)^{\abfcf{\ia}{\ib}}\xm{\ia\sgpp\ib}\ .
			\end{align}
		\end{enumerate}
	\end{corollary}
	
	\begin{remark}\label{rem:serre-TQ}
		\hfill
		\begin{enumerate}\itemsep0.2cm
			\item \label{rem:serre-TQ-(1)} If $\ib\simeq S^1$ and $\ia\subseteq \ib$, then $(\ia, \ib)\in \cqserre{X}$. Hence, by (2) above $[\xpm{\ia}, \xpm{\beta}]=0$.
			\item 	For $X=\R$, the relations above coincide with Equations~\eqref{eq:serre-new}. 
		\end{enumerate}
	\end{remark}
	
	\subsection{Borcherds--Kac--Moody subalgebras}\label{ss:bkm-and-sgp}
	
	Our main goal is to provide an explicit description of the maximal ideal 
	$\r_X\subset\wt{\g}(\cq{X})$, thus providing a complete presentation by generators
	and relations of $\cqg{X}$. To this end, it will be crucial to consider symmetric Borcherds--Kac--Moody
	subalgebras in $\cqg{X}$ associated with certain finite configurations of disjoint intervals.
	
	\begin{definition}\label{def:irreducibleset}
		Let $\calJ=\{\ia_k\}_{k}$ be a finite set of intervals $\ia_k\in\intsf{X}{}$. We say
		that $\calJ$ is \textit{irreducible} if the following conditions hold: 
		\begin{enumerate}\itemsep0.2cm
			\item \label{def:irreducibleset-1} every interval is elementary; 
			\item \label{def:irreducibleset-2} given two intervals $\ia,\ib\in\calJ$, $\ia\neq \ib$, one of the following mutually exclusive cases occurs:
			\begin{enumerate}[label=(\alph*)]\itemsep0.2cm
				\item \label{def:irreducibleset-a} $\ia\sgpp \ib$ is defined;
				\item \label{def:irreducibleset-b} $\ia\perp\ib$; 
				\item \label{def:irreducibleset-c} $\ia\simeq S^1$ and $\ib\subset \ia$\ .
			\end{enumerate}
		\end{enumerate}
	\end{definition}
	
	Assume henceforth that $\calJ$ is an irreducible set of intervals.
	Let $\sfA_{\calJ}$ be the matrix given by the values of $\kappa_X$ on $\calJ$, 
	\ie $\sfA_{\calJ}=\big(\kappa_{X}(\ia, \ib)\big)_{\ia, \ib\in \calJ}$. Note that the diagonal
	entries of $\sfA_{\calJ}$ are either $2$ or $0$, while off--diagonal the only possible
	entries are $0,-1,-2$. Let $\calQ_{\calJ}$ be the corresponding quiver with Cartan matrix $\sfA_{\calJ}$. 
	Note that a contractible elementary interval in $\calJ$ corresponds to a vertex of $\calQ_{\calJ}$ without 
	loops at it. For example, we obtain the following quivers.
	
	\begin{align}
		\begin{array}{|c|c|c|}
			\hline
			\text{Configuration of intervals} & \text{Borcherds--Cartan diagram}\\
			\hline &\\
			\begin{tikzpicture}[scale=.35]
				\begin{scope}[on background layer]
					\draw [white] (0,0) rectangle (15,10);
				\end{scope}
				\coordinate (C1) at (3, 5);
				\coordinate (C2) at (6, 5);
				\coordinate (C3) at (9, 5);
				\coordinate (C4) at (12, 7.5);
				\coordinate (C5) at (12, 2.5);
				\coordinate (CON1) at (10,5);
				\coordinate (CON2) at (12,5);
				\draw [->,very thick, purple]	(C1) -- (C2);
				\draw [->,very thick, blue]		 (C2) -- (C3);
				\draw [->,very thick, yellow]      (C3) arc (270:360:3);
				\draw [->,very thick, orange]      (C3) arc (90:0:3);
				\node at (4.5, 6) {$\ia_1$};
				\node at (7.5, 6) {$\ia_2$};
				\node at (13, 6.5) {$\ia_3$};
				\node at (13, 3.5) {$\ia_4$};
			\end{tikzpicture}
			&
			\begin{tikzpicture}[scale=.35]
				\begin{scope}[on background layer]
					\draw [white] (0,0) rectangle (15,10);
				\end{scope}
				\node (V1) at (3.5,5)      [circle,draw=purple,fill=purple, inner sep=3pt]    {};
				\node (V2) at (7.5,5)      [circle,draw=blue,fill=blue, inner sep=3pt]          {};
				\node (V3) at (10.5,7.5) [circle,draw=yellow,fill=yellow, inner sep=3pt]   {};
				\node (V4) at (10.5,2.5) [circle,draw=orange,fill=orange, inner sep=3pt]  {};
				\draw [->, very thick] (V1) -- (V2);  
				\draw [<-, very thick] (V3) -- (V2);
				\draw [<-, very thick] (V4) -- (V2);
				\node at (3.5, 6)    {$\alpha_1$};
				\node at (7.5, 6)    {$\alpha_2$};
				\node at (11.5, 7.5)  {$\alpha_3$};
				\node at (11.5, 2.5)  {$\alpha_4$};
			\end{tikzpicture}
			\\
			\hline 
			\begin{tikzpicture}[scale=.35]
				\begin{scope}[on background layer]
					\draw [white] (0,0) rectangle (15,10);
				\end{scope}
				\draw [->, very thick, purple] (7.5,7.5) arc (90:270:2.5);
				\draw [->, very thick, blue] (7.5,2.5) arc (-90:90:2.5);
				\node at (4, 5)    {$\ia_1$};
				\node at (11, 5)    {$\ia_2$};
			\end{tikzpicture}
			&
			\begin{tikzpicture}[scale=.35]
				\begin{scope}[on background layer]
					\draw [white] (0,0) rectangle (15,10);
				\end{scope}
				\node (V1) at (5,5)      [circle,draw=purple,fill=purple, inner sep=3pt]    {};
				\node (V2) at (10,5)      [circle,draw=blue,fill=blue, inner sep=3pt]          {};
				\coordinate (CON1) at (6.5, 7.5);
				\coordinate (CON2) at (8.5, 7.5);
				\coordinate (CON3) at (6.5, 2.5);
				\coordinate (CON4) at (8.5, 2.5);
				\draw [->, very thick] (V1) .. controls (CON1) and (CON2) .. (V2);  
				\draw [<-, very thick] (V1) .. controls (CON3) and (CON4) .. (V2);  
				\node at (5, 6)    {$\alpha_1$};
				\node at (10.2, 6)    {$\alpha_2$};
			\end{tikzpicture}
			\\
			\hline 
			\begin{tikzpicture}[scale=.35]
				\begin{scope}[on background layer]
					\draw [white] (0,0) rectangle (15,10);
				\end{scope}
				\draw [->, very thick, purple] (7.5,7.5) arc (90:180:2.5);
				\draw [->, very thick, yellow] (5,5) arc (180:270:2.5);
				\draw [->, very thick, blue] (7.5,2.5) arc (270:360:2.5);
				\draw [->, very thick, orange] (10,5) arc (0:90:2.5);
				\draw [<-, very thick, green] (2.5,2.5) arc (270:360:2.5);
				\draw [->, very thick, red] (10,5) arc (180:90:2.5);
				\node at (3.5, 3.5)      {$\ia_1$};
				\node at (5.5, 7.5)    {$\ia_2$};
				\node at (5.5, 2.5)    {$\ia_3$};
				\node at (9.5, 7.5)    {$\ia_4$};
				\node at (9.5, 2.5)    {$\ia_5$};
				\node at (11.5, 6.5)      {$\ia_6$};
			\end{tikzpicture}
			&
			\begin{tikzpicture}[scale=.35]
				\begin{scope}[on background layer]
					\draw [white] (0,0) rectangle (15,10);
				\end{scope}
				\node (V1) at (2.5,5)        [circle,draw=green,fill=green, inner sep=3pt]    {};
				\node (V2) at (5,5)      [circle,draw=purple,fill=purple, inner sep=3pt]          {};
				\node (V3) at (7.5,2.5)      [circle,draw=yellow,fill=yellow, inner sep=3pt]          {};
				\node (V4) at (7.5,7.5)      [circle,draw=orange,fill=orange, inner sep=3pt]          {};
				\node (V5) at (10,5)      [circle,draw=blue,fill=blue, inner sep=3pt]          {};
				\node (V6) at (12.5,5)      [circle,draw=red,fill=red, inner sep=3pt]          {};
				\node at (2.5, 6)    {$\alpha_1$};
				\node at (4.8, 6)    {$\alpha_2$};
				\node at (7.5, 1.5)    {$\alpha_3$};
				\node at (7.5, 8.5)    {$\alpha_4$};
				\node at (10, 6)    {$\alpha_5$};
				\node at (12.5, 6)    {$\alpha_6$};
				\draw [<-, very thick] (V1) -- (V2);  
				\draw [->, very thick] (V2) -- (V3);
				\draw [->, very thick] (V3) -- (V5);
				\draw [->, very thick] (V5) -- (V4);
				\draw [->, very thick] (V4)-- (V2);  
				\draw [->, very thick] (V5) -- (V6);
			\end{tikzpicture}
			\\
			\hline
		\end{array}
	\end{align}
	
	Instead, an interval of $\calJ$ homeomorphic to $S^1$ corresponds in $\calQ_{\calJ}$ to 
	a vertex having exactly one loop at it, as in the following examples.
	
	\begin{align}
		\begin{array}{|c|c|}
			\hline
			\text{Configuration of intervals} & \text{Borcherds--Cartan diagram}\\
			\hline &\\
			\begin{tikzpicture}[scale=.35]
				\begin{scope}[on background layer]
					\draw [white] (0,0) rectangle (15,10);
				\end{scope}
				\draw [->, very thick, purple] (10,5) arc (0:360:2.5);
				\draw [->, very thick, blue] (7.5,7.5) arc (90:180:2.5);
				\draw [<-, very thick, green] (2.5,2.5) arc (270:360:2.5);
				\node at (3.5, 3.5)      {$\ia_1$};
				\node at (5.5, 7.5)    {$\ia_2$};
				\node at (11, 5)    {$\ia_3$};
			\end{tikzpicture}
			&
			\begin{tikzpicture}[scale=.35]
				\begin{scope}[on background layer]
					\draw [white] (0,0) rectangle (15,10);
				\end{scope}
				\node (V2) at (7.5, 7.5)  [circle,draw=blue,fill=blue, inner sep=3pt]    {};
				\node (V1) at (5, 5)  [circle,draw=green,fill=green, inner sep=3pt]    {};
				\node (V3) at (10, 5)  [circle,draw=purple,fill=purple, inner sep=3pt]    {};
				\draw [->, very thick] (V2) -- (V1);   
				\draw [->, very thick] (V3) -- (V1);
				\draw [->, very thick] (13,5) arc (0:360:1.5);  
				\node at (10, 5)  [circle,draw=purple,fill=purple, inner sep=3pt]    {};
				\node at (4.8, 6)    {$\alpha_1$};
				\node at (7.5, 8.5)    {$\alpha_2$};
				\node at (9.5, 6)    {$\alpha_3$};
			\end{tikzpicture}\\
			\text{($\ia_3$ is a full circle)} & \\
			\hline
		\end{array}
	\end{align}
	\vspace{-0.5cm}
	\begin{align}
		\begin{array}{|c|c|}
			\hline &\\
			\begin{tikzpicture}[scale=.35]
				\begin{scope}[on background layer]
					\draw [white] (0,0) rectangle (15,10);
				\end{scope}
				\draw [->, very thick, purple] (12.5,5) arc (0:360:2.5);
				\draw [<-, very thick, green] (2.5,5) arc (180:540:2.5);
				\node at (5, 8.5)      {$\ia_1$};
				\node at (10, 8.5)      {$\ia_2$};
			\end{tikzpicture}
			&
			\begin{tikzpicture}[scale=.35]
				\begin{scope}[on background layer]
					\draw [white] (0,0) rectangle (15,10);
				\end{scope}
				\node (V1) at (5, 5)  [circle,draw=green,fill=green, inner sep=3pt]    {};
				\node (V3) at (10, 5)  [circle,draw=purple,fill=purple, inner sep=3pt]    {};
				\coordinate (CON1) at (6.5, 7.5);
				\coordinate (CON2) at (8.5, 7.5);
				\coordinate (CON3) at (6.5, 2.5);
				\coordinate (CON4) at (8.5, 2.5);
				\draw [->, very thick] (V1) .. controls (CON1) and (CON2) .. (V3);  
				\draw [<-, very thick] (V1) .. controls (CON3) and (CON4) .. (V3);  
				\draw [->, very thick] (13,5) arc (0:360:1.5);  
				\draw [->, very thick] (2,5) arc (180:540:1.5);  
				\node at (V3) [circle,draw=purple,fill=purple, inner sep=3pt]    {};
				\node at (V1) [circle,draw=green,fill=green, inner sep=3pt]    {};
				\node at (5, 6.5)    {$\alpha_1$};
				\node at (10, 6.5)    {$\alpha_2$};
			\end{tikzpicture}\\
			\text{($\ia_1,\ia_2$ are full circles)} & \\
			\hline
		\end{array}
	\end{align}
	
	We shall consider two Lie algebras associated to $\calJ$. First, let $\g(\calJ)$ be 
	the Lie subalgebra of $\cqg{X}$ generated by the elements $\xpm{\ia}$ and $\xz{\ia}$ with $\ia\in \calJ$. Then, let $\g_{\calQ_{\calJ}}$ be the symmetric Borcherds--Kac--Moody algebra 
	of $\calQ_{\calJ}$ (\ie the derived Lie algebra of $\g(\sfA_{\calJ})$ --- cf. Section~\ref{s:km-vkm}).
	
	\begin{proposition}\label{prop:bkm-sgp}
		The assignment 
		\begin{align}
			e_\ia\mapsto \xp{\ia}\ , \quad f_\ia\mapsto \xm{\ia} \quad \text{and} \quad h_\ia \mapsto \xz{\ia}
		\end{align}
		for any $\ia\in \calJ$, defines a surjective homomorphism of Lie algebras $\Phi_\calJ\colon \g_{\calQ_{\calJ}}\to \g(\calJ)$.
	\end{proposition}
	
	\begin{proof}
		It is enough to show that $\Phi_{\calJ}$ is a Lie algebra map. The surjectivity is clear. 
		Recall that $\g_{\calQ_{\calJ}}$ is generated by $\{e_\ia, h_\ia, f_\ia\,|\, \ia\in\calJ\}$ with
		the following defining relations:
		\begin{align}\label{eq:km-rel-TQ}
			[h_\ia,e_{\ib}]=\kappa_{X}(\ia,\ib)\, e_{\ib}\ ,
			\quad
			[h_\ia,f_{\ib}]=-\kappa_{X}(\ia,\ib)\, e_{\ib}\ ,
			\quad
			[e_\ia,f_{\ib}]=\drc{\ia\ib}\, h_\ia\ ,
		\end{align}
		\begin{align}\label{eq:Serrerel-BKM-1-TQ}
			\sfad(e_\ia)^{1-\rbf{\alpha}{\beta}}(e_{\ib})=0=\sfad(f_{\ia})^{1-\rbf{\alpha}{\beta}}(f_{\ib})
			\qquad \text{ if }\quad \ia\not\simeq S^1\ ,
		\end{align}
		\begin{align}\label{eq:Serrerel-BKM-2-TQ}
			[e_\ia,e_{\ib}]=0=[f_{\ia},f_{\ib}]\qquad 
			\text{ if }\quad \ia\simeq S^1\quad \text{ and } \quad \rbf{\alpha}{\beta}=0\ .
		\end{align}
		Note that, for any $\ia,\ib\in\calJ$, their difference $\ia\sgpm \ib$ is defined
		only in the case \ref{def:irreducibleset-c}, thus necessarily $\kappa_X\rbf{\ia}{\ib}=0$. Therefore, the relation
		\eqref{eq:km-rel-TQ} is easily seen to be satisfied in $\g(\calJ)$. We shall prove 
		that the (standard) Serre relations \eqref{eq:Serrerel-BKM-1-TQ} and 
		\eqref{eq:Serrerel-BKM-2-TQ} hold in $\g(\calJ)$.
		
		Let $\ia,\ib\in\calJ$, $\ia\neq \ib$.
		First, note that, by construction, $\rbf{\alpha}{\beta}=0$ if and only if
		we are either in case \ref{def:irreducibleset-b}, \ie $\ia,\ib$ are perpendicular, or in case \ref{def:irreducibleset-c}, \ie $\ia\simeq S^1$
		and $\ib\subset S^1$. Thus, by Corollary~\ref{prop:serre-TQ}--\eqref{prop:serre-TQ-(1)} and \eqref{prop:serre-TQ-(2)} (and Remark~\ref{rem:serre-TQ}--\eqref{rem:serre-TQ-(1)}),
		$[\xpm{\ia},\xpm{\beta}]=0$. Therefore, relations \eqref{eq:Serrerel-BKM-1-TQ} (for $\rbf{\alpha}{\beta}=0$) and equation
		\eqref{eq:Serrerel-BKM-2-TQ} hold. 
		
		Assume now that $\rbf{\alpha}{\beta}\neq0$ (\ie $\rbf{\alpha}{\beta}=-1,-2$) and $\ia\sgpp \ib$ is defined.
		Then, necessarily, one of the following occurs:
		\begin{enumerate}\itemsep0.2cm
			\item $\rbf{\alpha}{\beta}=-1$ and $(\ia, \ib)\in \cqserre{X}$\ ;
			\item $\rbf{\alpha}{\beta}=-2$ and $\ia\sgpp \ib\simeq S^1$\ .
		\end{enumerate}
		In case (1), it follows from Corollary~\ref{cor:serre-rel} that 
		$[\xpm{\ia}, [\xpm{\ia}, \xpm{\beta}]]=0$ (we assume $\ia$ is contractible)
		and therefore the Serre relation \eqref{eq:Serrerel-BKM-1-TQ} with $\rbf{\alpha}{\beta}=-1$ is 
		satisfied. In case (2), it is easy to see that we are in the case described by Proposition~\ref{prop:locallynilpotent}
		with $k=2$. Therefore,
		\begin{align}
			[\xpm{\ia},[\xpm{\ia}, [\xpm{\ia}, \xpm{\beta}]]]=0\ ,
		\end{align} 
		and the Serre relation \eqref{eq:Serrerel-BKM-1-TQ} with $\rbf{\alpha}{\beta}=-2$ is 
		satisfied. The result follows.
	\end{proof}
	
	\subsection{A presentation by generators and relations}
	We now prove the main result of this section, providing a presentation by generators 
	and relations of continuum Kac--Moody algebras. For simplicity, we set
	$\abf{\ia}{\ib}\coloneqq\abfcf{\ia}{\ib}$ and $\rbf{\ia}{\ib}\coloneqq\rbfcf{\ia}{\ib}$.
	Recall also that $\ia\perp\ib$ if $\ia\cap\ib=\emptyset$ and $\ia\sgpp\ib$ is not defined.
	
	\begin{theorem}\label{thm:main}
		The continuum Kac--Moody algebra $\cqg{X}$ is generated by the elements 
		$\xpm{\ia}$ and $\xz{\ia}$ for $\ia\in \intsf{X}{}$, subject to the following 
		defining relations: 
		\begin{enumerate}
			\item \label{thm:main-1} for any $\ia, \ib\in\intsf{X}{}$ such that $\ia\sgpp \ib$ is defined, 
			\begin{align}
				\xz{\alpha\sgpp\beta}=\xz{\ia}+\xz{\ib}\ ;
			\end{align}
			\item \label{thm:main-2}
			for any $\ia, \ib\in\intsf{X}{}$, 
			\begin{align}
				[\xz{\ia},\xz{\ib}] &=0\ ,\label{eq:gX-rel-1}\\[3pt]
				[\xz{\ia}, \xpm{\beta}] &=\pm\rbf{\ia}{\ib}\, \xpm{\beta}\ , \label{eq:gX-rel-2}\\[3pt]
				[\xp{\ia},\xm{\ib}]&=\drc{\ia,\ib}\, \xz{\ia}+(-1)^{\abf{\ia}{\ib}}\rbf{{\ia}}{{\ib}}
				\left(\xp{\ia\sgpm\ib}-\xm{\ib\sgpm\ia}\right)\ ;\label{eq:gX-rel-3}
			\end{align}
			\item \label{thm:main-3} if $(\ia, \ib)\in \cqserre{X}$, then
			\begin{align}\label{eq:gX-rel-4}
				\begin{aligned}
					[\xp{\ia}, \xp{\ib}]= & (-1)^{\abf{\ib}{\ia}}\xp{\ia\sgpp\ib} \ , \\
					[\xm{\ia}, \xm{\ib}]= & (-1)^{\abf{\ia}{\ib}}\xm{\ia\sgpp\ib}\ .
				\end{aligned}
			\end{align}
		\end{enumerate}
	\end{theorem}
	
	\begin{proof}
		Recall that, by definition, $\cqg{X}=\cqwtg{X}/\r_X$, where $\cqwtg{X}$ is the Lie algebra generated by $\xpm{\ia}$ and $\xz{\ia}$, $\ia\in\intsf{X}{}$,
		with relations \eqref{thm:main-1} and \eqref{thm:main-2}, and $\r_X\subset\cqwtg{X}$ is the sum of all two--sided graded ideals with trivial intersection with the commutative subalgebra generated
		by $\xz{\ia}$ with $\ia\in\intsf{X}{}$. Let $\stackrel{\circ}{\r}_X\subset\cqwtg{X}$ be the ideal generated by relations relation \eqref{eq:gX-rel-4}. By Corollary~\ref{prop:serre-TQ}, we know that $\stackrel{\circ}{\r}_X\subseteq\r_X$.
		Thus we have to prove that $\stackrel{\circ}{\r}_X=\r_X$. We obtain this identity as
		a consequence of the Gabber--Kac theorem for Borcherds--Kac--Moody algebras
		\cite{gabber-kac-81,borcherds-88}.
		
		Set $\stackrel{\circ}{\g}\!\!(X)\coloneqq\cqwtg{X}/\! \stackrel{\circ}{\r}_X$. Let $\pi\colon \cqwtg{X}\to\, \stackrel{\circ}{\g}\!\!(X)$ be the natural projection and assume that there exists $v\in\pi(\r_X)$ with $v\neq0$. Let $\calJ(v)$ be any finite set of intervals such that $v$ belongs $\stackrel{\circ}{\g}\!\!(\calJ(v))$, where $\stackrel{\circ}{\g}\!\!(\calJ(v))\! \subseteq\; \stackrel{\circ}{\g}\!\!(X)$ is the Lie subalgebra generated by the elements $\xpm{\ia}$ and $\xz{\ia}$ with $\ia\in\calJ(v)$. Using relations \eqref{eq:gX-rel-4}, we can always assume that $\calJ(v)$ is an irreducible set of intervals (cf. Section~\ref{ss:bkm-and-sgp}). Moreover, since the result of Proposition~\ref{prop:bkm-sgp} relies exclusively on the relation \eqref{eq:gX-rel-4}, we can conclude that the homomorphism $\Phi_{\calJ(v)}$ factors through $\stackrel{\circ}{\g}\!\!(\calJ(v))$, \ie there exists a surjective homomorphism $\stackrel{\circ}{\Phi}_{\calJ(v)}\colon \g_{\calQ_{\calJ(v)}}\to\stackrel{\circ}{\g}(\calJ(v))$ such that $\Phi_{\calJ(v)}=\stackrel{\circ}{\pi} \circ \stackrel{\circ}{\Phi}_{\calJ(v)}$, where $\stackrel{\circ}{\pi}\colon  \stackrel{\circ}{\g}\!\!(\calJ(v)) \to\g(\calJ(v))$ is the canonical projection. Since $v\in\pi(\r_X)$ and $\stackrel{\circ}{\Phi}_{\calJ(v)}$ is the identity on the Cartan subalgebras, $\big(\stackrel{\circ}{\Phi}_{\calJ(v)}\big)^{-1}(v)$ generates a two--sided ideal which trivially intersect the Cartan subalgebra in $\g_{\calQ_{\calJ}}$. By \cite[Corollary~2.6]{borcherds-88}, this is necessarily trivial. Therefore, $v=0$ and $\stackrel{\circ}{\r}_X=\r_X$. The result follows.
	\end{proof}
	
	\begin{rem}
		It follows from the proof above that the morphism $\Phi_\calJ:\g_{\calQ_{\calJ}}\to\g(\calJ)$
		from Proposition~\ref{prop:bkm-sgp} is an isomorphism.\hfill $\triangle$
	\end{rem}
	
	\subsection{Finite quivers and colimit structure}
	The proof of Theorem~\ref{thm:main} crucially relies on the fact that $\cqg{X}$ 
	can be covered by symmetric Borcherds--Kac--Moody algebras. We describe this
	relation in greater details.
	
	\begin{lemma}
		Let $\calJ,\calJ'$ be two irreducible finite sets of intervals in $X$.
		\begin{enumerate}\itemsep0.2cm
			\item If $\calJ'\subseteq\calJ$, there is a canonical embedding
			$\phi'_{\calJ,\calJ'}\colon \g(\calJ')\to\g(\calJ)$ sending $\xpm{\ia}\mapsto \xpm{\ia}$ and $\xz{\ia}\mapsto \xz{\ia}$ for $\ia\in\calJ'$.
			\item If $\calJ$ is obtained from $\calJ'$ by replacing an element $\ia_s\in\calJ'$
			with two intervals $\ia_1,\ia_2$ such that $\ia_s=\ia_1\sgpp \ia_2$,
			there is a canonical embedding $\phi''_{\calJ,\calJ'}\colon \g(\calJ')\to\g(\calJ)$, which sends
			\begin{align}
				\xpm{\ia}& \mapsto \xpm{\ia} \quad \text{and} \quad \xz{\ia}\mapsto \xz{\ia} \quad \text{for $\ia\in\calJ'\setminus\{\ia_s\}$}\ , \quad \xz{\ia_s} \mapsto \xz{\ia_1}+\xz{\ia_2}\ , \\[2pt]
				\xp{\ia_s} &\mapsto (-1)^{\abfcf{\ia_2}{\ia_1}}\, [\xp{\ia_1},\xp{\ia_2}] \ , \quad \xm{\ia_s}\mapsto (-1)^{\abfcf{\ia_1}{\ia_2}}\, [\xm{\ia_1},\xm{\ia_2}] \ .
			\end{align}	
		\end{enumerate}
	\end{lemma}
	
	\begin{proof}
		(1) is clear. (2) is a straightforward consequence of Theorem~\ref{thm:main}.	
	\end{proof}
	
	\begin{definition}
		Let $\calQ$ be a finite quiver. We say that $\calQ$ \emph{has shape $X$} if there exists an
		irreducible set of intervals $\calJ$ such that $\calQ=\calQ_\calJ$. We denote by $\mathsf{Sh}(X)$
		the set of all such pairs $(\calQ,\calJ)$.
	\end{definition}
	
	Thus, for any $(\calQ,\calJ)\in\mathsf{Sh}(X)$, there is a canonical 
	embedding $\psi_\calJ\colon \g_\calQ\to\cqg{X}$ which factors through the 
	isomorphism $\Phi_{\calJ}\colon\g_{\calQ}\to\g(\calJ)$ given by Proposition~\ref{prop:bkm-sgp}.
	We also obtain analogous injective morphisms
	\begin{align}
		\phi'_{\calQ,\calQ'}\coloneqq  \Phi_{\calJ}^{-1}\circ \phi_{\calJ,\calJ'}'\circ \Phi_{\calJ'}\quad\text{and}\quad \phi''_{\calQ,\calQ'} \coloneqq \Phi_{\calJ}^{-1}\circ \phi_{\calJ,\calJ'}''\circ \Phi_{\calJ'}\ .
	\end{align}
	Since the morphisms $\phi'_{\calQ,\calQ'}, \phi''_{\calQ,\calQ'}$ clearly
	form a direct system and every element in $\cqg{X}$ is contained in the image of $\psi_{\calJ}$
	for some $\calJ$, we obtain the following colimit realization of the continuum Kac--Moody 
	algebra $\cqg{X}$.
	
	\begin{theorem}\label{thm:colim}
		The morphisms $\psi_{\calJ}$ induce a canonical isomorphism of Lie algebras 
		\begin{align}
			\psi_X\colon \colim_{\scriptscriptstyle(\calQ,\calJ)\in\mathsf{Sh}(X)}\, \g_{\calQ}\to \cqg{X} \ .
		\end{align}
	\end{theorem}
	
	\begin{example}
		In the case of $X=\R, S^1$, we compare our construction with the Lie algebras appearing in 
		\cite{sala-schiffmann-17}, \ie the Lie algebras $\sl(\R)$ and $\sl(S^1)$ underlying of the 
		quantum group $\bfU_\upsilon(\sl(S^1))$.
		
		If $X=\R$, one checks easily that $\g_\R$ coincides 
		with the Lie algebra $\sl(\R)$ introduced in Section~\ref{ss:semigroup-line}.
		Consider now the case $X=S^1$. By Theorem~\ref{thm:main}, one 
		checks immediately that the two Lie algebras do not coincide, but 
		their difference is reduced to the elements $\xpm{S^1}$.
		More precisely, let $\ol{\g}_{S^1}$ be the subalgebra in $\g_{S^1}$ generated by the elements 
		$\xpm{\ia}$ and $\xz{\ia}$, $\ia\neq S^1$. 
		Note that the elements $\xpm{S^1}$ and $\xz{S^1}$ generate a \textit{Heisenberg Lie algebra $\heis$ 
			of order one} (cf.\  \cite[Section 2.8]{kac-90}). 
		It is then clear that $\g_{S^1}=\ol{\g}_{S^1}\oplus\heis$ and there is a canonical 
		embedding $\sl(S^1)\to\g_{S^1}$, whose image is $\ol{\g}_{S^1}\oplus\bsfld\cdot \xz{S^1}$. 
	\end{example}

	\subsection{Continuum Kac--Moody algebras with marked support}
	
	We conclude this section by mentioning that Theorem~\ref{thm:main} can be easily extended to
	the case of a {\em vertex space with (possibly infinitely many) marked points}. This shall
	be thought of as a generalization of the case of intervals with integral or rational boundaries. 
	
	Let $Y\subseteq X$ be a subset. 
	We say that an interval $\alpha$ is \textit{supported on $Y$} if $\partial\ia\coloneqq\overline{\ia}\smallsetminus \ia^\circ \subset Y$. 
	We denote by $\intsf{X}{Y}$ the set of all intervals in $X$ supported on $Y$. 
	It is clear that $\intsf{X}{Y}$ is a subsemigroup of $\intsf{X}{}$ and the restrictions
	of $\kappa_X$ and $\xi_X$ to it gives rise to a good Cartan semigroup
	$\cq{X,Y}\coloneqq(\intsf{X}{Y}, \kappa_X,\xi_X)$. 
	%
	Finally, the continuum Kac--Moody algebra of $X$ 
	\textit{supported on $Y$} is the semigroup Lie algebra 
	$\g_{X,\, Y}\coloneqq \g(\cq{X,Y})$.
	
	One checks easily that an analogue of Theorem~\ref{thm:main} holds for $\g_{X,\, Y}$, 
	which can therefore presented in terms of explicit Serre relations involving only 
	intervals supported on $Y$. In particular, $\g_{X,\, Y}$ can be thought of as a subalgebra 
	of $\cqg{X}$ and, for $Y'\subseteq Y$, one has $\g_{X,\, Y'}\subseteq\g_{X,\, Y}$.
	
	\begin{example}
		For $X=\R$ and $Y=\Z,\Q$, we obtain the sets of intervals with integral and rational 
		boundaries defined in Remark~\ref{rmk:slK}. More generally, we can consider
		integral and rational points in an arbitrary vertex space $X$. Namely, let $\KZQ$. 
		We say that a point $x\in X$ is a \textit{$\K$-point} if 
		there exists a chart $(U, A, \phi)$  around $x$ such that $\phi_i(x)\in \K$ holds for any $i$.
		We denote the subsets of integral and rational points by $X_\Z$ and $X_\Q$, respectively. 
		In analogy with the case of $\sl(\Z)$ and $\sl(\Q)$, one can consider 
		$\g_{X,X_\Z}$ and $\g_{X,X_\Q}$. Note that, in contrast with the case of $\cqg{X}$, 
		these Lie algebras have countable dimensions.
	\end{example}
	
	\begin{example}
		Let $x\in X$. Clearly, a loop centered a $x$ exists if and only if $x$ is contained in the 
		image of a circle in $X$. In this case, it follows that $\g_{X,\{x\}}\simeq\heis$.
		Conversely, if $x$ is not contained in a circle, $\g_{X,\{x\}}=\{0\}$.
	\end{example}
	
	\begin{example}
		The finite supports provide an alternative approach to the 
		symmetric Borcherds--Kac--Moody subalgebras. Namely, let $(\calQ,\calJ)\in\mathsf{Sh}(X)$ 
		such that $\bigcup_{\ia\in\calJ}\ia$ is connected and set $\partial\calJ\coloneqq\bigcup_{\ia\in\calJ}\partial\ia$.
		Then, we have $\g_{\calQ}\simeq\g(\calJ)=\g_{X,\partial\calJ}$ and finally
		\begin{align}
			\cqg{X}=\colim_{\scriptscriptstyle Y\subset X, \, |Y|<\infty} \g_{X,Y}\ .
		\end{align}
	\end{example}

	\addtocontents{toc}{\protect\setcounter{tocdepth}{1}}
	\appendix
	
	\bigskip\section{Generalities on partial semigroups}\label{app:sgp}
	
	In this section, we review several basic notions related to partial semigroups,
	following \cite{appel-toledano-19a}. Note that Lemmas~\ref{lem:cancellation-associative}
	and \ref{lem:cancellation-associative-2} are instrumental to the notion of admissible
	pair from Definition~\ref{def:admissible-pair}.
	
	\subsection{Basic definitions}\label{ss:sgp}
	
	\begin{definition}\label{def:partial-semigroup}
		A \textit{partial semigroup} is a tuple $(\sgp, \dsgp_{\sigma}, \sigma)$, where $\sgp$ is a set, 
		$\dsgp_{\sigma}\subseteq\sgp\times\sgp$ a subset, and $\sigma\colon \dsgp_{\sigma}\to\sgp$ 
		a map such that, for any $\alpha,\beta,\gamma\in\sgp$, 
		\begin{align}
			\sigma(\sigma(\alpha,\beta),\gamma)=\sigma(\alpha,\sigma(\beta,\gamma))
		\end{align}
		when both sides are defined, that is if the pairs $(\alpha,\beta)$, $(\sigma(\alpha,\beta),\gamma)$, 
		$(\beta,\gamma)$, $(\alpha,\sigma(\beta,\gamma))$ belong to $\dsgp_{\sigma}$. We say that a partial semigroup is \textit{total} if $\dsgp_{\sigma}=\sgp\times\sgp$.
		
		A partial semigroup is \textit{commutative} if $\dsgp_{\sigma}$ is symmetric, \ie 
		$(\alpha,\beta)\in\dsgp_{\sigma}$ if and only if $(\beta,\alpha)\in\dsgp_{\sigma}$, 
		in which case $\sigma(\alpha,\beta)=\sigma(\beta,\alpha)$.
	\end{definition}
	
	
	\begin{example}
		Let
		\begin{align}
			\rtl_+\coloneqq\bigoplus_{i\in\bfI}\Z_{\geqslant0}\, {\alpha}_i\subseteq{\h}^\ast\ ,
		\end{align}
		be the root lattice of a Kac--Moody algebra $\g$ and $\rts_+\coloneqq \{\alpha\in\rtl_+\setminus\{0\}\;\vert \; {\g}_{\alpha}\neq0\}$ 
		the set of positive roots (cf.\ Section~\ref{ss:km}). Then, $\rts_+$ is naturally endowed with a structure of partial semigroup $\sigma$, 
		induced by $\rtl_+$. That is, for any $\alpha,\beta\in\rts_+$, we set
		\begin{align}
			\sigma(\alpha,\beta)=
			\begin{cases}
				\alpha+\beta & \text{if } \alpha+\beta\in\rts_+ \ ,\\[4pt]
				\text{  \;n.d. } & \text{otherwise}\ .
			\end{cases}
		\end{align}
	\end{example}
	
	
	\begin{remark}
		It is common in the literature (see \eg \cite{evseev}) to assume that the semigroup law 
		$\sigma$ is strongly associative, \ie for any $\alpha,\beta, \gamma\in\sgp$,
		\begin{align}
			(\alpha,\beta), (\sigma(\alpha,\beta),\gamma)\in\dsgp_{\sigma}
			\quad\text{if and only if}\quad(\beta,\gamma), 
			(\alpha, \sigma(\beta,\gamma))\in\dsgp_{\sigma}\ .
		\end{align}
		This definition is stronger than the one given above, and is not suited for our purposes, 
		since it does not hold for root systems. For instance, in the root system of $\mathfrak{sl}(4)$, 
		$\alpha_2+\alpha_1$ and $(\alpha_2+\alpha_1)+\alpha_3$ are defined, but clearly $\alpha_1+\alpha_3$ is not. 
	\end{remark}
	
	
	We briefly recall the straightforward notions of a morphism of partial semigroups and of 
	a partial subsemigroup.
	Let $(\sfS,\sigma), (\sfT,\tau)$ be partial semigroups. A \textit{morphism} $\phi\colon (\sfS,\sigma)\to(\sfT,\tau)$ 
	is a map such that $(\alpha,\beta)\in\dsgp_{\sigma}$ if and only if $(\phi(\alpha),\phi(\beta))\in\sfT
	^{(2)}_{\tau}$, and $\phi(\sigma_{\sgp}(\alpha,\beta))=\sigma_{\sfT}(\phi(\alpha),\phi(\beta))$ for any 
	$(\alpha,\beta)\in\dsgp_{\sigma}$.
	
	Any subset $\sgp'\subseteq\sgp$ inherits a partial semigroup structure. Namely, we denote by $\sft(\sgp')$ 
	the semigroup with underlying set $\sgp'$, 
	\begin{align}
		\sft(\sgp')^{(2)}_{\sigma}\coloneqq \{(\alpha,\beta)\in\sgp'\times\sgp'\;\vert\; (\alpha,\beta)\in\dsgp_{\sigma}\;\text{and}\;
		\sigma(\alpha,\beta)\in\sgp'\}\ ,
	\end{align}
	and semigroup law induced by that of $\sgp$. Note that a priori $\sft(\sgp')^{(2)}_{\sigma}\subseteq (\sgp'\times\sgp')\cap\dsgp_{\sigma}$. The corresponding embedding $\sft(\sgp')\to\sgp$ is a morphism of semigroups if and only if $\sgp'$ is a 
	\textit{sub--semigroup} of $\sgp$, \ie if 
	$(\alpha,\beta)\in(\sgp'\times\sgp')\cap\dsgp_{\sigma}$ implies $\sigma(\alpha,\beta)\in
	\sgp'$ (which means exactly $\sft(\sgp')^{(2)}_{\sigma}= (\sgp'\times\sgp')\cap\dsgp_{\sigma}$).
	
	Finally, for any $\alpha\in\sgp$, set
	\begin{align}
		\sgp^{(2)}_{\sigma,\, \alpha}\coloneqq\{(\beta,\gamma)\in\dsgp\:\vert\;\sigma(\beta,\gamma)=\alpha\} \ .
	\end{align}
	
	\begin{definition}\label{def:saturated}
		We say that a subset $\sgp'\subseteq\sgp$ is \textit{saturated} if $\sgp^{(2)}_{\sigma,\, \alpha}\subseteq\sgp'\times \sgp'$ for any $\alpha\in\sgp'$.
	\end{definition}
	
	
	\subsection{Partial semigroups with cancellation laws}
	
	\begin{definition}\label{def:cancellation}
		A partial semigroup $(\sgp,\sigma)$ is (right) {\em cancellative} if, for any two 
		pairs $(\alpha,\beta), (\alpha',\beta)\in\dsgp_{\sigma}$,
		\begin{align}
			\sigma(\alpha,\beta)=\sigma(\alpha',\beta)\,\Longrightarrow\,\alpha=\alpha' \ .
		\end{align} 
		Then, a \textit{(right)} \textit{partial cancellation law} on $\sgp$ is a partial map $\nu\colon \sgp\times\sgp\to\sgp$ defined 
		on a (possibly empty) subset $\sgp_{\nu}^{(2)}\subseteq\sgp\times\sgp$ such that 
		\begin{itemize}
			\item
			for any $(\alpha,\beta)\in\dsgp_{\nu}$, then, whenever defined, 
			$\sigma(\nu(\alpha,\beta),\beta)=\alpha$, and $\nu(\alpha,\nu(\alpha,\beta))=\beta$;
			\item
			for any $\alpha,\beta,\gamma\in\sgp$, then
			\begin{align}\label{eq:cancellation-associative}
				\begin{aligned}
					\nu(\sigma(\alpha,\beta),\gamma)=\sigma(\alpha,\nu(\beta,\gamma)) \ ,\\
					\sigma(\nu(\alpha,\beta),\gamma)=\nu(\alpha,\nu(\beta,\gamma))\ ,\\
					\nu(\alpha,\sigma(\beta,\gamma))=\nu(\nu(\alpha,\beta),\gamma) \ ,
				\end{aligned}
			\end{align}
			when both sides are defined.
		\end{itemize}
		We say that $\nu$ is \textit{strong} if, for any $\alpha,\beta,\gamma\in\sgp$, 
		$\sigma(\alpha,\beta)=\gamma$
		if and only if $\alpha=\nu(\gamma,\beta)$.
	\end{definition}
	
	In particular, if $\nu$ is strong, one has $\nu(\sigma(\alpha,\beta),\beta)=\alpha$ for any 
	$\alpha, \beta\in\sgp$.
	Left partial cancellation laws are similarly defined.
	
	
	\begin{remark}
		Note that every cancellative partial semigroup is endowed with a standard cancellation law. Namely, let $(\sgp,\sigma)$ be a right cancellative partial semigroup. Then, for any $\alpha,\beta\in\sgp$, 
		there is at most one element $\gamma\in\sgp$ such that $(\gamma,\beta)\in\sgp_{\sigma}^{(2)}$ and 
		$\sigma(\gamma,\beta)=\alpha$, \ie $\vert \sgp_{\sigma,\,\alpha}^{(2)}\cap\sgp\times\{\beta\}\vert \leq1$. 
		Set
		\begin{align}
			\sgp_{\nu}^{(2)}\coloneqq \{(\alpha,\beta)\in\sgp\times\sgp\;\vert \;\vert \sgp_{\sigma,\,\alpha}^{(2)}\cap\sgp\times\{\beta\}\vert=1\}\ .
		\end{align}
		Then the partial map $\nu\colon \sgp_{\nu}^{(2)}\subseteq\sgp\times\sgp\to\sgp$ defined by 
		$\nu(\alpha,\beta)=\gamma$, where $\gamma$ is the only element in $\sgp_{\sigma,\,\alpha}^{(2)}
		\cap\sgp\times\{\beta\}$, is a right partial cancellation law.
	\end{remark}
	
	\subsection{Commutative partial semigroups}
	Let $(\sgp,\sgpp)$ be a commutative partial semigroup with a maximal cancellation law 
	$\sgpm\colon \sgp\times\sgp\to\sgp$. To alleviate the notation, for any $\alpha,\beta\in\sgp$, we write 
	$\alpha\sgpp\beta$ and $\alpha\sgpm\beta$ in place of $\sgpp(\alpha,\beta)$ and $\sgpm(\alpha,\beta)$.
	
	\begin{definition}\label{def:positive-semigroup}
		An element $0\in\sgp$ is a \textit{partial zero} if, whenever defined, $\alpha\sgpp0=\alpha=0\sgpp\alpha$ 
		for any $\alpha\in \sgp$. We denote by $\mathsf{Z}(\sgp)$ the set of partial zeros in $\sgp$.
		We say that $(\sgp,\sgpp,\sgpm)$ is \textit{positive} if
		the following holds:
		\begin{enumerate}\itemsep0.15cm
			\item \label{def:positive-semigroup-(1)} $\mathsf{Z}(\sgp)=\emptyset$;
			\item \label{def:positive-semigroup-(2)} the elements
			$\alpha\sgpm\alpha$ and $(\alpha\sgpm\beta)\sgpm\alpha$ are never defined;
			\item \label{def:positive-semigroup-(3)} the elements $\alpha\sgpm\beta$ and $\beta\sgpm\alpha$ are never simultaneously defined;
			\item \label{def:positive-semigroup-(4)} up to commutation, $\sgpp$ is \textit{strongly associative}, \ie $(\alpha\sgpp\beta)\sgpp\gamma$ 
			is defined if and only if either $\alpha\sgpp(\beta\sgpp\gamma)$ or $\beta\sgpp(\alpha\sgpp\gamma)$ is defined.
		\end{enumerate}
	\end{definition}
	
	We shall make use of the following elementary lemma.
	\begin{lemma}\label{lem:cancellation-associative}
		Assume that $(\sgp,\sgpp,\sgpm)$ is positive. Then 
		\begin{enumerate} \itemsep0.2cm
			\item \label{item:I1} the elements $(\gamma\sgpm\beta)\sgpm\alpha, \gamma\sgpm(\alpha\sgpp\beta), (\gamma\sgpm\alpha)\sgpm\beta$ are either all of them not defined or at least two of them are simultaneously defined;
			\item \label{item:I2} at most three of the elements $\alpha\sgpm(\gamma\sgpm\beta), \beta\sgpm(\gamma\sgpm\alpha), (\alpha\sgpm\gamma)\sgpp\beta, \alpha\sgpp(\beta\sgpm\gamma),  (\alpha\sgpp\beta)\sgpm\gamma$ are simultaneously defined.
		\end{enumerate}
	\end{lemma}
	
	\begin{proof}
		By strong associativity and thanks to Formula~\eqref{eq:cancellation-associative}, 
		the existence of an element implies
		that of at least another one. Namely, set $x=(\gamma\sgpm\beta)\sgpm\alpha$. 
		Then, $\gamma=(x\sgpp\alpha)\sgpp\beta$ and therefore either $\gamma=x\sgpp(\alpha\sgpp\beta)$ or $\gamma=(x\sgpp\beta)\sgpp\alpha$.
		In the first case, we get $\gamma\sgpm(\alpha\sgpp\beta)$ and in the second one we get $(\gamma\sgpm\alpha)\sgpm\beta$.
		Similarly for the other cases.
		
		Note that, by Definition~\ref{def:positive-semigroup}--\eqref{def:positive-semigroup-(3)}, the elements $\alpha\sgpm(\gamma\sgpm\beta)$ 
		and $\alpha\sgpp(\beta\sgpm\gamma)$ (resp. $\beta\sgpm(\gamma\sgpm\alpha)$ and $(\alpha\sgpm\gamma)\sgpp\beta$) 
		are never simultaneously defined. This proves the assertion.
	\end{proof}
	
	By using the same arguments as in the proof of the above lemma, we get the following.
	\begin{lemma}\label{lem:cancellation-associative-2}
		Under the same hypothesis of Lemma~\ref{lem:cancellation-associative}, the elements in \eqref{item:I1} (resp.\ in \eqref{item:I2}) pairwise coincide (whenever defined).\footnote{For example, set $y=(\gamma\sgpm\alpha)\sgpm\beta$, $z=\gamma\sgpm(\alpha\sgpp\beta)$. Then
			\begin{align}
				y=(((x\sgpp\alpha)\sgpp\beta)\sgpm\alpha)\sgpm\beta
				=((\beta\sgpp((x\sgpp\alpha)\sgpm\alpha))\sgpm\beta
				=(\beta\sgpp x)\sgpm\beta=x
			\end{align}
			where the second equality is \eqref{eq:cancellation-associative}. Similarly for $x=z=y$.}
	\end{lemma}
	
	\begin{example}\label{ex:semigroup-line}
		\hfill
		\begin{enumerate}
			\item
			Every root system is a semigroup with respect to the operations $\oplus,\ominus$ induced by the inclusion
			$\rts_+\subset(\rtl,+,-)$.
			\item
			Let $\KZQR$. In Section~\ref{sss:definition-sl(K)}, we introduced the set $\intsf{\R}{\K}$ together with the partial operations $\sgpp$ and $\sgpm$ (cf.\ Formulas \eqref{eq:sgp-intv-add} and \eqref{eq:sgp-intv-sub}). Then $\csgp(\K)\coloneqq(\intsf{\R}{\K}, \sgpp, \sgpm)$ is a semigroup.
		\end{enumerate}
	\end{example}
	
	
	\bigskip\section{Proof of Proposition~\ref{prop:relationssl}}\label{app:relationssl}
	
	In this section, we complete the proof of Proposition~\ref{prop:relationssl}, providing 
	a new presentation for the Lie algebra of the line. 
	Specifically, we shall prove that the relations \eqref{eq:kac-rel-3-new} and 
	\eqref{eq:serre-new} hold in $\sl(\R)$.
	
	\subsection*{Relation~(\ref{eq:serre-new})}   
	We first observe that, if $J_1, J_2$ are nested, $J_1\sgpp J_2$ is not defined and
	\eqref{eq:serre-new} coincides with \eqref{eq:nest-finite}. If $J_1\intnext J_2$ or
	$J_2\intnext J_1$, then \eqref{eq:serre-new} coincides with \eqref{eq:Joining-e-finite} and
	\eqref{eq:Joining-f-finite}. It remains to prove that, if $J'_1\intcap J'_2$, then
	\begin{align}
		[e_{J'_1},e_{J'_2}]=0=[f_{J'_1}, f_{J'_2}]\ .
	\end{align}
	In this case, there exists an interval $J_2$ such that $J'_1=J_1\sgpp J_2$, with $J_1=J'_1\sgpm J_2$ 
	and $J_1\intnext J_2$, and $J'_2=J_2\sgpp J_3$, with $J_3=J'_2\sgpm J_2$ and $J_2\intnext J_3$.
	It is therefore equivalent to show that
	\begin{align}
		[[e_{J_1}, e_{J_2}],[e_{J_2}, e_{J_3}]]=0=[[f_{J_1}, f_{J_2}],[f_{J_2}, f_{J_3}]]\ .
	\end{align}
	By Jacobi identity,
	\begin{align}
		[[e_{J_1}, e_{J_2}],[e_{J_2}, e_{J_3}]]=
		-[[e_{J_2},e_{J_2\sgpp J_3}], e_{J_1}]
		-[[e_{J_2\sgpp J_3}, e_{J_1}], e_{J_2}]
		=0\ .
	\end{align}
	Now, $[e_{J_2},e_{J_2\sgpp J_3}]=0$ by \eqref{eq:nest-finite} since $J_2\rsub J_2\sgpp J_3$, hence the first 
	quantity on the RHS is zero. Since $J_1\intnext J_2\sgpp J_3$, we have $[e_{J_2\sgpp J_3}, e_{J_1}]=e_{J_1\sgpp J_2\sgpp J_3}$, 
	hence by \eqref{eq:nest-finite} the second quantity on the RHS is zero as well. The computation for $f$ is similar. Therefore, 
	\eqref{eq:serre-new} holds.
	
	\subsection*{Relation~(\ref{eq:kac-rel-3-new})} 
	We shall now prove that relation \eqref{eq:kac-rel-3-new} holds. 
	We proceed case--by--case, according to the relative position of 
	two arbitrary intervals. If $J_1=J_2$, $J_1\intnext J_2$, $J_2\intnext J_1$, 
	or $J_1\perp J_2$, this is clear. We shall analyze the other cases separately. 
	
	\subsection*{Case $J_1\rsub J_2$} In this case, $\abf{\cf{J_1}}{\cf{J_2}}=0$, $\abf{\cf{J_2}}{\cf{J_1}}=1$, $
	\rbf{\cf{J_1}}{\cf{J_2}}=1$,  
	$J_2\sgpm J_1$ is defined, and $J_1\sgpp J_2, J_1\sgpm J_2$ are not. We have to show that
	\begin{align}
		[e_{J_1}, f_{J_2}]=-f_{J_2\sgpm J_1}\ .
	\end{align}
	
	By definition, $J_2=J_1\sgpp J_3$, where $J_3=J_2\sgpm J_1$ and $J_1\intnext J_3$. 
	Therefore, by \eqref{eq:Joining-f-finite}, it is equivalent to show that
	\begin{align}
		[e_{J_1}, [f_{J_1},f_{J_3}]]=f_{J_3}\ .
	\end{align}
	By Jacobi identity, we have
	\begin{align}
		[e_{J_1}, [f_{J_1},f_{J_3}]]&=-[f_{J_1}, [f_{J_3}, e_{J_1}]]-[f_{J_3},[e_{J_1},f_{J_1}]]=
		[h_{J_1}, f_{J_3}]=-\rbf{\cf{J_1}}{\cf{J_3}}f_{J_3}\ ,
	\end{align}
	where the second identity follows from \eqref{eq:kac-rel-2}. It remains to observe that 
	$\rbf{\cf{J_1}}{\cf{J_3}}=-1$. 
	\subsection*{Case $J_2\rsub J_1$} This case is identical to the previous one, but we will prove it
	for completeness. In this case, $\abf{\cf{J_1}}{\cf{J_2}}=1$, $\abf{\cf{J_2}}{\cf{J_1}}=0$, $
	\rbf{\cf{J_1}}{\cf{J_2}}=1$,  
	$J_1\sgpm J_2$ is defined, and $J_1\sgpp J_2, J_2\sgpm J_1$ are not. We have to show that
	\begin{align}
		[e_{J_1}, f_{J_2}]=-e_{J_1\sgpm J_2}\ .
	\end{align}
	
	By definition, $J_1=J_2\sgpp J_3$, where $J_3=J_1\sgpm J_2$ and $J_2\intnext J_3$. 
	Therefore, by \eqref{eq:Joining-e-finite}, it is equivalent to show that
	\begin{align}
		[[e_{J_2},e_{J_3}], f_{J_2}]=-e_{J_3}\ .
	\end{align}
	By Jacobi identity, we have
	\begin{align}
		[[e_{J_2},e_{J_3}], f_{J_2}]&=-[[e_{J_3}, f_{J_2}],e_{J_2}]-[[f_{J_2}, e_{J_2}],e_{J_3}]
		=[h_{J_2},e_{J_3}]=\rbf{\cf{J_2}}{\cf{J_3}}e_{J_3}\ .
	\end{align}
	It remains to observe that $\rbf{\cf{J_2}}{\cf{J_3}}=-1$, since $J_2\intnext J_3$. 
	\subsection*{Case $J_1\lsub J_2$} In this case, $\abf{\cf{J_1}}{\cf{J_2}}=1$, $\abf{\cf{J_2}}{\cf{J_1}}=0$, 
	$\rbf{\cf{J_1}}{\cf{J_2}}=1$,  
	$J_2\sgpm J_1$ is defined, and $J_1\sgpp J_2, J_1\sgpm J_2$ are not. We have to show that
	\begin{align}
		[e_{J_1}, f_{J_2}]=f_{J_2\sgpm J_1}\ .
	\end{align}
	By definition, $J_2=J_1\sgpp J_3$, where $J_3=J_2\sgpm J_1$ and in this case $J_3\intnext J_1$. 
	Therefore, by \eqref{eq:Joining-f-finite}, it is equivalent to show that
	\begin{align}
		[e_{J_1}, [f_{J_3},f_{J_1}]]=-f_{J_3}\ .
	\end{align}
	By Jacobi identity, we have
	\begin{align}
		[e_{J_1}, [f_{J_3},f_{J_1}]]&=-[f_{J_3}, [f_{J_1}, e_{J_1}]]-[f_{J_1},[e_{J_1},f_{J_3}]]=
		-[h_{J_1}, f_{J_3}]=\rbf{\cf{J_1}}{\cf{J_3}}f_{J_3}\ .
	\end{align}
	It remains to observe that $\rbf{\cf{J_1}}{\cf{J_3}}=-1$. 
	The case $J_2\lsub J_1$ is identical.
	\subsection*{Case $J_1<J_2$} In this case, $\abf{\cf{J_1}}{\cf{J_2}}=0$, $\abf{\cf{J_2}}{\cf{J_1}}=0$, 
	$\rbf{\cf{J_1}}{\cf{J_2}}=0$,  and $J_1\sgpp J_2$, $J_2\sgpm J_1$, $J_1\sgpm J_2$ are not defined. 
	We have to show that
	\begin{align}
		[e_{J_1}, f_{J_2}]=0\ .
	\end{align}
	By definition, $J_2=J_3\sgpp J_1\sgpp J_4$ with $J_3\intnext J_1\intnext J_4$, 
	where $J_3, J_4$ are the two connected components of $J_2\setminus J_1$. It is
	therefore equivalent to show
	\begin{align}
		[e_{J_1}, [f_{J_3},f_{J_1\sgpp J_4}]]=0\ .
	\end{align}
	By Jacobi identity,
	\begin{align}
		[e_{J_1}, [f_{J_3},f_{J_1\sgpp J_4}]]=
		-[f_{J_3},[f_{J_1\sgpp J_4}, e_{J_1}]]
		-[f_{J_1\sgpp J_4},[e_{J_1}, f_{J_3}]]
		=0\ ,
	\end{align}
	since $J_1\rsub J_1\sgpp J_4$ and $J_3\intnext J_1$. The case $J_2<J_1$ is identical.
	\subsection*{Case $J_1\intcap J_2$} In this case, $\abf{\cf{J_1}}{\cf{J_2}}=-1$, 
	$\abf{\cf{J_2}}{\cf{J_1}}=1$, $\rbf{\cf{J_1}}{\cf{J_2}}=0$,  
	and $J_1\sgpp J_2$, $J_2\sgpm J_1$, $J_1\sgpm J_2$ are not defined. We have to show that
	\begin{align}
		[e_{J_1}, f_{J_2}]=0\ .
	\end{align}
	In this case, there exists an interval $J'_2$ such that $J_1=J'_1\sgpp J'_2$, with $J'_1=J_1\sgpm J'_2$ 
	and $J'_1\intnext J'_2$, and $J_2=J'_2\sgpp J'_3$, with $J'_3=J_2\sgpm J'_2$ and $J'_2\intnext J'_3$.
	It is therefore equivalent to show that
	\begin{align}
		[e_{J'_1\sgpp J'_2}, [f_{J'_2}, f_{J'_3}]]=0\ .
	\end{align}
	By Jacobi identity,
	\begin{align}
		[e_{J'_1\sgpp J'_2}, [f_{J'_2}, f_{J'_3}]]=
		-[f_{J'_2}, [f_{J'_3}, e_{J'_1\sgpp J'_2}]]
		-[f_{J'_3}, [e_{J'_1\sgpp J'_2}, f_{J'_2}]]=0\ ,
	\end{align}
	since $J'_1\sgpp J'_2\intnext J'_3$ and $J'_2\lsub J'_1\sgpp J'_2$.
	The case $J_2\intcap J_1$ is identical.
	
	\bigskip\section{Proof of Proposition~\ref{pr:sgp-serre-general}}\label{app:sgp-serre-general}
	
	In this section, we complete the proof of Proposition~\ref{pr:sgp-serre-general},
	providing necessary and sufficient conditions for the Serre relations to hold in a 
	semigroup Lie algebra.\\
	
	For any $\oa,\ob\in\sgp$, set
	\begin{align}
		X_{\pm,\, \oa, \,\ob} & \coloneqq [\xpm{\oa}, \xpm{\ob}]-
		\mu_{\pm}(\oa,\ob)\, \xpm{\oa\sgpp\ob}\ .
	\end{align}
	We denote by $S_{\pm,\alpha,\beta}$ the subspace spanned by the elements 
	$X_{\pm, \, \oa, \, \ob}$ with $\oa\in \csgp_{\leqslant\alpha}^{\pm}$, $\ob\in \csgp_{\leqslant\beta}^{\pm}$. We shall prove that, for any $\oc\in\sgp$, 
	$[S_{\pm,\alpha,\beta}, \xmp{\oc}]\subseteq S_{\pm,\alpha,\beta}$.
	
	By the Jacobi identity, the commutation relations \eqref{eq:sgp-lie-3} and \eqref{eq:sgp-lie-4}, the commutator
	of $X_{\pm,\, \oa, \, \ob}$ and $\xmp{\oc}$ is an element in $S_{\pm,\alpha,\beta}$
	if and only if the following identities hold in $L_0$ and $L_-$:
	\begin{align}\label{eq:sgp-serre-h}
		\mu_{\pm}(\oa,\ob) \drc{\oa\sgpp\ob,\oc} \xz{\oc}=
		\drc{\ob,\oc\sgpm\oa}\xi_{\mp}(\oc,\oa) \xz{\ob}
		-\drc{\oa,\oc\sgpm\ob}\xi_{\mp}(\oc,\ob) \xz{\oa}\ ,
	\end{align}
	\begin{multline}\label{eq:sgp-serre-f}
		\mu_{\pm}(\oa,\ob)\xi_{\mp}(\oc,\oa\sgpp\ob) \xmp{\oc\sgpm(\oa\sgpp\ob)}
		=\xi_{\mp}(\oc,\oa)\xi_{\mp}(\oc\sgpm\oa,\ob) 
		\xmp{(\oc\sgpm\oa)\sgpm\ob}\\[3pt]
		-\xi_{\mp}(\oc,\ob)\xi_{\mp}(\oc\sgpm\ob,\oa) \xmp{(\oc\sgpm\ob)\sgpm\oa}\ ,
	\end{multline}
	and the element
	\begin{multline}\label{eq:sgp-serre-e}
		\drc{\oa,\oc}\,\kappa(\oa,\ob) \xpm{\ob}
		+\xi_{\pm}(\ob,\oc)[\xpm{\oa},\xpm{\ob\sgpm\oc}]
		+\xi_{\mp}(\oc,\oa)\xi_{\pm}(\ob,\oc\sgpm\oa) \xpm{\ob\sgpm(\oc\sgpm\oa)}-\drc{\ob,\oc}\,\kappa(\ob,\oa) \xpm{\oa}\\[3pt]
		+\xi_{\pm}(\oa,\oc)[\xpm{\oa\sgpm\oc},\xpm{\ob}]
		-\xi_{\mp}(\oc,\ob)\xi_{\pm}(\oa,\oc\sgpm\ob) \xpm{\oa\sgpm(\oc\sgpm\ob)}
		-\mu_{\pm}(\oa,\ob)\xi_{\pm}(\oa\sgpp\ob,\oc) \xpm{(\oa\sgpp\ob)\sgpm\oc}
	\end{multline}
	belongs to $S_{\pm,\alpha,\beta}$.
	\subsection*{The identity (\ref{eq:sgp-serre-h})}  
	
	The identity is non--trivial only if $\oc=\oa\sgpp\ob$, in which case 
	$\oa=\oc\sgpm\ob$, $\ob=\oc\sgpm\oa$, and it reduces to relation~\eqref{eq:sgp-serre-general-1}, \ie one has
	\begin{align}
		\xi_{\mp}(\oa\sgpp\ob,\oa)\xz{\ob}-\xi_{\mp}(\oa\sgpp\ob,\ob)\xz{\oa}=
		\mu_{\pm}(\oa,\ob)(\xz{\oa}+\xz{\ob})\ .
	\end{align}
	and therefore,
	\begin{align}
		\xi_{\mp}(\oa\sgpp\ob,\oa)=\mu_{\pm}(\oa,\ob)=-\xi_{\mp}(\oa\sgpp\ob,\ob)\ .
	\end{align}
	
	\subsection*{The identity (\ref{eq:sgp-serre-f})}  
	By Lemma~\ref{lem:cancellation-associative-2}, it is enough to observe that 
	the elements $\oc\sgpm(\oa\sgpp\ob)$,
	$(\oc\sgpm\ob)\sgpm\oa$, and $(\oc\sgpm\oa)\sgpm\ob$
	coincide whenever defined. Therefore, \eqref{eq:sgp-serre-f} reduces to 
	\eqref{eq:sgp-serre-general-3}.
	\subsection*{The identity (\ref{eq:sgp-serre-e})}\label{sss:pf-serre-last} 
	We first consider the cases $\oc=\oa,\ob$. 
	
	Assume that $\oa\sgpp\ob$ is not defined.
	\begin{itemize}[leftmargin=1em]\itemsep0.2cm
		\item
		If $(\ob,\oa)\in\dsgp_{\sgpm}$ and $(\oa,\ob)\not\in\dsgp_{\sgpm}$, then, for $\oc=\oa$,
		we get
		\begin{align}
			\kappa(\oa,\ob)\xpm{\ob}+\xi_{\pm}(\ob,\oa)[\xpm{\oa},\xpm{\ob\sgpm\oa}]\ ,
		\end{align}
		which gives
		\begin{align}
			\kappa(\oa,\ob)=-\xi_{\pm}(\ob,\oa)\xi_{\mp}(\ob,\oa)\ .
		\end{align}
		For $\oc=\ob$, we get 
		\begin{align}
			-\kappa(\ob,\oa)\xpm{\oa}+
			\xi_{\mp}(\ob,\oa)\xi_{\pm}(\ob,\ob\sgpm\oa)\xpm{\ob\sgpm(\ob\sgpm\oa)}\ ,
		\end{align}
		which gives
		\begin{align}
			\kappa(\ob,\oa)=\xi_{\pm}(\ob,\oa)\xi_{\mp}(\ob,\ob\sgpm\oa)=
			-\xi_{\pm}(\ob,\oa)\xi_{\mp}(\ob,\oa)\ ,
		\end{align}
		where the second equality follows from \eqref{eq:sgp-serre-general-1}. 
		\item
		The case $(\ob,\oa)\not\in\dsgp_{\sgpm}$ and $(\oa,\ob)\in\dsgp_{\sgpm}$
		is identical to the previous one.
		\item If $(\oa,\ob),(\ob,\oa)\not\in\dsgp_{\sgpm}$, \eqref{eq:sgp-serre-e}
		reduces to the condition $\kappa(\oa,\ob)=0=\kappa(\ob,\oa)$. 
	\end{itemize}
	If $\oa\sgpp\ob$ is defined, it is enough to add to the previous identities the summand 
	$\xi_{\pm}(\oa\sgpp\ob,\oa)\xi_{\mp}(\oa\sgpp\ob,\ob)$.
	This proves \eqref{eq:sgp-serre-general-2}.
	
	We move to the case $\oc\neq\oa,\ob$. By Lemma~\ref{lem:cancellation-associative-2}, the elements $\oa\sgpp(\ob\sgpm\oc)$, $\ob\sgpm(\oc\sgpm\oa)$,
	$(\oa\sgpm\oc)\sgpp\ob$, $\oa\sgpm(\oc\sgpm\ob)$, 
	$(\oa\sgpp\ob)\sgpm\oc$ coincide whenever defined. In particular, it follows that
	the element \eqref{eq:sgp-serre-e} belongs to the subspace $S_{\pm,\alpha,\beta}$ if and only if
	\eqref{eq:sgp-serre-general-4} holds.

	\bigskip\section{Proof of Theorem~\ref{thm:sym-serre-rel}}\label{app:sym-serre-rel}
	
	In this section, we complete the proof of Theorem~\ref{thm:sym-serre-rel}. 
	Given a good Cartan semigroup $\csgp$, for any 
	$(\ia,\ib)\in\serre{\csgp}$, we have to show that the relation
	$[\xpm{\ia}, \xpm{\ib}]=\xi_{\mp}(\alpha\sgpp\beta,\alpha)\cdot \xpm{\ia\sgpp\ib}$
	holds in $\g(\csgp)$. By Proposition~\ref{pr:sgp-serre-general}, it is enough to show 
	that the relations \eqref{eq:sgp-serre-general-1}, 
	\eqref{eq:sgp-serre-general-2}, \eqref{eq:sgp-serre-general-3}, and \eqref{eq:sgp-serre-general-4} hold for any $\oa\in\csgp_{\leqslant\alpha}^{\pm}$, 
	$\ob\in\csgp_{\leqslant\beta}^{\pm}$, $\oc\in\sgp$.
	
	The first two follow, respectively, from the properties \eqref{eq:sym-sgp-datum-3}, \eqref{eq:sym-sgp-datum-4}, and \eqref{eq:sym-sgp-datum-2} of good Cartan semigroups. 
	In order to prove the remaining two relations, we shall proceed by analyzing different cases. 
	Recall that the two relations we are interested in are the following:
	\begin{itemize}
		\item[\eqref{eq:sgp-serre-general-3}] For any $\oa\in \csgp_{\leqslant\alpha}^{\pm}$, $\ob\in \csgp_{\leqslant\beta}^{\pm}$, $\oc\in\sgp$,
		\begin{align}
			\drc{\oc\sgpm(\oa\sgpp\ob)}
			\xi_{\pm}(\oa\sgpp\ob,\oa)\xi_{\pm}&(\oc,\oa\sgpp\ob)=\\[3pt]
			=&\drc{(\oc\sgpm\oa)\sgpm\ob}\xi_{\pm}(\oc,\oa)\xi_{\pm}(\oc\sgpm\oa,\ob)
			-\drc{(\oc\sgpm\ob)\sgpm\oa}\xi_{\pm}(\oc,\ob)\xi_{\pm}(\oc\sgpm\ob,\oa) \ .
		\end{align}
		\item[\eqref{eq:sgp-serre-general-4}] For any $\oa\in \csgp_{\leqslant\alpha}^{\pm}$, $\ob\in \csgp_{\leqslant\beta}^{\pm}$, 
		$\oc\in\sgp$, $\oc\neq\oa,\ob$,
		\begin{align}
			\xi_{\pm}(\oa,\oc)\xi_{\mp}((\oa\sgpm\oc)\sgpp\ob,\ob)
			-\xi_{\pm}(\ob,\oc)\xi_{\mp}(\oa\sgpp(\ob\sgpm\oc),\oa)&=\\[3pt]
			\drc{\ob\sgpm(\oc\sgpm\oa)}\xi_{\mp}(\oc,\oa)\xi_{\pm}(\ob,\oc\sgpm\oa)
			-\drc{\oa\sgpm(\oc\sgpm\ob)}&\xi_{\mp}(\oc,\ob)\xi_{\pm}(\oa,\oc\sgpm\ob)\\[3pt]
			\phantom{\drc{\ob\sgpm(\oc\sgpm\oa)}\xi_{\mp}(\oc,\oa)\xi_{\pm}(\ob\sgpm\oc,\oa)\qquad}
			-&\drc{(\oa\sgpp\ob)\sgpm\oc}\xi_{\mp}(\oa\sgpp\ob,\oa)\xi_{\pm}(\oa\sgpp\ob,\oc)
			\ .
		\end{align}
	\end{itemize}
	Note that, as we show in Corollary~\ref{cor:perp-comm}, 
	the relations \eqref{eq:sgp-serre-general-3} and \eqref{eq:sgp-serre-general-4} are 
	satisfied whenever $\oa\perp\ob$. Therefore, we can assume
	$\oa\not\perp\ob$.
	
	\subsection*{Proof of relation (\ref{eq:sgp-serre-general-3})} 
	
	First, note that, since $\oa\not\perp\ob$, the relation \eqref{eq:sgp-serre-general-3} is trivial whenever $\oa\sgpp\ob$ is not defined,
	since $\oc\sgpm(\oa\sgpp\ob)$, $(\oc\sgpm\oa)\sgpm\ob$, and $(\oc\sgpm\ob)\sgpm\oa$ cannot exist in this case. Indeed,  the first element
	does not exist by definition. By Lemma~\ref{lem:cancellation-associative}, the second and third element either do not exist
	or they both exist. However, the latter situation cannot occur because of condition (L1). Therefore we can assume that $\oa\sgpp\ob$ is defined
	and thus real by Remark~\ref{rem:serre-induction}. 
	
	Next, \eqref{eq:sgp-serre-general-3} is trivial whenever $\oc\sgpm(\oa\sgpp\ob)$ is not defined. Indeed, in this case
	it follows by strong associativity (cf.\ Definition~\ref{def:positive-semigroup}--\eqref{def:positive-semigroup-(4)}) that the elements $(\oc\sgpm\oa)\sgpm\ob$ 
	and $(\oc\sgpm\ob)\sgpm\oa$ are also not defined. Therefore  we can assume that $\oc\sgpm(\oa\sgpp\ob)$ is defined. 
	
	Thus, we are left considering only the case in which both $\oa\sgpp\ob$ and $\oc\sgpm(\oa\sgpp\ob)$ are defined. 
	
	\noindent
	\textit{Case $\oa,\ob\in\rsgp$.}
	We claim that, if $\oa,\ob\in\rsgp$, we can assume $\oc\in\rsgp$. Indeed, in this case $\oa\sgpp\ob$ is real and, by the reality condition, if $\oc\not\in\rsgp$, there exists a real element $\oc'$ such that $\oc\sgpm\oc'$ is defined and orthogonal to 
	$\oa\sgpp\ob$. It follows that $\oc\sgpm\oc'$ is orthogonal to $\oa$ and $\ob$ and thus, by condition (L1), $\oc'\sgpm\oa$ (resp. $\oc'\sgpm\ob$) is defined whenever
	$\oc\sgpm\oa$ (resp. $\oc\sgpm\ob$) is. Therefore, by orthogonality, we have 
	\begin{align}
		\xi_{\pm}(\oc,\oa\sgpp\ob) = & \xi_{\pm}(\oc',\oa\sgpp\ob)\ , \quad \xi_{\pm}(\oc,\oa)=\xi_{\pm}(\oc',\oa)\ , \quad \xi_{\pm}(\oc,\ob)=\xi_{\pm}(\oc',\ob)\\[2pt]
		\xi_{\pm}(\oc\sgpm\oa,\ob) = & \xi_{\pm}(\oc'\sgpm\oa,\ob)\ , \xi_{\pm}(\oc\sgpm\ob,\oa)=\xi_{\pm}(\oc'\sgpm\ob,\oa) \ .
	\end{align}
	It is also clear that 
	\begin{align}
		\drc{(\oc\sgpm\oa)\sgpm\ob}=\drc{(\oc'\sgpm\oa)\sgpm\ob}\quad \text{and}\quad\drc{(\oc\sgpm\ob)\sgpm\oa}=\drc{(\oc'\sgpm\ob)\sgpm\oa}
	\end{align}
	since both $\oa$
	and $\ob$ are orthogonal to $\oc\sgpm\oc'$\footnote{Therefore, for example, $(\oc\sgpm\oa)\sgpm\ob=((\oc\sgpm\oc')\sgpp(\oc'\sgpm\oa))\sgpm\ob=
		(\oc\sgpm\oc')\sgpp((\oc'\sgpm\oa)\sgpm\ob)$.}. This proves that the relation \eqref{eq:sgp-serre-general-3} holds for the triple $(\oa,\ob,\oc)$ if and only if
	it holds for the triple $(\oa,\ob,\oc')$. Therefore we can assume $\oc\in\rsgp$.
	
	\noindent
	\textit{Case $\oa\in\rsgp$, $\ob\in\isgp$.} Note that, in this case, $\oc$ is necessarily an element in $\isgp$.

	Assume therefore that either $(\oa,\ob,\oc)\in\rsgp\times\rsgp\times\rsgp$ or $(\oa,\ob,\oc)\in\rsgp\times\isgp\times\isgp$, with $\oa\not\perp\ob$ and $\oc\sgpm(\oa\sgpp\ob)$ defined. 
	We first observe that in any good  semigroup the conditions \eqref{eq:sgp-serre-general-3} can be simplified. 
	Namely, it follows from the admissibility condition \eqref{item:A1} from Definition~\ref{def:admissible-triple} and Lemma~\ref{lem:cancellation-associative} that either none or exactly two elements 
	among $\oc\sgpm(\oa\sgpp\ob)$, $(\oc\sgpm\oa)\sgpm\ob$, and $(\oc\sgpm\ob)\sgpm\oa$ can be 
	simultaneously defined. Recall that, by definition, $\xi_+(x,y)=\xi_-(y,x)$ and $\xi\coloneqq\xi_+$. Then, one checks easily that 
	\eqref{eq:sgp-serre-general-3} reduces to the condition 
	\begin{align}\label{eq:aux-serre-1}
		\xi(x\sgpp y,x)\xi((x\sgpp y)\sgpp z, x \sgpp y )=
		\xi((x\sgpp y)\sgpp z, x)\xi(y\sgpp z, y)\ ,
	\end{align}
	where either $(x,y,z)=(\oa,\ob, (c\sgpm\oa)\sgpm\ob))$ or $(x,y,z)=(\ob,\oa, (\oc\sgpm\ob)\sgpm\oa)$. 
	It is easy to check that this holds. Indeed,
	\begin{align}
		\xi(x\sgpp y,x)\xi((x\sgpp y)\sgpp z, x\sgpp y)
		&=\phantom{-}\xi(x\sgpp y,x)\xi(x\sgpp(y\sgpp z), x\sgpp y)\\
		&=\phantom{-}\xi(x\sgpp y,x)\xi(y\sgpp z,y)\\
		&=-\xi(x\sgpp y,y)\xi(y\sgpp z,y)\\
		&=-\xi((x\sgpp y)\sgpp z, y\sgpp z)\xi( y\sgpp z, y)\\
		&=-\xi(x\sgpp(y\sgpp z),y\sgpp z)\xi(y\sgpp z,y)\\
		&=\phantom{-}\xi(x\sgpp(y\sgpp z),x)\xi(y\sgpp z, y) \ ,
	\end{align}
	where the second and fourth identities rely on \eqref{eq:sym-sgp-datum-5}, while the third and sixth ones 
	rely on \eqref{eq:sym-sgp-datum-3} and \eqref{eq:sym-sgp-datum-4}, and the first and fifth ones follow by associativity.
	Note that we are allowed to use \eqref{eq:sym-sgp-datum-5} since the elements $\oc$, $\oa\sgpp\ob$, and $\oc\sgpm\oa$ (resp.\ $\oc\sgpm\ob$)
	can never be locally degenerate.
	
	\subsection*{Proof of relation (\ref{eq:sgp-serre-general-4})} 
	
	First, we claim that, if $\oa,\ob\in\rsgp$, we can assume $\oc\in\rsgp$, since \eqref{eq:sgp-serre-general-4}
	is trivial otherwise. Indeed, if $\oc\not\in\rsgp$, the elements $\oa\sgpm\oc$, $\ob\sgpm\oc$,
	and $(\oa\sgpp\ob)\sgpm\oc$ are certainly not defined, since $\sgp_{\leqslant a}$ 
	and $\sgp_{\leqslant b}$ are contained in $\rsgp$. Now assume that $\oa\sgpm(\oc\sgpm\ob)$ is defined and $\xi_{\pm}(\oa,\oc\sgpm\ob)\neq0$. 
	Then, $\oc\sgpm\ob$ is defined, is real (since $\oa$ is), and belongs to $\csgp_{\leqslant a}^{\pm}$. 
	It follows that $\oc=(\oc\sgpm\ob)\sgpp\ob$ is necessarily real, because the pair
	$(\oc\sgpm\ob, \ob)$ belongs to $\serreadm{\csgp}$ by Remark~\ref{rem:serre-induction}. 
	Therefore, either $\oa\sgpm(\oc\sgpm\ob)$ is not defined or $\xi_{\pm}(\oa,\oc\sgpm\ob)=0$, and
	similarly for $\ob\sgpm(\oc\sgpm\oa)$. It follows that \eqref{eq:sgp-serre-general-4} is trivial if $\oc\not\in\rsgp$.
	
	Assume therefore that either $(\oa,\ob,\oc)\in\rsgp\times\rsgp\times\rsgp$ or $(\oa,\ob,\oc)\in\rsgp\times\isgp\times\sgp$, with
	$\oa\not\perp\ob$. By Lemma~\ref{lem:cancellation-associative} and the admissibility condition (\ref{item:A2}) in Definition~\ref{def:admissible-pair},
	either none or exactly two elements among $(\oa\sgpp\ob)\sgpm\oc$, $(\oa\sgpm\oc)\sgpp\ob$, $\oa\sgpp(\ob\sgpm\oc)$,
	$\oa\sgpm(\oc\sgpm\ob)$, and $\ob\sgpm(\oc\sgpm\oa)$
	are simultaneously defined. Then, one checks easily that \eqref{eq:sgp-serre-general-4}, similarly to the case of 
	\eqref{eq:sgp-serre-general-3}, reduces to the conditions
	\begin{align}
		\label{eq:aux-serre-2}
		\xi(x\sgpp z,(x\sgpp z)\sgpp y)
		\xi((x\sgpp z)\sgpp y, y\sgpp z)
		&=\xi(y,y\sgpp z)\xi(x\sgpp z,x) \ ,\\
		\label{eq:aux-serre-3}
		\xi(x\sgpp(y\sgpp z),y\sgpp z)\xi(y,x\sgpp y)&=
		\xi(z,y\sgpp z)\xi(x\sgpp(y\sgpp z),y) \ ,\\
		\label{eq:aux-serre-4}
		\xi(x,(x\sgpp y)\sgpp z)\xi(y,y\sgpp z)&=
		\xi(y,(x\sgpp y)\sgpp z)\xi(x,x\sgpp z)\ ,
	\end{align}
	whenever all terms are defined. Therefore, we are left to prove the identities \eqref{eq:aux-serre-2},
	\eqref{eq:aux-serre-3}, and \eqref{eq:aux-serre-4}. One checks by direct inspection that, as before, these 
	follow directly from properties \eqref{eq:sym-sgp-datum-3}, \eqref{eq:sym-sgp-datum-4}, and 
	\eqref{eq:sym-sgp-datum-5}.
	
	\bigskip\section{Proof of Proposition~\ref{prop:topologicalquiver}}\label{app:topologicalquiver}
	
	In this section,  we complete the proof of Proposition~\ref{prop:topologicalquiver}, 
	showing the continuum quiver $\cq{X}$ is a good Cartan semigroup. We shall show that 
	$\cq{X}$ satisfies the conditions $(1)$--$(5)$ from Definition~\ref{def:sym-sgp}. 
	We proved that \eqref{def:sym-sgp-1}, \eqref{def:sym-sgp-2}, and \eqref{def:sym-sgp-3} hold.  It remains to prove that the functions $\xi_{X}$ 
	and $\kappa_X$ satisfy \eqref{def:sym-sgp-4} and \eqref{def:sym-sgp-5}. Note that $\kappa_X$ is symmetric and satisfies the 
	condition \eqref{eq:sym-sgp-datum-1} by definition. Below, we prove the conditions~\eqref{eq:sym-sgp-datum-3}, \eqref{eq:sym-sgp-datum-4}, \eqref{eq:sym-sgp-datum-5}, and \eqref{eq:sym-sgp-datum-2}.
	
	\subsection*{Proof of conditions~(\ref{eq:sym-sgp-datum-3}) and (\ref{eq:sym-sgp-datum-4})}\label{ss:proof-3-4}  
	
	\subsubsection*{Case 1: $\ia,\ib$ are elementary intervals} Note that, whenever $\ia\sgpp \ib$ is defined, 
	$\abfcf{\ia\sgpp \ib}{\ia}+\abfcf{\ia\sgpp \ib}{\ib}=1$
	and $\rbfcf{\ia\sgpp \ib}{\ia}=1=\rbfcf{\ia\sgpp \ib}{\ib}$.
	Therefore, $\xi_{X}$ satisfies \eqref{eq:sym-sgp-datum-3} and \eqref{eq:sym-sgp-datum-4}, \ie
	\begin{align}
		\xi_{X}(\ia\sgpp \ib,\ia)&=(-1)^{\abfcf{\ia\sgpp \ib}{\ia}}\rbfcf{\ia\sgpp \ib}{\ia}=
		-(-1)^{\abfcf{\ia\sgpp \ib}{\ib}}\rbfcf{\ia\sgpp \ib}{\ib}=-\xi_{X}(\ia\sgpp \ib,\ib)\ ,\\
		\xi_{X}(\ia,\ia\sgpp \ib)&=(-1)^{\abfcf{\ia}{\ia\sgpp \ib}}\rbfcf{\ia}{\ia\sgpp \ib}=
		-(-1)^{\abfcf{\ib}{\ia\sgpp \ib}}\rbfcf{\ib}{\ia\sgpp \ib}=-\xi_{X}(\ib,\ia\sgpp \ib)\ .
	\end{align}
	Note that, if either $\ia\sgpp \ib$, $\ia\sgpm \ib$ or $\ib\sgpm \ia$ is defined, 
	$\xi_{X}(\ia,\ib)=-\xi_{X}(\ib,\ia)$. 
	
	\subsubsection*{Case 2: $\ia,\ib$ are contractible intervals} If $\ia\sgpp \ib$ is contractible or undefined, the proof is essentially identical to the case of elementary intervals. If $\ia\sgpp \ib$ is non--contractible, then the conditions holds by a direct computation. 
	
	\subsubsection*{Case 3: $\ia$ is a contractible interval, $\ib$ is homeomorphic to $S^1$} By a direct computation, one sees that, if $\ia\sgpp \ib$ is defined, then both sides of  \eqref{eq:sym-sgp-datum-3} (resp.\ \eqref{eq:sym-sgp-datum-4})
	equals $-1$ (resp. $1$). 
	
	\subsubsection*{Case 4: $\ia$ is a contractible interval, $\ib\neq S^1$ is a non--contractible interval} First, by Lemma~\ref{lem:int-X}, $\ib$ is of the form $S^1\sgpp \bigoplus_{k} T_k$ for some pairwise disjoint contractible intervals $T_k$. Now, $\ia\sgpp \ib$ exists if and only if $S^1\to \ia$ or $T_h\to \ia$ for some $h$. In the first case, $\ia$ is perpendicular to all $T_k$, so the check of conditions~\eqref{eq:sym-sgp-datum-3} and \eqref{eq:sym-sgp-datum-4} reduces to the third case above. On the other hand, in the second case, $\ia$ is perpendicular to $\ib\sgpm T_h$, so the check of conditions~\eqref{eq:sym-sgp-datum-3} and \eqref{eq:sym-sgp-datum-4} reduces to the second case above.
	
	\subsubsection*{Case 5: $\ia,\ib$ are non--contractible intervals} Since $\ia\sgpp \ib$ is never defined, \eqref{eq:sym-sgp-datum-3} and  \eqref{eq:sym-sgp-datum-4} are automatically satisfied. 
	
	\subsection*{Proof of the condition~(\ref{eq:sym-sgp-datum-2})} 
	
	\subsubsection*{Case 1: $\ia,\ib$ are elementary intervals} If $\ia\sgpp \ib$  is defined, then 
	\begin{align}
		\xi_{X}(\ia\sgpp \ib, \ia)\xi_{X}(\ia,\ia\sgpp \ib)=-\xi_{X}(\ia\sgpp \ib, \ia)^2=-1=\kappa_{X}(\ia, \ib) \ .
	\end{align}
	If either $\ia\sgpm \ib$ or $\ib\sgpm \ia$ is defined, then
	\begin{align}
		-\xi_{X}(\ia, \ib)\xi_{X}(\ib,\ia)=\xi_{X}(\ia, \ib)^2=1=\kappa_{X}(\ia, \ib) \ .
	\end{align}
	Finally, if $\ia\sgpp \ib$, $\ia\sgpm \ib$ and $\ib\sgpm \ia$ are not defined, then
	$\kappa_{X}(\ia,\ib)=0$. 
	
	\subsubsection*{Case 2: $\ia,\ib$ are contractible intervals} Since we need to verify \eqref{eq:sym-sgp-datum-2} only for those $\ia, \ib$ such that $\ia\sgpp \ib$ is real (cf.\ condition $(\ast)$ in \eqref{eq:sym-sgp-datum-2}) and if $\ia\sgpm \ib$ or $\ib\sgpm \ia$ are defined, they are real, this case reduces to the previous one.
	
	\subsubsection*{Case 3: $\ia$ is a contractible interval, $\ib$ is homeomorphic to $S^1$} If $\ia\sgpp \ib$ is defined, in particular, it is not locally degenerate. Thus, by using similar arguments as in the third case of Section~\ref{ss:proof-3-4}, both sides of \eqref{eq:sym-sgp-datum-2} equals $-1$. Next, $\ia\sgpm \ib$ is never defined, while if $\ib\sgpm \ia$ is defined, this implies $\ia\subset \ib$ and therefore \eqref{eq:sym-sgp-datum-2} is trivial. 
	Finally, if neither $\ia\sgpp \ib$ or $\ib\sgpm \ia$ are defined, then $\ia\cap \ib=\emptyset$ and the result follows.
	
	\subsubsection*{Case 4: $\ia$ is a contractible interval, $\ib\neq S^1$ is a non--contractible interval} It follows by the same arguments as in the previous section.
	
	\subsection*{Proof of the condition~(\ref{eq:sym-sgp-datum-5})} 
	
	\subsubsection*{Case 1: $\ia,\ib, \ic$ are elementary intervals} First, $\ia\sgpm \ib$ is defined if and only if either $\ib\lsub \ia$ or $\ib\rsub \ia$. Note also that, if $\ia\sgpp \ic$ and $\ib\sgpp \ic$ are both defined, then 
	one of the following holds (we use the notation from Section~\ref{ss:semigroup-line}):
	\begin{itemize}
		\item $\ib\rsub \ia$, $\ic\to \ia$, $\ic\to \ib$;
		\item $\ib\lsub \ia$ $\ia\to \ic$, $\ib\to \ic$.
	\end{itemize}
	In both cases, we have $\ia\sgpp \ic= (\ia\sgpm \ib)\sgpp (\ib\sgpp \ic)$, hence
	\begin{align}
		\abfcf{\ia\sgpp \ic}{\ib\sgpp \ic}= &\abf{\cf{\ib\sgpp \ic}+\cf{(\ia\sgpm \ib)}}{\cf{\ib\sgpp \ic}}=\abfcf{\ib\sgpp \ic}{\ib\sgpp \ic} + \abf{\cf{(\ia\sgpm \ib)}}{\cf{\ib}+\cf{\ic}}\\
		=& \abfcf{\ic}{\ic} + \abf{\cf{(\ia\sgpm \ib)}}{\cf{\ib}}=\abfcf{\ia}{\ib}\ .
	\end{align}
	Here, we have applied Formula~\eqref{eq:abf-line} and Remark~\ref{rem:euler-form-identities-I} (since $\ic\perp \ia\sgpm \ib$). One can show similarly that $\abfcf{\ib\sgpp \ic}{\ia\sgpp \ic}=\abfcf{\ib}{\ia}$. Thus, condition \eqref{eq:sym-sgp-datum-5} holds.
	
	\subsubsection*{Case 2: $\ia,\ib, \ic$ are contractible intervals} Assume that there exists $\ia\sgpm \ib$. If $\ia\sgpp \ic$ and $\ib\sgpp \ic$ are contractible, we can reduce to the case of elementary intervals. It remains to be checked when
	\begin{enumerate}\itemsep0.2cm
		\item $\ib\sgpp \ic$ is non--contractible and it is not homeomorphic to $S^1$;
		\item $\ia\sgpp \ic$ is non--contractible and it is not homeomorphic to $S^1$, and $\ib\sgpp \ic$ is contractible but not contained in the locally degenerate part of $\ia\sgpp \ic$.
	\end{enumerate}
	In the first case, we can decompose $\ic$ as $\ia_1''\sgpp \ia_2''$ such that both $\ia\sgpp \ia_1''$ and $\ib\sgpp \ia_1''$ exist and are contractible and $\ia_2''$ is perpendicular to both $\ia$ and $\ib$. Thus, we can reduce the the situation described in the above paragraph. On the other hand, in the second case one needs to do a direct computation to show that \eqref{eq:sym-sgp-datum-5} holds.
	
	\subsubsection*{Case 3: $\ia, \ib$ are contractible elements, $\ic$ is homeomorphic to $S^1$} Assume that $\ia\sgpm \ib$, $\ia\sgpp \ic$ and $\ib\sgpp \ic$ are defined. In this case, one can explicitly verify that
	\begin{align}
		\xi_{X}(\ia\sgpp \ic, \ib\sgpp \ic)=(-1)^{-1}(-1)=1=\xi_{X}(\ia, \ib)\ .
	\end{align}
	
	\subsubsection*{Case 4: $\ia, \ib$ are contractible elements, $\ic\neq S^1$ is non--contractible} First, by Lemma~\ref{lem:int-X}, $\ic$ is of the form $S^1\sgpp \bigoplus_{k} T_k$ for some pairwise disjoint contractible intervals $T_k$. Now, $\ia\sgpm \ib$, $\ia\sgpp \ic$ and $\ib\sgpp \ic$ are defined. Thus, we have to consider only the following two situations: $S^1\to \ia, S^1\to \ib$ or $T_h\to \ia, T_h\to \ib$ for some $h$. In the first situation, we reduce to the third case above, while in the second one to the second case above.
	
	\subsubsection*{Case 5: $\ia$ is a contractible interval, $\ib$ is homeomorphic to $S^1$, $\ic$ is an interval} $\ia\sgpm \ib$ is never defined, while if $\ib\sgpm \ia$ is defined, this implies $\ia\subset \ib$. Thus, $\xi_X(\ib, \ia)=0$. Let $\ic$ be an interval such that both $\ia\sgpp \ic$ and $\ib\oplus \ic$ exist and are not homeomorphic to $S^1$. First note that $\ic$ can be only contractible since sums of non--contractible intervals never exist. In addition, $\ib\oplus \ic=(\ib\sgpm \ia)\sgpp (\ib\sgpp \ic)$ and $\ic\perp \ib\sgpm \ia$. Thus,
	\begin{align}
		\abfcf{\ib\sgpp \ic}{\ia\sgpp \ic}= &\abf{\cf{\ia\sgpp \ic}+\cf{\ib\sgpm \ia}}{\cf{\ia\sgpp \ic}}=\abfcf{\ib\sgpp \ic}{\ib\sgpp \ic} + \abf{\cf{\ib\sgpm \ia}}{\cf{\ia}+\cf{\ic}}\\[2pt]
		=& \abfcf{\ib}{\ib} + \abf{\cf{\ib\sgpm \ia}}{\cf{\ia}}=\abfcf{\ib}{\ia}=0\ .
	\end{align}
	Similarly, $\abfcf{\ia\sgpp \ic}{\ib\sgpp \ic}=0$. 
	
	\subsubsection*{Case 6: $\ia$ is a contractible interval, $\ib\neq S^1$ is a non--contractible interval, $\ic$ is an interval} $\ia\sgpm \ib$ is never defined, while $\ib\sgpm \ia$ is defined if and only if $\ia\subset \ib$. By Lemma~\ref{lem:int-X}, $\ib$ is of the form $S^1\sgpp \bigoplus_{k} T_k$ for some pairwise disjoint contractible intervals $T_k$.	Let $\ic$ be an interval such that $\ia\sgpp \ic$ and $\ib\sgpp \ic$ are defined. We have that $\ic$ is real, since sums of imaginary elements are never defined. We have two mutually exclusive cases:
	\begin{itemize}\itemsep0.2cm
		\item $\ic$ is perpendicular to $T_k$ for all $k$, so to verify that \eqref{eq:sym-sgp-datum-5} holds, one can suitably reduce to the fifth case above;
		\item there exists $T_h$ such that $T_h\to \ic$, so to verify that \eqref{eq:sym-sgp-datum-5} holds, one can suitably reduce to the second case above.
	\end{itemize}
	
	\subsubsection*{Case 7:  $\ia,\ib$ are non--contractible intervals, $\ic$ is an interval} Assume that $\ib\sgpm \ia$ is defined, hence it is necessarily real. As before, we decompose $\ib$ as $\ib=S\sgpp T$, with $S\in \isgp$ and $T\in \rsgp$, such that $\ia\sgpm \ib$ is perpendicular to $S$. Let $\ic$ be an interval such that $\ia\sgpp \ic$ and $\ib\sgpp \ic$ are defined. As seen before, $\ic$ is necessarily a real element. If $\ic$ sums the imaginary part $S$, then set $S'\coloneqq S\oplus \ic$. Therefore,
	\begin{align}
		\abfcf{\ia\sgpp \ic}{\ib\sgpp \ic}&=\abf{\cf{\ia\sgpm \ib}+\cf{T}+\cf{S'}}{\cf{T}+\cf{S'}}=\abfcf{\ia\sgpm \ib}{T}+\abfcf{T\sgpp S'}{T\sgpp S'}\\
		&= \abfcf{\ia\sgpm \ib}{T}=\abfcf{\ia}{\ib}\ .
	\end{align}
	Similarly, one can prove $\abfcf{\ib\sgpp \ic}{\ia\sgpp \ic}=\abfcf{\ib}{\ia}$. If $\ic$ sums the real part $T$, denote by $T'$ the sum between them. We have $T'\to \ia\sgpm \ib$. In this case, we get:
	\begin{align}
		\abfcf{\ia\sgpp \ic}{\ib\sgpp \ic}&=\abf{\cf{\ia\sgpm \ib}+\cf{T'}+\cf{S}}{\cf{T'}+\cf{S}}=\abfcf{\ia\sgpm \ib}{T'}+\abfcf{T'\sgpp S}{T'\sgpp S}\\
		&= \abfcf{\ia\sgpm \ib}{T'}=0=\abfcf{\ia}{\ib}\ .
	\end{align}
	Similarly, one can show $\abfcf{\ib\sgpp \ic}{\ia\sgpp \ic}=-1=\abfcf{\ib}{\ia}$.

	\bigskip
	
	\addtocontents{toc}{\protect\setcounter{tocdepth}{1}}

\end{document}